\documentclass[12pt,twoside]{article}
\usepackage{amsmath,mathrsfs,amssymb,amsthm,amscd}
\usepackage{times}
\usepackage{xy}

\usepackage{enumerate}
\def\titlerunning#1{\gdef\titrun{#1}}
\makeatletter
\def\author#1{\gdef\autrun{\def\and{\unskip, }#1}\gdef\@author{#1}}
\def\address#1{{\def\and{\\\hspace*{18pt}}\renewcommand{\thefootnote}{}%
\footnote {#1}}%
\markboth{\autrun}{\titrun}}
\makeatother
\def\email#1{e-mail: #1}

\usepackage[hyperindex]{hyperref}
\hypersetup{breaklinks=true}
\usepackage{graphicx,color}
\usepackage{ifpdf}
\xyoption{all}

\numberwithin{equation}{section}
\nopagebreak%
\theoremstyle{plain}
\newtheorem{theorem}[equation]{Theorem}
\newtheorem{corollary}[equation]{Corollary}
\newtheorem{proposition}[equation]{Proposition}
\newtheorem{lemma}[equation]{Lemma}

\newtheorem{theorem-definition}[equation]{Theorem Definition}

\theoremstyle{definition}
\newtheorem{definition}[equation]{Definition}
\newtheorem{notation}[equation]{Notation}
\newtheorem{example}[equation]{Example}

\newtheorem{remark}[equation]{Remark}

\DeclareMathOperator{\Tor}{Tor}
\DeclareMathOperator{\tr}{tr} 

\DeclareMathOperator{\ocone}{\overline{cone}}

\DeclareMathOperator{\Td}{Td}
\DeclareMathOperator{\cht}{\widetilde{ch}}

\DeclareMathOperator{\Sym}{Sym}
\DeclareMathOperator{\pd}{\partial}
\DeclareMathOperator{\cpd}{\overline{\partial}}

 \DeclareMathOperator{\pr}{pr}
 \DeclareMathOperator{\dd}{d}

\DeclareMathOperator{\ch}{ch} 
\DeclareMathOperator{\rk}{rk} 
 
\DeclareMathOperator{\Ob}{Ob} 
 
 \DeclareMathOperator{\Ker}{Ker}
\DeclareMathOperator{\Hom}{Hom}

\DeclareMathOperator{\Id}{id}

\DeclareMathOperator{\Spec}{Spec}

\DeclareMathOperator{\cone}{cone}

\DeclareMathOperator{\KA}{\mathbf{KA}}

\DeclareMathOperator{\Db}{\mathbf{D}^{b}}

\DeclareMathOperator{\hDb}{\mathbf{\widehat{D}}^{b}}

\DeclareMathOperator{\oDb}{\overline{\mathbf{D}}^{b}}
\DeclareMathOperator{\oSm}{\overline{\mathbf{Sm}}}
\DeclareMathOperator{\Sm}{\mathbf{Sm}}

\DeclareMathOperator{\im}{Im\,}
\renewcommand{\Im}{{\im}}

\newcommand{\bomega}{\boldsymbol{\omega}}
\newcommand{\ov}{\overline}
\newcommand{\dra}{\dashrightarrow}
\newcommand{\Ld}{}
\newcommand{\Rd}{}

\newcommand{\Lotimes}{\otimes^{\Ld}}
\newcommand{\FS}{\text{\rm FS}}

\newcommand{\can}{\text{{\rm can}}}

\newcommand{\ttwist}{\flat}

\newcommand{\ZZ}{{\mathbb Z}}
\newcommand{\RR}{{\mathbb R}}

\newcommand{\CC}{{\mathbb C}}

\newcommand{\OO}{{\mathcal O}}

\newcommand{\PP}{{\mathbb P}}
\newcommand{\DD}{{\mathbb D}}

\def\?{\ ???\ \immediate\write16{}%
\immediate\write16{Warning: There was still a question mark . . . }%
\immediate\write16{}}

\begin{document}
\titlerunning{holomorphic torsion}
\title{Generalized holomorphic analytic torsion}
\author{Jos\'e Ignacio Burgos Gil \and Gerard Freixas i Montplet \and
  R\u azvan Li\c tcanu}
\date{}
\maketitle
\address{J.~I. Burgos Gil: Instituto de Ciencias Matem\'aticas (CSIC-UAM-UCM-UC3), Calle Nicol\'as Cabrera 15
28049 Madrid, Spain;
\email{burgos@icmat.es} \and G. Freixas i Montplet:
Institut de Math\'ematiques de Jussieu (IMJ), Centre National
  de la Recherche Scientifique (CNRS), France;
\email{freixas@math.jussieu.fr} \and R. Li\c tcanu:
Faculty of Mathematics, University Al. I. Cuza Iasi, Romania;
\email{litcanu@uaic.ro}}

\begin{abstract}
\noindent
In this paper we extend the holomorphic analytic torsion classes of
Bismut and K\"ohler to arbitrary projective morphisms between
smooth algebraic complex varieties. To this end,
we propose an axiomatic definition and give a
classification of the
theories of generalized holomorphic analytic torsion classes for
projective morphisms.
The extension of the holomorphic analytic torsion classes of
Bismut and K\"ohler is obtained as the theory of generalized analytic
torsion classes associated to $-R/2$, $R$ being the $R$-genus.
As application of the axiomatic characterization, we give new simpler
proofs of known properties of holomorpic analytic torsion classes, we
give a characterization of the $R$ genus, and we
construct a direct image of hermitian structures for projective morphisms.
\end{abstract}

\maketitle

\section{Introduction}
\label{sec:introduction}

The aim of this paper is to extend the classes of analytic torsion
forms introduced by Bismut and K\"ohler to arbitrary projective
morphisms between complex algebraic varieties. The main tool for this
extension is an axiomatic characterization of all the possible
theories of holomorphic analytic torsion classes. Before stating
what we mean by a theory of holomorphic analytic torsion
classes, we briefly recall the origin of the analytic torsion.

The R-torsion is a topological invariant attached to certain
euclidean flat vector bundles on a finite CW-complex.
This invariant was introduced by
Reidemeister and generalized by Franz in order to distinguish
non-homeomorphic lens spaces that have the same
homology and homotopy groups. Let $W$ be a  connected CW-complex and
let $K$ be an orthogonal representation of $\pi _{1}(W)$. Then $K$
defines a flat vector bundle with an euclidean inner product
$E_{K}$. Assume that the
chain complex of $W$ with values in $E_{K}$ is acyclic. Then the
R-torsion is the
determinant of this complex with respect to a preferred basis.

Later, Ray and Singer introduced an analytic analogue of the
R-torsion and they conjectured that, for compact riemannian manifolds,
this analytic torsion
agrees with the R-torsion. This conjecture was proved by Cheeger and
M\"uller. If $W$ is a riemannian manifold and $K$ is as before, then
we have the de Rham complex of $W$ with values in $E_{K}$ at our
disposal. The hypothesis on $K$ implies that $(\Omega ^{\ast}(W,E_{K}),\dd)$ is
also acyclic. Then the analytic torsion is essentially the determinant of
the de Rham complex. Here the difficulty lies in that the vector
spaces $\Omega ^{p}(W,E_{K})$ are infinite dimensional and therefore the
``determinant'' has to be defined using a zeta function regularization
involving the laplacian. More details on the construction of
R-torsion and analytic torsion can be found in \cite{RaySinger:RTorsion}.

Ray and Singer observed that, with the help of
hermitian metrics, the acyclicity condition can be removed. Moreover,
their definition of analytic torsion can
be extended to any elliptic complex.
In the paper \cite{RaySinger:ATCM}, they introduced a
holomorphic analogue of the analytic torsion as the determinant of the
Dolbeault complex. They also studied some of its
properties and computed some examples. In particular, they showed that
this invariant depends on the complex structure and they gave a hint
that the holomorphic analytic torsion should be interesting in number
theory. This holomorphic analytic torsion and its generalizations are
the main object of study of the present paper. Since this is the only
kind of analytic torsion that we will consider, throughout the paper,
by analytic torsion we
will mean holomorphic analytic torsion.

In the paper \cite{Quillen:dCRo}, Quillen, using the
analytic torsion, associated to each holomorphic hermitian vector
bundle on a Riemann surface a hermitian metric on the determinant of
its cohomology. Furthermore, he showed that this metric varies smoothly
with the holomorphic structure on the vector bundle. He also computed
the curvature of the hermitian line bundle on the space of
all complex structures obtained in this way.

Subsequently Bismut and Freed \cite{BismutFreed:EFI},
\cite{BismutFreed:EFII} generalized the construction of Quillen to
families of Dirac operators on the fibers of a smooth fibration. They
obtained a smooth metric and a unitary connection on the determinant
bundle associated with the family of Dirac operators. Furthermore,
they computed the curvature of this connection, which agrees with the
degree 2 part of the differential form obtained by Bismut in his proof
of the Local Index theorem \cite{Bismut:indexdirac}. Later, in a
series of papers \cite{BismutGilletSoule:at},
\cite{BismutGilletSoule:atII}, \cite{BismutGilletSoule:atIII}, Bismut,
Gillet and Soul\'e considered the case of a holomorphic submersion
endowed with a holomorphic hermitian vector bundle. They defined a
Quillen type metric on the determinant of the cohomology of the
holomorphic vector bundle. In the locally K\"ahler case, they showed
the compatibility with the constructions of Bismut-Freed. In addition
they described the variation of the Quillen metric under change of the
metric on the vertical tangent bundle and on the hermitian vector
bundle. The results of \cite{BismutGilletSoule:at},
\cite{BismutGilletSoule:atII}, \cite{BismutGilletSoule:atIII}
represent a rigidification of \cite{BismutFreed:EFI},
\cite{BismutFreed:EFII}.  All in all, these works explain the
relationship between analytic torsion and the Atiyah-Singer index
theorem and, in the algebraic case, with Grothendieck's relative
version of the Riemann-Roch theorem.

In \cite{Deligne:dc}, Deligne, inspired by the Arakelov formalism,
gave a formula for the Quillen metric that can be seen as a very
precise version of the degree one case of the Riemann-Roch theorem for
families of curves. This result is in the same spirit as the
arithmetic Riemann-Roch theorem of Faltings \cite{Faltings:cas}.

In the paper \cite{GSATAT}, Gillet and Soul\'e conjectured an
arithmetic Riemann-Roch formula that generalizes the results of
Deligne and Faltings. Besides the analytic torsion or its avatar, the
Quillen metric, this Riemann-Roch formula involves a mysterious new
odd additive characteristic class, the $R$-genus, that they computed
with the help of Zagier.

In the work \cite{BismutLebeau:CiQm} Bismut and Lebeau studied the
behavior of the analytic torsion with respect to complex
immersions. Their compatibility formula also involved the
$R$-genus. Later Bost \cite{Bost:immersion} and Roessler
\cite{Roessler:ARR} explained, using geometric arguments, why the same
genus appears both in the arithmetic Riemann-Roch formula and the
Bismut-Lebeau compatibility formula. However these geometric arguments
do not characterize the $R$-genus.

Gillet and Soul\'e \cite{GilletSoule:aRRt} proved the
degree one part of the arithmetic Riemann-Roch theorem. A crucial
ingredient of the proof is the compatibility formula of Bismut-Lebeau.

In order to establish the arithmetic Riemann-Roch theorem in all
degrees it was necessary to generalize the analytic torsion and define
higher analytic torsion classes. It was clear from
\cite{GilletSoule:aRRt} that, once a suitable theory of higher
analytic torsion classes satisfying certain properties were developed,
then the arithmetic Riemann-Roch theorem would follow. A first
definition of such forms was given by Gillet and Soul\'e in
\cite{GSATAT}, but they did not prove all the necessary properties. A
second equivalent definition was given in \cite{Bismut-Kohler} by
Bismut and K\"ohler, where some of the needed properties are
proved. The compatibility of higher analytic torsion classes with
complex immersions, i.e. the generalization of Bismut-Lebeau
compatibility formula, was proved in \cite{Bismut:Asterisque}.  As a
consequence, Gillet, Soul\'e and R\"ossler
\cite{GilletRoesslerSoule:_arith_rieman_roch_theor_in_higher_degrees}
extended the arithmetic Riemann-Roch theorem to arbitrary degrees.

In the book \cite{Faltings:RR}, Faltings followed a similar strategy
to define direct images of hermitian vector bundles and proved an
arithmetic Riemann-Roch formula up to a unique unknown odd genus.

The arithmetic Riemann-Roch theorems of Gillet-Soul\'e and Faltings
deal only with projective morphisms between arithmetic varieties such
that, at the level of complex points, define a submersion. By
contrast, in his thesis \cite{Zha99:_rieman_roch} Zha follows a
completely different strategy to establish an arithmetic Riemann-Roch
theorem without analytic torsion. His formula does not involve the
$R$-genus. Moreover Zha's theorem is valid for any projective morphism
between arithmetic varieties.

In \cite{Soule:lag}, Soul\'e advocates for an axiomatic
characterization of the analytic torsion, similar to the axiomatic
characterization of Bott-Chern classes given by Bismut-Gillet-Soul\'e
in \cite{BismutGilletSoule:at}.  Note that the R-torsion has also been
generalized to higher degrees giving rise to different higher torsion
classes. In \cite{Igusa:Axioms}, Igusa gives an axiomatic
characterization of these higher torsion classes.

We now explain more precisely what we mean by a theory of generalized
analytic torsion classes. The central point is the relationship
between analytic torsion and the Grothendieck-Riemann-Roch theorem.

Let $\pi \colon X\to Y$ be a smooth projective morphism of smooth
complex varieties. Let $\omega $ be a closed $(1,1)$ form on $X$ that
induces a K\"ahler metric on the fibers of $\pi $ and, moreover, a hermitian metric
on the relative tangent bundle $T_{\pi }$. We denote $\ov T_{\pi }$ the relative tangent bundle provided with this metric.

Let $\ov F=(F,h^{F})$ be a hermitian vector bundle on $X$ such that
for every $i\ge 0$, $R^{i}\pi _{\ast} F$ is locally free. We consider
on $R^{i}\pi _{\ast} F$ the $L^{2}$ metric obtained using Hodge theory
on the fibers of $\pi $ and denote the corresponding hermitian vector
bundle as $\ov{R^{i}\pi _{\ast}F}$.  To these data, Bismut and
K\"ohler associate an analytic torsion differential form $\tau $ that
satisfies the differential equation
 \begin{equation}\label{eq:82}
  \ast \partial\bar \partial\tau = \sum(-1)^{i}\ch(\ov{R^{i}\pi
    _{\ast}F})-\pi _{\ast}(\ch(\ov F)\Td(\ov T_{\pi })),
\end{equation}
where $\ast $ is a normalization factor that is irrelevant here (see
\ref{sec:high-analyt-tors}). Moreover, if we consider the class of
$\tau $ up to $\Im\partial+\Im\bar \partial$, then $\tau $ behaves nicely
with respect to changes of metrics.

The Grothendieck-Riemann-Roch theorem in de Rham cohomology says that
the differential form on the right side of equation \eqref{eq:82} is
exact. Therefore, the existence of the higher analytic torsion classes
provides us an analytic proof of this theorem.

Since the Grothendieck-Riemann-Roch theorem is valid with more
generality, it is natural to extend the notion of higher analytic
torsion classes to non-smooth morphisms. To this end we will use the
language of hermitian structures on the objects of the bounded derived
category of coherent sheaves developed in
\cite{BurgosFreixasLitcanu:HerStruc}. In particular we will make
extensive use of the category $\oDb$ introduced in \emph{loc. cit.}.
Since, from now on, derived
categories will be the natural framework, all functors will tacitly
be assumed to be derived functors.
By reasons explained in \emph{loc. cit.} we will restrict ourselves to
the algebraic
category. Let $f\colon X \to Y$ be a projective
morphism between smooth complex algebraic varieties. Let $\ov F$ be a
hermitian vector bundle on $X$. Now, the relative tangent complex $T_{f
}$ and the derived direct image $f_{\ast}F$ are objects of the
bounded derived category of coherent sheaves on $X$ and $Y$
respectively. Since $X$ and $Y$ are smooth, using resolutions by
locally free sheaves, we can choose hermitian structures on $T_{f
}$ and $f_{\ast}F$.
Hence we have characteristic
forms $\ch(\ov{f_{\ast}F})$ and $\Td(\ov T_{f})$.
We denote by $\overline f$ the morphism $f$ together
with the choice of hermitian structure on $T_{f}$. Then the triple
$\ov{\xi}=(\ov f,\ov F,\ov{f_{\ast}F})$ will be called a
\emph{relative hermitian vector bundle}. This is a particular case of
the relative metrized complexes of Section~\ref{sec:transv-morph}.

Then, a \emph{generalized
analytic torsion class} for $\ov{\xi }$ is the class modulo
$\Im\partial+\Im\bar \partial$
of a current that satisfies the differential equation
\begin{equation}
  \label{eq:83}
  \ast \partial\bar \partial\tau = \ch(\ov{f
    _{\ast}F})-f_{\ast}(\ch(\ov F)\Td(\ov T_{f})).
\end{equation}
Note that such current $\tau$ always exists.
Again, the Grothendieck-Riemann-Roch theorem in de Rham cohomology
implies that the right hand side of equation \eqref{eq:83} is an exact
current. Thus, if $Y$ is proper, the $dd
^{c}$-lemma implies the existence of such a current. When $Y$ is
non-proper,  a compactification argument allows us
to reduce to the proper case.

Of course, in each particular case, there are many choices for $\tau
$. We can add to $\tau $ any closed current and obtain a new solution
of equation \eqref{eq:83}. By a
\emph{theory of generalized analytic torsion classes} we mean a
coherent way of choosing a solution of equation  \eqref{eq:83} for all
relative hermitian vector bundles, satisfying certain natural
minimal set of properties.

Each theory of generalized analytic torsion classes gives
rise to a definition of direct images in arithmetic $K$-theory and
therefore to an arithmetic Riemann-Roch formula. In fact, the
arithmetic Riemann-Roch theorems of Gillet-Soul\'e and of Zha
correspond to different choices of a theory of generalized analytic
torsion
classes.
We leave for a
subsequent paper the discussion of the relation with the arithmetic
Riemann-Roch formula.

Since each projective morphism is the composition of a closed
immersion followed by the projection of a projective bundle, it is
natural to study first the analytic torsion classes for closed
immersions and projective bundles and then combine them in a global
theory of analytic torsion classes.

In \cite{BurgosLitcanu:SingularBC} the authors studied the case of closed
immersions (see Section
\ref{sec:analyt-tors-clos}). The generalized analytic torsion classes for closed
immersions are called
singular Bott-Chern classes and we will use both terms
interchangeably. The definition of a \emph{theory of
  singular Bott-Chern classes} is obtained by imposing axioms
analogous to those
defining the classical Bott-Chern classes
\cite{GilletSoule:MR854556}. Namely, a theory of singular Bott-Chern
classes is an assignment that, to each relative hermitian vector
bundle $\ov \xi =(\ov f,\ov F,\ov{f_{\ast}F})$, with $f$ a closed
immersion, assigns the class of a current $T(\ov \xi )$ on $Y$,
satisfying the
following properties:
\begin{enumerate}
\item \label{item:5} the differential equation \eqref{eq:83};
\item \label{item:10} functoriality for morphisms that are transverse to $f$;
\item \label{item:54} a normalization condition.
\end{enumerate}
A crucial observation is that, unlike
the classical situation, these axioms do not uniquely characterize
the singular Bott-Chern classes. Consequently there are various
nonequivalent
theories of singular Bott-Chern classes. They are classified
by an arbitrary characteristic class of $F$ and $T_{f}$. If we further
impose the condition that the theory is \emph{transitive} (that is,
compatible with composition of closed immersions) and \emph{compatible with
the projection formula} then the ambiguity is reduced to an arbitrary additive
genus on $T_{f}$.
 The uniqueness can be obtained by
adding to the conditions \ref{item:5}--\ref{item:54}  an
additional homogeneity property. The theory obtained is transitive and
compatible with the projection formula and agrees (up to
normalization) with the theory
introduced in \cite{BismutGilletSoule:MR1086887}.

Similarly, one can define a theory of analytic torsion classes for
projective spaces (Section \ref{sec:projsp}). This is an assignment that, to each relative
hermitian vector bundle $\ov \xi =(\ov f,\ov F,\ov{f_{\ast}F})$, where
$f\colon \PP^{n}_{Y}\to Y$ is the projection of a  trivial projective
bundle, assigns the class of a current $T(\ov \xi )$ satisfying the
properties analogous to \ref{item:5}-\ref{item:54} below, plus the
additivity and the compatibility with the projection formula.
The theories of analytic torsion classes for
projective spaces are classified by their values in the cases $Y=\Spec
\CC$, $n\ge 0$, $F=\mathcal{O}(k)$, $0\le k \le n$ for one particular
choice of metrics (see Theorem \ref{thm:1}).

We say that a theory of analytic torsion classes for closed immersions
and one for projective spaces are compatible if they satisfy a
compatibility equation similar to Bismut-Lebeau compatibility formula
for the diagonal
immersion $\Delta\colon \PP^{n}_{\CC}\to \PP^{n}_{\CC}\times
\PP^{n}_{\CC}$, $n\ge 0$. Given a theory of singular Bott-Chern classes that
is transitive and compatible with the projection formula, there exists
a unique theory of analytic torsion classes for projective spaces that
is compatible with it (Theorem \ref{thm:15}).

The central result of this paper (Theorem \ref{thm:gen_anal_tor}) is
that, given a theory of singular
Bott-Chern classes and a compatible theory of analytic torsion classes for
projective spaces, they can be combined to produce a unique theory of
generalized analytic torsion classes (Definition
\ref{def:genAT}). Moreover, every theory of analytic torsion classes
arises in this way. Thus we have a complete classification of the theories of
generalized analytic torsion classes by additive
genera.

Once we have proved the classification theorem, we derive several
applications.
The first consequence of Theorem \ref{thm:gen_anal_tor} is that the
classes of the analytic torsion forms of Bismut-K\"ohler arise as the
restriction to K\"ahler fibrations of the theory of generalized
analytic torsion classes associated to minus one half of the $R$-genus
(Theorem \ref{thm:19}). In particular, we have succeeded to extend
Bismut-K\"ohler analytic torsion classes to arbitrary projective
morphisms in the algebraic category.
Moreover, we
reprove and generalize the theorems of
Berthomieu-Bismut \cite{BerthomieuBismut} and Ma \cite{Ma:MR1765553},
\cite{Ma:MR1796698} on the compatibility of analytic torsion with the
composition of submersions (Corollary~\ref{cor:comp_sub}).

The second application of the classification theorem is a
characterization of the $R$-genus.
From the axiomatic point of view, the role played by
the $R$-genus is mysterious. It would seem more natural to consider the generalized
analytic torsion classes associated to the trivial genus
$0$. This is the choice made implicitly by Zha in his thesis
\cite{Zha99:_rieman_roch}. In fact, with our point of view, one of the
main results of Zha's thesis is the existence of a theory of analytic
torsion classes associated to the trivial genus. This theory leads to
an arithmetic Riemann-Roch
formula identical to the classical one without any correction
term. Thus, one is tempted to consider the $R$-genus as an artifact of
the analytic definition of the analytic torsion. Nevertheless, by the
work of several authors, the $R$-genus
seems to have a deeper meaning. A paradigmatic example is the
computation by Bost and K\"uhn \cite{Kuehn:gainc} of the arithmetic
self-intersection of the line bundle of modular forms on a modular
curve, provided with the Petersson metric. This formula gives an
arithmetic meaning to the first term of the $R$-genus. Thus it is
important to characterize the $R$-genus from an axiomatic
point of view and to understand its role in the above computations.

From a theorem of Bismut \cite{Bismut:deRham} we know that the
Bismut-K\"ohler analytic torsion classes of the relative de Rham complex
of a K\"ahler fibration (with the appropriate hermitian structures)
vanish. This
result is important because one of the main difficulties
to apply the arithmetic Riemann-Roch theorem is precisely the
estimation of the analytic torsion. Moreover, this result explains why
the terms of the $R$-genus appear in different arithmetic
computations. For instance, the equivariant version of this result
(due to Maillot and Roessler in degree 0 and to Bismut in general)
allows Maillot and Roessler \cite{MaillotRoessler:MR2123937} to prove
some cases of a conjecture of Gross-Deligne.

The above vanishing property characterizes the analytic
torsion classes of Bismut and K\"ohler. In order to show this,
we first construct the
dual theory $T^{\vee}$ to a given theory $T$ of generalized analytic
torsion classes (Theorem Definition \ref{thm-def:T_dual}). A theory is
self-dual ($T=T^{\vee}$) if and only if the even coefficients of the associated genus
vanish (Corollary \ref{cor:char_self_dual}). In particular,
Bismut-K\"ohler's theory is self-dual. Self-duality can also be characterized in
terms of the de Rham complex of smooth morphisms (Theorem
\ref{thm:char_van_dR}). A theory $T$ is self-dual if its components of
bidegree $(2p-1,p)$, $p$ odd, in the Deligne complex, vanish on the
relative de Rham
complexes of K\"ahler fibrations. Finally,
in Theorem \ref{thm:23} we show that, if it exists, a theory of
analytic torsion classes that vanishes, on all degrees, on the relative de Rham
complexes of K\"ahler fibrations is unique, hence it agrees with
Bismut-K\"ohler's one.
In fact, to characterize this theory, it is enough to assume the
vanishing of the analytic torsion classes for the relative de Rham
complexes of K\"ahler fibrations of
relative dimension one.
To establish this characterization we
appeal to the non-vanishing of the tautological class $\kappa_{g-2}$
on the moduli stack $\mathcal{M}_{g}$ of smooth curves of genus~$g\geq
2$.

The third application of generalized analytic torsion classes is
the construction of direct images of hermitian structures. We consider
the category $\oSm_{\ast/\CC}$ introduced in
\cite{BurgosFreixasLitcanu:HerStruc}. The objects of this category are
smooth complex varieties, and the morphisms are projective morphisms
equipped with a hermitian structure on the relative
tangent complex. Assume that we have chosen a theory of generalized
analytic torsion classes. Let $\ov f\colon X\to Y$ be a morphism in
$\oSm_{\ast/\CC}$. One would like to define a direct image functor
$f_{\ast}\colon \oDb(X)\to \oDb(Y)$.
It turns out
that, using analytic torsion, we can not define the direct image
functor on the category $\oDb$
and we have to introduce a new category $\hDb$. Roughly speaking, the
relation between $\hDb$ and $\Db$, is the same as the relation between
the arithmetic $K$-groups and the usual
$K$-groups (\cite{GilletSoule:vbhm}). Then we are able to define a
direct image functor
$f_{\ast}\colon \hDb(X)\to \hDb(Y)$,
that satisfies the composition rule, projection formula and base
change. Moreover, if the theory of generalized analytic torsion is
self-dual (as the Bismut-K\"ohler theory) this functor satisfies a
Grothendieck duality theorem. In a forthcoming paper, the direct image functor
will be the base of an arithmetic Grothendieck-Riemann-Roch
theorem for projective morphisms.

The last application that we discuss is a new proof of
a theorem of Bismut-Bost on the singularity of the Quillen metric for
degenerating families of curves, whose singular fibers have at most
ordinary double points \cite{Bismut-Bost}.
In contrast with
\emph{loc. cit.}, where the spectral definition of the Ray-Singer
analytic torsion is required, our arguments rely on the existence of a
generalized theory for arbitrary projective morphisms and some
elementary computations of Bott-Chern classes.
This theorem has already
been generalized by Bismut \cite{Bismut:degeneracy} and Yoshikawa
\cite{Yoshikawa} to families of varieties of arbitrary dimension.
In fact, our approach is very similar to the one in
\cite{Bismut:degeneracy} and \cite{Yoshikawa}. One of the main
ingredients of their proof is  Bismut-Lebeau immersion formula, while our
approach uses implicitly Bismut's generalization of the immersion
formula in the comparison between Bismut-K\"ohler analytic torsion and
a theory of generalized analytic torsion classes. But what we want to
emphasize is that, once we have identified Bismut-K\"ohler as (part
of) a theory of generalized analytic torsion classes, many arguments
can be simplified considerably because the theory has been extended to
non-smooth projective morphism.
For simplicity, we treat only the case of families of curves and the
Quillen metric, but the methods can be applied to higher dimensional
families and analytic torsion forms of higher degree.

A few words about notations.  The normalizations of characteristic
classes and Bott-Chern classes in this paper differ from the ones used
by Bismut, Gillet-Soul\'e and other authors. The first difference is
that they work with real valued characteristic classes, while we use
characteristic classes in Deligne cohomology, that naturally include
the algebro-geometric twist. The second difference is a factor $1/2$
in Bott-Chern classes, that explains the factor $1/2$ that appears in
the characteristic class associated to the
torsion classes of Bismut-K\"ohler. This change of normalization
appears already in \cite{Burgos:CDB}
and its objective is to avoid the factor $1/2$ that appears in the
definition of arithmetic degree in \cite[\S 3.4.3]{GilletSoule:ait}
and the factor $2$ that appears in \cite[Theorem
3.5.4]{GilletSoule:ait} when relating Green currents with Beilinson
regulator. The origin of this factor is that the natural second order
differential equation that appears when defining Deligne-Beilinson
cohomology is $\dd_{\mathcal{D}}=-2\partial\bar \partial$, while the
operator used when dealing with real valued forms is
\begin{math}
  \dd \dd^{c}=\frac{1}{4\pi i}\dd_{\mathcal{D}}.
\end{math}
Thus the characteristic classes that appear in the present article
only agree with the ones in the papers of Bismut, Gillet and Soul\'e
after renormalization.  With respect to the work of these authors we
have also changed the sign of the differential equation that
characterizes singular Bott-Chern classes. In this way, the same
differential equation appears when considering both, singular
Bott-Chern classes and analytic torsion classes.  This change is
necessary to combine them.

We point out that our construction of generalized analytic
torsion classes is influenced by
the thesis of Zha \cite{Zha99:_rieman_roch}, where the author
uses implicitly a theory of analytic torsion classes different from
that of Bismut-K\"ohler.

In the unpublished e-print \cite{Weng:rBCsc}, L. Weng gives another
axiomatic approach to analytic torsion classes, only for smooth morphisms between K\"ahler
fibrations. This forces him to include a
continuity condition with respect to the deformation to the normal cone
as one of the axioms. The remaining axioms he uses are: the differential
equation, functoriality with respect to cartesian squares,
compatibility with respect to the projection formula and two anomaly
formulas. A collection of differential forms satisfying these axioms
are called relative Bott-Chern secondary characteristic classes.
These characteristic classes are not unique. The main
result of Weng's paper is that any
two such
theories are related by an additive genus.
Moreover he is able to obtain a weak form of the existence theorem for
relative Bott-Chern secondary characteristic classes.

Further applications of the generalized analytic torsion
classes are left for future work. We plan to prove generalizations of
the arithmetic Grothendieck-Riemann-Roch theorem of Gillet-Soul\'e
\cite{GilletSoule:aRRt} and Gillet-R\"ossler-Soul\'e
\cite{GilletRoesslerSoule:_arith_rieman_roch_theor_in_higher_degrees}
to arbitrary projective morhisms, along the lines of
\cite{BurgosLitcanu:SingularBC}.

It is possible to compute explicitly the characteristic
numbers of the unique theory of analytic torsion classes for projective
spaces compatible with the homogeneous theory for closed
immersions. This computation
makes the characterization of generalized analytic torsion
classes more precise. Nevertheless, since this computation is much more transparent
when written in terms of properties of arithmetic Chow groups and
the Riemann-Roch theorem, we leave it to the paper devoted to the
arithmetic Riemann-Roch theorem.

We also plan to study the possible axiomatic characterization of
equivariant analytic torsion classes. Note that the characterization of
equivariant singular Bott-Chern forms has already been obtained by Tang in
\cite{Tang:ueBC}.

\section{Deligne complexes, transverse morphisms and relative metrized complexes}
\label{sec:transv-morph}

In this section we fix the notations and conventions used through the
article, we also
recall the definition of transverse morphisms and
we review some basic properties. Finally we introduce the notion of
relative metrized complex, and explain some basic constructions.

The natural context where one can define the Bott-Chern
classes and the analytic torsion classes is that of Deligne
complexes. For the convenience of the reader we will summarize in this
section the basic facts about the Deligne complexes we will use in the
sequel. For more details the reader is referred to \cite{Burgos:CDB}
and \cite{BurgosKramerKuehn:cacg}.

 \begin{definition} \label{def:22}
A \emph{Dolbeault complex} $A=(A^{\ast}_{\mathbb{R}},\dd_{A})$ is
a bounded below graded complex of real vector spaces 
equipped with a compatible  bigrading on $A_{\mathbb{C}}=A_{\mathbb{R}}
\otimes_{\mathbb{R}}{\mathbb{C}}$, i.e.,
\begin{displaymath}
A^{n}_{\mathbb{C}}=\bigoplus_{p+q=n}A^{p,q},
\end{displaymath}
satisfying the following properties:
\begin{enumerate}
\item[(i)]
The differential $\dd_{A}$ can be decomposed as the sum $\dd_{A}=
\partial+\bar{\partial}$ of operators $\partial$ of type $(1,0)$,
respectively $\bar{\partial}$ of type $(0,1)$.
\item[(ii)]
The symmetry property $\overline{A^{p,q}}=A^{q,p}$ holds, where
$\overline{\phantom{M}}$ denotes complex conjugation.
\end{enumerate}
\end{definition}

The basic example of Dolbeault complex is the complex of differential
forms on a smooth
variety $X$ over $\CC$, denoted $E^{\ast}(X)_{\RR}$.

Following \cite[\S 5.2]{BurgosKramerKuehn:cacg}, to a Dolbeault complex
one assigns a Deligne complex denoted $\mathcal{D}^{\ast}(A,\ast)$. In
this paper we will only use the following pieces of this complex:
\begin{align*}
  \mathcal{D}^{2p+1}(A,p)&= (A^{p,p+1}\oplus A^{p+1,p})\cap (2\pi
  i)^{p}A^{2p+1}_{\RR},\\
  \mathcal{D}^{2p}(A,p)&= A^{p,p}\cap (2\pi
  i)^{p}A^{2p}_{\RR},\\
  \mathcal{D}^{2p-1}(A,p)&= A^{p-1,p-1}\cap (2\pi
  i)^{p-1}A^{2p-2}_{\RR},\\
  \mathcal{D}^{2p-2}(A,p)&= (A^{p-2,p-1}\oplus A^{p-1,p-2})\cap (2\pi
  i)^{p-1}A^{2p-3}_{\RR}.
\end{align*}
The differential of the Deligne complex, denoted by
$\dd_{\mathcal{D}}\colon \mathcal{D}^{n}(A,p)\to \mathcal{D}^{n+1}(A,p)
$ is given, in the above degrees by
\begin{alignat*}{2}
  \text{if }\eta &\in \mathcal{D}^{2p}(A,p), &\quad
  \dd_{\mathcal{D}}\eta&=\dd \eta ,\\
  \text{if }\eta &\in \mathcal{D}^{2p-1}(A,p), &\quad
  \dd_{\mathcal{D}}\eta&=-2\partial\bar \partial \eta ,\\
    \text{if }\eta =(u,v)&\in \mathcal{D}^{2p-2}(A,p), &\quad
  \dd_{\mathcal{D}}\eta&=-\partial u- \bar \partial v.
\end{alignat*}

When $A$ is a Dolbeault algebra, that is, $A$ is a graded commutative
real differential algebra and the
product is compatible with the bigrading, then
$\mathcal{D}^{\ast}(A,\ast)$ has a product
\begin{displaymath}
  \bullet\colon \mathcal{D}^{n}(A,p)\otimes
  \mathcal{D}^{m}(A,q)\longrightarrow \mathcal{D}^{n+m}(A,p+q)
\end{displaymath}
that is graded commutative with respect to the first degree, it is
associative up to homotopy and satisfies the Leibnitz rule. The only
case where we will need the explicit formula for the product is for
$\omega \in
\mathcal{D}^{2p}(A,p)$ and $\eta \in \mathcal{D}^{m}(A,q)$:
\begin{math}
  \omega \bullet \eta =\omega \land \eta.
\end{math}

The \emph{Deligne algebra of differential forms} on $X$
is defined to be
\begin{displaymath}
  \mathcal{D}^{\ast}(X,\ast):=\mathcal{D}^{\ast}(E^{\ast}(X)_{\RR},\ast).
\end{displaymath}

If $X$ is equi-dimensional of dimension $d$, there is a  natural
trace map given by
\begin{displaymath}
  \int \colon H^{2d}_{c}(X,\RR(d))\to \RR,\quad
  \omega \longmapsto \frac{1}{(2\pi i)^{d}}\int_{X}\omega.
\end{displaymath}

To take this trace map into account the Dolbeault complex of currents
is constructed as follows. Denote by $E^{\ast}_{c}(X)_{\RR}$ the space
of differential forms with compact support. Then $D_{p,q}(X)$ is the
topological dual of $E^{p,q}_{c}(X)$ and $D_{n}(X)_{\RR}$ is the topological
dual of $E^{n}_{c}(X)_{\RR}$. In this complex the differential is given by
\begin{displaymath}
  \dd T(\eta)=(-1)^{n}T(\dd \eta)
\end{displaymath}
for $T\in D_{n}(X)_{\RR}$. For $X$ equi-dimensional of dimension $d$
we write
\begin{displaymath}
  D^{p,q}(X)=D_{d-p,d-q}(X),\qquad D^{n}(X)_{\RR}=(2\pi i)^{-d}D_{2d-n}(X).
\end{displaymath}

With these definitions, $D^{\ast}(X)_{\RR}$ is a Dolbeault complex and
it is a Dolbeault module over $E^{\ast}(X)_{\RR}$. We will denote
\begin{displaymath}
  \mathcal{D}_{D}^{\ast}(X,\ast):=\mathcal{D}^{\ast}(D^{\ast}(X)_{\RR},\ast).
\end{displaymath}
for the Deligne complex of currents on $X$. The trace map above defines an element
\begin{displaymath}
  \delta _{X}\in \mathcal{D}_{D}^{0}(X,0).
\end{displaymath}
More generally, if $Y\subset X$ is a subvariety of pure codimension
$p$, then the current integration along $Y$, denoted $\delta _{Y}\in
\mathcal{D}_{D}^{2p}(X,p)$  is given by
\begin{displaymath}
  \delta _{Y}(\omega )=\frac{1}{(2\pi i)^{d-p}}\int_{Y}\omega .
\end{displaymath}

Moreover, if $S\subset T^{\ast}X_{0}$ is a closed conical subset of the
cotangent bundle of $X$ with the zero section removed, we will denote
by $(\mathcal{D}_{D}^{\ast}(X,S,\ast),\dd_{\mathcal{D}})$ the Deligne
complex of currents on $X$ whose wave front set is contained in $S$.
For instance, if $N^{\ast}_{Y}$ is the conormal bundle to
$Y$, then
\begin{displaymath}
  \delta _{Y}\in \mathcal{D}^{2p}_{D}(X,N^{\ast}_{Y},p).
\end{displaymath}

If $\omega $ is a locally integrable differential form, we associate
to it a current
\begin{displaymath}
 [\omega ](\eta)=\frac{1}{(2\pi i)^{\dim X}}\int_{X}\eta\land
  \omega.
\end{displaymath}

This map induces an isomorphism $\mathcal{D}^{\ast}(X,\ast)\to
\mathcal{D}_{D}^{\ast}(X,\emptyset,\ast)$ that we use to identify
them. For instance, when in a formula
sums of currents and differential forms appear, we will tacitly assume
that the differential forms are converted into currents by this map.

Note also that, if $f\colon X\to Y$ is a proper morphism of smooth complex
varieties of relative dimension $e$, then there are direct image
morphisms
\begin{displaymath}
  f_{\ast}\colon \mathcal{D}^{n}_{D}(X,p)\longrightarrow
  \mathcal{D}^{n-2e}_{D}(X,p-e).
 \end{displaymath}
If $f$ is smooth,
the direct image of differential forms is defined by, first
converting them into currents and then applying the above direct image of
currents. If $f$ is a smooth morphism of relative dimension $e$ we can
convert them back  into differential forms. This
procedure gives us $1/(2\pi i)^{e}$ times the usual integration along the
fiber.

We shall
use the notations and definitions of \cite{BurgosLitcanu:SingularBC}. In
particular, we write
\begin{align*}
  \widetilde
  {\mathcal{D}}^{n}(X,p)&=\left. \mathcal{D}^{n}(X,p)\right/
  \dd_{\mathcal{D}}\mathcal{D}^{n-1}(X,p),\\
  \widetilde
  {\mathcal{D}}^{n}_{D}(X,p)&=\left. \mathcal{D}^{n}_{D}(X,p)\right/
  \dd_{\mathcal{D}}\mathcal{D}^{n-1}_{D}(X,p).
\end{align*}

We now recall the definition of the set of normal direction of a map
and the definition of transverse morphisms.

\begin{definition}\label{def:15}
Let $f\colon X\to Y$ be a morphism of smooth complex varieties.
  Let $T^{\ast} Y_{0}$ be the cotangent bundle to $Y$ with the zero
  section removed.
  The \emph{set of normal directions of} $f$ is the conic subset of
  $T^{\ast} Y_{0}$ given by
  \begin{displaymath}
    N_{f}=\{(y,v)\in T^{\ast} Y_{0}| \dd f^{t}v=0\}.
  \end{displaymath}
\end{definition}

\begin{definition}\label{def:4}
  Let $f\colon X\to Y$ and $g\colon Z\to Y$ be morphisms of smooth
  complex varieties. We say that $f$ and $g$ are \emph{transverse} if
  \begin{math}
    N_{f}\cap N_{g}=\emptyset.
  \end{math}
\end{definition}

It is easily seen that, if $f$ is a closed immersion, this definition
of transverse morphisms agrees with that given in
\cite[IV-17.13]{GrothendieckDieudonne:EGAIV4}.
If $f$ and $g$ are transverse, then the cartesian product
$X\underset{Y}{\times}Z$ is smooth.
For lack of a good reference we
prove the following result.

\begin{proposition}
  Let $f\colon X\to Y$ and $g\colon Z\to Y$ be transverse morphisms of smooth
  complex varieties. Then they are tor-independent.
\end{proposition}
\begin{proof}
  Since the conditions of being transverse and being
  tor-independent are both local on $Y$, $X$ and $Z$ we may assume
  that the map $f$ factorizes as $X\overset{i}{\to}Y\times
  \mathbb{A}^{n}\overset{p}{\to} Y$, where $i$ is a closed immersion and $p$
  is the projection. Let $g'\colon Z\times \mathbb{A}^{n}\to Y\times
  \mathbb{A}^{n}$
  be the morphism $g\times \Id$. If $f$ and $g$ are transverse then
  $i$ and $g'$ are transverse. While, if $i$ and $g'$ are
  tor-independent then $f$ and $g$ are tor-independent. Hence we
  may suppose that  $f$ is a closed immersion.

  Since every closed immersion between smooth schemes is regular, we
  may assume that $Y=\Spec A$, $X=\Spec A/I$, where $I$ is an ideal generated
  by a regular sequence $(s_1,\dots,s_k)$ and $Z=\Spec B$. The
  transversality condition
  implies that $(s_1,\dots,s_k)$ is a regular sequence generating $
  IB$. Let $K$ be the Koszul resolution of $A/I$ attached to the above
  sequence. Then $K\otimes_{A} B$ is the Koszul resolution of $B/IB$,
  hence exact. Therefore, $\Tor_{A}^{i}(A/I,B)=0$ for all $i\ge
  1$. Thus $f$ and $g$ are tor-independent.
\end{proof}

Let now
$Y''\overset{h}{\to}Y'\overset{g}{\to}Y$ be morphisms of smooth complex
varieties such that $g$ and $g\circ h$ are smooth. We form
the cartesian diagram
\begin{displaymath}
  \xymatrix{
  X'' \ar[r]\ar[d]^{f''}& X'\ar[r]\ar[d]^{f'} & X\ar[d]^{f} \\
  Y'' \ar[r]^{h}& Y' \ar[r]^{g} & Y.
}
\end{displaymath}
The smoothness of $g$ implies that
$N_{f'}=g^{\ast}N_{f}$. Then the smoothness of $g\circ h$ implies that
$h$ and $f'$ are transverse. Therefore, any current $\eta\in
\mathcal{D}_{D}^{\ast}(Y',N_{f'},\ast)$ can be pulled back to a
current $h^{\ast}\eta \in \mathcal{D}_{D}^{\ast}(Y'',N_{f''},\ast)$.

The following result will be used to characterize several Bott-Chern
classes and analytic torsion classes.

\begin{lemma} \label{lemm:uniquenes_phi}
  Let $f\colon X\to Y$ be a morphism of smooth complex varieties. Let
  $\widetilde \varphi$ be an assignment that, to each smooth morphism
  of complex
  varieties $g\colon Y'\to Y$ and each acyclic complex $\ov A$ of
  hermitian vector bundles on $X':=X\underset{Y}{\times }Y'$
  assigns a class
  \begin{displaymath}
   \widetilde \varphi(\ov A)\in \bigoplus
  _{n,p}\widetilde{\mathcal{D}}_{D}^{n}(Y',g^{\ast}N_{f},p)
  \end{displaymath}
  fulfilling the following properties:
  \begin{enumerate}
  \item (Differential equation) the equality
    \begin{math}
      \dd_{\mathcal{D}}\widetilde \varphi(\ov A)
      =0
    \end{math}
    holds;
  \item (Functoriality) for each morphism $h\colon Y''\to Y'$ of smooth complex varieties
     with $g\circ h $ smooth, the relation
    \begin{math}
      h^{\ast} \widetilde \varphi(\ov A)=\widetilde \varphi(h^{\ast}\ov A)
    \end{math}
    holds;
  \item (Normalization) if $\ov A$ is orthogonally split, then $\widetilde
    \varphi(\ov A)=0$.
  \end{enumerate}
  Then $\widetilde \varphi=0$.
\end{lemma}
\begin{proof}
  The argument of the proof of \cite[Thm.
  2.3]{BurgosLitcanu:SingularBC} applies \emph{mutatis mutandis} to
  the present situation. One only needs to observe that all the
  operations with differential forms of that argument can be extended to
  the currents that appear in the present situation thanks to the
  hypothesis about their wave front sets.
\end{proof}

In the paper \cite{BurgosFreixasLitcanu:HerStruc} we defined and
studied hermitian structures on objects of the bounded derived
category of coherent sheaves on a smooth complex variety. The language
and the results of \emph{loc. cit.} will be used extensively in this
paper. We just mention here that a hermitian metric on an object
$\mathcal{F}$ of $\Db(X)$ is an isomorphism $E\dra \mathcal{F}$ in
$\Db(X)$, with $E$ a bounded complex of vector bundles, together with
a choice of a hermitian metric on each constituent vector bundle of
$E$. Such an isomorphism always exists due to the fact that we work in
the algebraic category.  A hermitian structure is an equivalence class
of hermitian metrics.  To each smooth complex variety $X$, we
associated the category $\oDb(X)$
(\cite[\S~3]{BurgosFreixasLitcanu:HerStruc}) whose objects are objects
of $\Db(X)$ provided with a hermitian structure. We introduced the 
hermitian cone (\cite[Def.~3.14]{BurgosFreixasLitcanu:HerStruc}),
denoted $\ocone$, of a morphism in $\oDb(X)$. We also 
defined Bott-Chern 
classes for isomorphisms
(\cite[Thm.~4.11]{BurgosFreixasLitcanu:HerStruc}) and distinguished
triangles (\cite[Thm.~4.18]{BurgosFreixasLitcanu:HerStruc}) in
$\oDb(X)$. We introduced a universal abelian group for additive
Bott-Chern classes.  Namely, the set of hermitian structures on a zero
object of $\Db(X)$ is an abelian group that we denote $\KA(X)$
(\cite[Def.~2.31]{BurgosFreixasLitcanu:HerStruc}).
Finally, we defined the category $\oSm_{\ast/\CC}$
(\cite[\S~5]{BurgosFreixasLitcanu:HerStruc}) whose objects are
smooth complex varieties and whose morphisms are projective morphisms
together with a hermitian structure on the relative
tangent complex.

We introduce now one of the central objects of the paper.
\begin{definition} \label{def:19}
  Let $f\colon X\to Y$ be a projective morphism of smooth complex
  varieties and $\ov f\in \Hom_{\oSm_{\ast/\CC}}(X,Y)$ a morphism
  over $f$. Let $\ov{\mathcal{F}}\in \Ob \oDb(X)$ and let $\ov{\Rd
    f_{\ast}\mathcal{F}}\in \Ob \oDb(Y)$ be an object over $\Rd
  f_{\ast}\mathcal{F}$. The triple $\ov \xi=(\ov f,\ov
  {\mathcal{F}},\ov{\Rd f_{\ast}\mathcal{F}})$ will be called \emph{a
    relative metrized complex}. When $f$  is a closed immersion we
  will also call it an \emph{embedded metrized complex}. When
  $\ov{\mathcal{F}}$ and $\ov{\Rd f_{\ast}\mathcal{F}}$ are clear from
  the context we will denote the relative metrized complex $\ov \xi $
  by the morphism $\ov f$.
\end{definition}

Let $\ov \xi=(\ov f,\ov
{\mathcal{F}},\ov{\Rd f_{\ast}\mathcal{F}})$ be a relative metrized
complex and let $g\colon Y'\to Y$ be a morphism
of smooth complex varieties that is transverse to $f$. Consider the
cartesian diagram
\begin{equation}\label{eq:9}
  \xymatrix{
  X'\ar [r]^{g'}\ar[d]_{f'}& X\ar[d]^{f}\\
  Y'\ar [r]_{g}& Y.
}
\end{equation}
Then $f'$ is still projective.  Moreover, the transversality condition
implies that the canonical morphism $\Ld {g'}^{\ast} T_{\ov f}\dra T_{f'}$ is a
hermitian structure on $T_{f'}$.
 We define
\begin{equation}
  \label{eq:67}
  \Ld g^{\ast}\ov
f=(f',\Ld {g'}^{\ast} T_{\ov f})\in \Hom_{\oSm_{\ast/\CC}}(X',Y').
\end{equation}

By tor-independence, there is a canonical isomorphism
\begin{math}
 \Ld g^{\ast}\Rd f_{\ast} \mathcal{F}\dra \Rd f'_{\ast}\Ld
 {g'}^{\ast}\mathcal{F}.
\end{math}
Therefore $\Ld g^{\ast}\ov{\Rd f_{\ast} \mathcal{F}}$ induces a
hermitian structure on $\Rd f'_{\ast}\Ld
 {g'}^{\ast}\mathcal{F}$.

 \begin{definition}
   The \emph{pull-back } of $\ov \xi$ by $g$ is the relative metrized
   complex
   \begin{displaymath}
     g^{\ast}\ov \xi=(\Ld g^{\ast}\ov f,\Ld
     {g'}^{\ast}\ov{\mathcal{F}},\Ld g^{\ast}\ov{\Rd f_{\ast}
       \mathcal{F}}).
   \end{displaymath}
 \end{definition}

\begin{definition}
  Let $\ov\xi=(\ov f\colon X\to Y,\ov{\mathcal{F}},\ov{\Rd
    f_{\ast}\mathcal{F}})$ be a relative metrized complex. Let
  $\ov{\mathcal{G}}$ be an object of $\oDb(Y)$. The hermitian
  structures on $\ov{\Rd f_{\ast}\mathcal{F}}$
  and $\ov{\mathcal{G}}$ induce a natural hermitian structure on $\Rd
  f_{\ast}(\mathcal{F}\Lotimes\Ld f^{\ast}\mathcal{G})$ that we denote
  $\ov{\Rd f_{\ast}
	\mathcal{F}}\Lotimes\ov{\mathcal{G}}$.
  The \emph{tensor product of $\ov{\xi}$ by $\ov{\mathcal{G}}$} is
  then defined to be the relative metrized complex
\begin{displaymath}
	\ov{\xi}\otimes\ov{\mathcal{G}}=(\ov{f}, \ov{\mathcal{F}}
	\Lotimes\Ld f^{\ast}\ov{\mathcal{G}},\ov{\Rd f_{\ast}
	\mathcal{F}}\Lotimes\ov{\mathcal{G}}).
\end{displaymath}
\end{definition}

\begin{definition}\label{def:18}
  Let $\overline{\xi}_i=(\ov{f},\ov{\mathcal{F}_{i}},\ov{\Rd
    f_{\ast}\mathcal{F}_{i}})$, $i=1,2$ be relative metrized coherent complexes
  on $X$.
  Then \emph{the direct sum relative metrized complex} is
  \begin{displaymath}
     \overline{\xi}_1\oplus\overline{\xi}_2 :=
    (\ov{f},\ov{\mathcal{F}_{1}}\oplus
    \ov{\mathcal{F}_{2}},\ov{\Rd f_{\ast}\mathcal{F}_{1}}\oplus
    \ov{\Rd f_{\ast}\mathcal{F}_{2}}).
  \end{displaymath}
\end{definition}

We now introduce a notation for Todd-twisted direct images of currents
and differential forms, that will simplify many formulas
involving the Todd genus.
 Let $\ov f=(f,\ov T_f)$ be a morphism in $\oSm_{\ast/\CC}$. To $\ov
 f$ we associate a Todd differential form $\Td(\ov f):=\Td(\ov
 T_{f})\in \bigoplus_{p} \mathcal{D}^{2p}(X,p)$
 \cite[(5.15)]{BurgosFreixasLitcanu:HerStruc}.
 Let
$S$
be a closed conic subset of $T^{\ast}X_{0}$. Then we denote
\begin{equation}
  \label{eq:28}
  f_{\ast}(S)=\{(f(x),\eta )\in T^{\ast}Y_{0}\mid (x,(\dd
  f)^{t}\eta)\in S\}\cup N_{f}.
\end{equation}
If $g\colon Y\to Z$ is another morphism of smooth complex varieties,
it is easy to see that we have $(g\circ f)_{\ast}(S)\subseteq
g_{\ast}f_{\ast}(S)$.

\begin{definition}
  Let $\ov f\colon X\to Y$ be a morphism in $\ov\Sm_{\ast/\CC}$ of
  relative dimension~$e$. For each closed conical subset $S\subset
  T^{\ast}X_{0}$ and each pair of integers $n$, $p$, we
  define the map
  \begin{displaymath}
    \ov f_{\ttwist}\colon \mathcal{D}^{n}_{D}(X,S,p)\to
    \mathcal{D}^{n-2e}_{D}(Y,f_{\ast}S,p-e), \quad
    \ov f_{\ttwist}(\omega )= f_{\ast}(\omega \bullet \Td(\ov f)).
  \end{displaymath}
\end{definition}

\begin{proposition}\label{prop:12}
  Let  $\ov f\colon X\to Y$  and $\ov g\colon Y \to Z$ be morphisms
  in $\ov\Sm_{\ast/\CC}$ of relative dimensions $e_{1}$ and $e_{2}$
  respectively. Let $S\subset T^{\ast}X_{0}$, $T\subset T^{\ast}Y_{0}$ be closed conical subsets and
  let $\ov h=\ov f\circ \ov g$, of relative dimension $e=e_1+e_2$.
  \begin{enumerate}
  \item \label{item:prop_12_2}The following diagram is commutative
    \begin{displaymath}
      \xymatrix{
        \mathcal{D}_{D}^{n}(X,S,p)\ar[r]^{\ov f_{\ttwist}}
        \ar[d]_{\ov h_{\ttwist}}&
        \mathcal{D}_{D}^{n-2e_1}(Y,f_{\ast} S,p-e_1) \ar [d]^{\ov g_{\ttwist}} \\
        \mathcal{D}_{D}^{n-2e}(Z,h_{\ast} S,p-e) \ar@{^{(}->}[r]&
        \mathcal{D}_{D}^{n-2e}(Z,g_{\ast}f_{\ast} S,p-e).
      }
     \end{displaymath}
   \item \label{item:prop_12_3} Let
     $\theta\in\mathcal{D}_{D}^{m}(X,S,q)$ and
     $\omega\in\mathcal{D}_{D}^{n}(Y,T,p)$. Assume $T\cap
     N_{f}=\emptyset$ and that $T+f_{\ast}S$ is disjoint with the zero
     section in $T^{\ast}Y_{0}$. Then $f^{\ast}T+S$ is disjoint with
     the zero section and there is an equality of currents
     \begin{displaymath}
     	\ov{f}_{\ttwist}(f^{\ast}(\omega)\bullet\theta)=\omega\bullet \ov{f}_{\ttwist}(\theta)
     \end{displaymath}
     in $\mathcal{D}_{D}^{n+m}(Y,W,p+q)$, with
     \begin{math}
     	W=f_{\ast}(S + f^{\ast}T)\cup f_{\ast}S\cup f_{\ast}f^{\ast}T.
     \end{math}
 \end{enumerate}
\end{proposition}
\begin{proof}
  For the first assertion, it is enough to notice the equality of currents
  \begin{displaymath}
  	\ov{g}_{\ttwist}(\ov{f}_{\ttwist}(\omega))=(g\circ f)_{\ast}(\omega\bullet f^{\ast}\Td(\ov{g})\bullet\Td(\ov{f}))).
  \end{displaymath}
  For the second, it is easy to see that $f^{\ast}T+S$ does not cross
  the zero section, and hence both sides of the equality are
  defined. It then suffices to establish the equality of currents
  \begin{math}
  	f_{\ast}(f^{\ast}\omega\bullet\theta)=f_{\ast}(\omega)\bullet\theta.
  \end{math}
  If $\theta$ and $\omega$ are smooth, then the equality follows from
  the definitions. The general case
  follows by approximation of $\theta$ and $\omega$ by smooth currents
  and the continuity of the operators $f^{\ast}$ and $f_{\ast}$.
\end{proof}

\begin{proposition} \label{prop:11} Let $\ov{f}$ be a morphism in
  $\ov\Sm_{\ast/\CC}$ of relative dimension $e$ and $S$ a closed
  conical subset of $T^{\ast}X_{0}$. Let $g\colon Y'\to Y$ be a morphism of
  smooth complex varieties transverse to $f$. Consider the cartesian
  diagram \eqref{eq:9} and let $\ov f'=g^{\ast} \ov f$.
  Suppose that $N_{g'}$ is
  disjoint with $S$. Then:
  \begin{enumerate}	
  \item $N_{g}$ and $f_{\ast}S$ are disjoint and
    $g^{\ast}f_{\ast}S\subset f'_{\ast}{g'}^{\ast}S$;
  \item the following diagram commutes:
    \begin{displaymath}
      \xymatrix{
        \mathcal{D}_{D}^{n}(X,S,p)\ar[r]^{\hspace{-1cm}\ov{f}_{\ttwist}}
        \ar[d]_{{g'}^{\ast}}
        &\mathcal{D}_{D}^{n-2e}(Y,f_{\ast}S,p-e)\ar[d]^{g^{\ast}}\\
        \mathcal{D}_{D}^{n}(X',{g'}^{\ast}S,
        p)\ar[r]^{\hspace{-1cm}\ov{f}'_{\ttwist}}
        &\mathcal{D}_{D}^{n-2e}(Y',f'_{\ast}{g'}^{\ast}S,p-e)\\
      }
    \end{displaymath}
  \end{enumerate}
\end{proposition}
\begin{proof}
  The first claim follows from the definitions. In
  particular the diagram makes sense. For the commutativity of the
  diagram, we observe that, since
  \begin{math}
    g'{}^{\ast}\Td(\ov f)=\Td(\ov f'),
  \end{math}
  it suffices to check the equality of currents
  \begin{math}
  	g^{\ast}f_{\ast}(\theta)={f'}_{\ast}{g'}^{\ast}(\theta)
  \end{math}
  for $\theta\in\mathcal{D}_{D}^{n}(X,S,p)$.

  By the continuity of the operators $g^{\ast}$, ${g'}^{\ast}$,
  $f_{\ast}$ and ${f'}_{\ast}$, it is enough to prove the relation
  whenever $\theta$ is smooth. Moreover, using a partition of unity
  argument we are reduced to the following local analytic statement.
  \begin{lemma} \label{lemm:1}
    Let $f\colon X\to Y$ and $g\colon Y'\to Y$ be transverse morphisms
    of complex manifolds. Let $\theta $ be a smooth differential form
    on $X$ with compact support. Consider the diagram
    \eqref{eq:9}. Then
    \begin{equation}\label{eq:52}
      g^{\ast}f_{\ast}(\theta)={f'}_{\ast}{g'}^{\ast}(\theta).
    \end{equation}
  \end{lemma}
  \noindent
  \emph{Proof.}
    The map $f$ can be factored as $X\overset{\varphi
    }{\longrightarrow}X\times Y\overset{p_{2}}{\longrightarrow} Y$,
    where $\varphi(x)=(x,f(x))$ is a closed immersion and $p_{2}$, the
    second projection, is smooth. Using again the continuity of the
    operators $g^{\ast}$ (respectively ${g'}^{\ast}$) and
    $f_{\ast}$ (respectively ${f'}^{\ast}$), we are reduced to prove the
    equation \eqref{eq:52} in the case when $f$ is smooth and in the
    case when $f$ is a closed immersion. The case when $f$ is smooth
    is clear. Assume now that $f$ is a closed immersion. By
    transversality, $f'$ is also a closed immersion of complex
    manifolds. We may assume that $\theta=f^{\ast}\widetilde \theta $
    for some smooth form $\widetilde \theta $ on $Y$. Then equation
    \eqref{eq:52} follows from the chain of equalities
    \begin{displaymath}
      g^{\ast}f_{\ast}\theta = g^{\ast}f_{\ast}f^{\ast}\widetilde
      \theta= g^{\ast}(\widetilde \theta \land \delta _{X})=
      g^{\ast}(\widetilde \theta) \land \delta _{X'}
      =f'_{\ast}f'{}^{\ast}g^{\ast}\widetilde \theta
      = f'_{\ast} g'{}^{\ast} f^{\ast}\widetilde \theta
      = f'_{\ast} g'{}^{\ast} \theta .
    \end{displaymath}
    This concludes the proof of the lemma and the proposition. \hfill
    $\square$
\end{proof}

\section{Analytic torsion for closed immersions}
\label{sec:analyt-tors-clos}

In the paper \cite{BurgosLitcanu:SingularBC} the authors study the
singular Bott-Chern classes associated to closed immersions of smooth
complex varieties. The singular Bott-Chern classes are the analogue,
for closed immersions, of the analytic torsion for smooth morphisms.
For this reason, we will call them also analytic torsion classes.
The aim of this section is to
recall the main results of \cite{BurgosLitcanu:SingularBC} and to
translate them into the language of derived categories.

\begin{definition} \label{def:1}
  A \emph{theory of analytic torsion classes for closed immersions} is
  a map that, to each embedded metrized complex $\ov \xi =(\ov
  f\colon X\to Y,
  \ov{\mathcal{F}}, \ov{\Rd f_{\ast}\mathcal{F}})$
  assigns a class
  \begin{displaymath}
    T(\ov \xi) \in
    \bigoplus _{p}\widetilde{\mathcal{D}}^{2p-1}_{D}(Y,N_{f},p)
  \end{displaymath}
  satisfying the following conditions.
  \begin{enumerate}
  \item (Differential equation) The equality
    \begin{math}
      \dd_{\mathcal{D}}T(\ov \xi) =
      \ch(\ov{\Rd
        f_{\ast}\mathcal{F}})-\ov f_{\ttwist}[\ch(\ov{\mathcal{F}})]
    \end{math}
    holds.
  \item (Functoriality) For every morphism $h\colon Y'\to Y$ of smooth complex varieties
     that is
    transverse to $f$ we have the equality
    \begin{math}
      h^{\ast}T(\ov \xi)=
      T(h^{\ast} \ov \xi).
    \end{math}
  \item (Normalization) If $X=\emptyset$ (hence $\ov
    {\mathcal{F}}=\ov 0$), $Y=\Spec \CC$, and $\ov {\Rd
      f_{\ast}\mathcal{F}}=\ov 0$, then
    \begin{math}
      T(\ov f,\ov 0,\ov 0)=0.
    \end{math}
  \end{enumerate}
\end{definition}

\begin{definition} Let $T$ be a theory of analytic torsion classes for closed
  immersions.
  \begin{enumerate}
  \item We say that $T$
    is \emph{compatible with the projection formula} if, for every embedded
    metrized complex $\ov \xi=(\ov f, \ov{\mathcal{F}}, \ov{\Rd
      f_{\ast}\mathcal{F}})$,
    and every object $\ov {\mathcal{G}}\in \oDb(Y)$, we have
    \begin{equation}\label{eq:69}
      T(\ov{\xi}\otimes\ov{\mathcal{G}})=
      T(\ov \xi) \bullet
      \ch(\ov {\mathcal{G}}).
    \end{equation}
  \item We say that $T$ is \emph{additive} if, given  $\overline{\xi}_i=
      (\ov{f},\ov{\mathcal{F}_{i}},\ov{\Rd f_{\ast}\mathcal{F}_{i}})$,
      $i=1,2$, two embedded
    metrized complexes, we have
    \begin{equation}
      \label{eq:49}
      T(\ov{\xi}_1\oplus\ov{\xi}_{2})=T(\ov\xi_{1})+T(\ov{\xi}_{2}).
    \end{equation}
  \item We say that $T$ is \emph{transitive} if, given a embedded
    metrized complex
    $\ov \xi=(\ov f, \ov{\mathcal{F}}, \ov{\Rd f_{\ast}\mathcal{F}})$, a
    closed immersion of smooth complex varieties
    $g\colon Y\to Z$, a morphism  $\ov g$ over $g$, and an object
  $\ov{\Rd (g\circ f)_{\ast}\mathcal{F}}\in \Ob \oDb(Z)$ over $\Rd
  (g\circ f)_{\ast}\mathcal{F}$, we have
  \begin{equation}
    \label{eq:68}
    T(\ov g\circ \ov f)=
    T(\ov g) +
    \ov g_{\ttwist}(T(\ov f)).
  \end{equation}
  \end{enumerate}
\end{definition}

\begin{remark} \label{rem:4}
  \begin{enumerate}
  \item \label{item:52} If $T$ is well defined for objects of $\oDb$,
    then the normalization condition in Definition \ref{def:1} and the
    normalization
    condition in \cite[Def.
    6.9]{BurgosLitcanu:SingularBC} are equivalent. The compatibility
    with the projection formula implies
    the normalization condition and the additivity (see
    \cite[Prop. 10.9]{BurgosLitcanu:SingularBC})
  \item \label{item:53} To check that a theory
    is compatible with the projection formula, it is enough to
    consider complexes consisting of a single
    hermitian vector bundle in degree 0.
  \end{enumerate}
\end{remark}

Let $X$ be a smooth complex variety and let $\ov N$ be a hermitian
vector bundle of rank~$r$. We denote by $P=\PP(N\oplus \mathbf{1})$
the projective
bundle obtained by completing $N$. Let $\pi _{P}\colon P\to X$ be the
projection and let $s\colon X\to P$ be the zero
section. Since $N$ can be identified with the normal bundle to $X$ on
$P$, the hermitian metric of $\ov N$ induces a hermitian structure on
$s$. We will denote it by $\ov s$. On $P$ we
have a tautological quotient vector bundle with an
induced metric $\ov Q$. For each hermitian vector bundle $\ov F$ on
$X$ we have the Koszul resolution $K(F,N)$  of $s_{\ast}F$. We
denote by $K(\ov F,\ov N)$ the Koszul resolution with the induced
metrics. See \cite[Def. 5.3]{BurgosLitcanu:SingularBC} for
details.

\begin{definition} \label{def:23}
  Let $T$ be a theory of analytic torsion classes for closed
  immersions. We say that $T$ is \emph{homogeneous} if, for every pair
  of hermitian vector bundles $\ov N$ and $\ov F$ with $\rk N=r$, there
  exists a homogeneous class of bidegree  $(2r-1,r)$  in the Deligne
  complex
  $$\widetilde e(\ov F,\ov N)\in
  \widetilde{\mathcal{D}}^{2r-1}_{D}(P,N_{s},r)$$
  such that
  \begin{equation}\label{eq:43}
    T(\ov s,\overline F,K(\ov F,\ov N))\bullet \Td(\ov Q)=
    \widetilde e(\ov F,\ov N)\bullet
    \ch(\pi _{P}^{\ast}\ov F).
  \end{equation}
\end{definition}

\begin{remark}
  Observe that Definition \ref{def:23} is equivalent to
  \cite[Def. 9.2]{BurgosLitcanu:SingularBC}. The advantage of
  the definition in this paper is that it treats on equal footing the
  case when $\rk F=0$.
\end{remark}

Let $\DD$ denote the base ring for Deligne cohomology (see
\cite{BurgosLitcanu:SingularBC} before Definition 1.5). A consequence
of \cite[Thm. 1.8]{BurgosLitcanu:SingularBC} is that there is a
bijection between the set of additive genus in Deligne cohomology and
the set of power series in one variable $\DD[[x]]$. To each power
series $\varphi \in \DD[[x]]$ it corresponds the unique additive genus
such that
\begin{math}
  \varphi(L)=\varphi (c_{1}(L))
\end{math}
for every line bundle $L$.

\begin{definition} \label{def:25}
  A \emph{real additive genus} is an additive genus such that the
  corresponding power series belongs to $\RR[[x]]$.
\end{definition}

Let
$\mathbf{1}_{1}\in \DD$ be the element represented by the constant
function 1 of $\mathcal{D}^{1}(\Spec \mathbb{C},1)=\RR$.
The main result of \cite{BurgosLitcanu:SingularBC} can be written
as follows.

\begin{theorem}\label{thm:17}
  \begin{enumerate}
  \item \label{item:55} There is a unique homogeneous theory of
    analytic torsion classes for
    closed immersions, that we denote $T^{h}$. This theory is compatible
    with the projection formula, additive and transitive.
  \item \label{item:56} Let $T$ be any transitive theory of analytic
    torsion classes
    for closed immersions, that is compatible with the projection
    formula. Then there is a unique real additive genus $S_{T}$
    such that, for any embedded metrized
    complex $\ov \xi :=(\ov f, \ov{\mathcal{F}}, \ov{\Rd
      f_{\ast}\mathcal{F}})$, we have
    \begin{equation}
      \label{eq:70}
      T(\ov \xi )-T^{h}(\ov \xi )=-f_{\ast}[\ch(\mathcal{F})\bullet
      \Td(T_{f}) \bullet S_{T}(T_{f})\bullet \mathbf{1}_{1}].
    \end{equation}
  \item \label{item:57}
    Conversely, any real additive
    genus $S$ defines, by means of equation \eqref{eq:70}, a
    unique transitive
    theory of analytic torsion classes $T_S$
    for closed immersions, that is compatible with the projection
    formula and additive.
  \end{enumerate}
  \end{theorem}
  \begin{proof}
    Existence and uniqueness for both $T^{h}$ and $T_S$ is the content of
    \cite{BurgosLitcanu:SingularBC} when restricting to triples
    $\ov{\xi}$ with $T_{\ov f}=\overline{N}_{X/Y}[-1]$, $\ov{\mathcal{F}}$ a
    hermitian vector bundle placed in degree 0 and $\ov{\Rd
    f_{\ast}\mathcal{F}}$ given by a finite locally free resolution. For
    the general case, we thus need to prove that the theories of
    analytic torsion classes for closed immersions in the sense of
    \emph{loc. cit.} uniquely extend to arbitrary $\ov{\xi}$, fulfilling
    the desired properties.

    Assume given a theory $T$ in the sense of
    \cite{BurgosLitcanu:SingularBC}, compatible with the projection
    formula and transitive. We will call $T$ the initial theory.
    We consider a triple $\ov{\xi}$ with $T_{\ov
      f}=\overline{N}_{X/Y}[-1]$ and $\ov{\mathcal{F}}\in\Ob\oDb(X)$. Choose
    a representative $\ov{F}\dra\mathcal{F}$ of the hermitian structure
    on $\ov{\mathcal{F}}$. We then define $T(\ov{\xi})$ by induction on
    the length of the complex $F$. First suppose that $F=F^{d}[-d]$
    consists of a
    single vector bundle placed in degree $d$. Choose a finite
    locally free resolution
    \begin{displaymath}
      \dots\to E^{1}\to E^{0}\to f_{\ast} F^{d}\to 0.
    \end{displaymath}
    Endow the vector bundles $E^{i}$ with smooth hermitian
    metrics. Observe that there is an induced isomorphism
    \begin{math}
      \overline{E}[-d]\overset{\sim}{\dra}\overline{\Rd f_{\ast}\mathcal{F}},
    \end{math}
     in $\oDb(Y)$,
    whose Bott-Chern classes $\cht$ are
    defined in \cite[\S~4]{BurgosFreixasLitcanu:HerStruc}. We then put
    \begin{equation}\label{eq:72}
      T(\ov{\xi})=(-1)^{d}T(\ov{N}_{X/Y}, \ov{F}^{d},\ov{E})
      +\widetilde{\ch}(\ov{E}[-d]\overset{\sim}{\dra}\ov{\Rd
        f_{\ast}\mathcal{F}}).
    \end{equation}
    This definition does not depend on the choice of
    representative of the hermitian structure on $\ov{\mathcal{F}}$, nor
    on the choice of $\overline{E}$, due to
    \cite[Thm.~4.11, Prop.~4.13]{BurgosFreixasLitcanu:HerStruc},
    and \cite[Cor.~6.14]{BurgosLitcanu:SingularBC}. The differential equation is
    satisfied as a consequence of the differential equations for
    $T(\ov{N}_{X/Y},\ov{F}^{d},\ov{E})$ and
    $\widetilde{\ch}(\ov{E}[-d]\overset{\sim}{\dra}\ov{\Rd
      f_{\ast}\mathcal{F}})$. The compatibility with pull-back by
    morphisms $h\colon Y'\to Y$ transverse to $f$ is immediate as
    well. Finally, notice that by construction, if
    $\ov{\xi}'=(\ov{N}_{X/Y}, \ov{\mathcal{F}}, \overline{\Rd
      f_{\ast}\mathcal{F}}')$, then
    \begin{equation}\label{eq:71}
      T(\ov{\xi'})=T(\ov{\xi})+
      \widetilde{\ch}(\ov{\Rd f_{\ast}\mathcal{F}}' ,
      \ov{\Rd f_{\ast}\mathcal{F}}).
    \end{equation}
    Now suppose that $T(\ov{\xi})$ has been defined for $F$ of length
    $l$, satisfying in addition (\ref{eq:71}). If $F$ has length $l+1$,
    let $F^{d}$ be the first non-zero vector bundle of $F$. Consider the
    exact sequence of complexes
    \begin{displaymath}
      (\ov{\varepsilon})\qquad 0\to\sigma^{> d}\ov{F}\to \ov{F}\to
      \ov{F}^{d}[-d]\to 0,
    \end{displaymath}
    where $\sigma^{>d}$ is the b\^ete filtration. Observe that as a
    distinguished triangle, $(\ov{\varepsilon})$ is tightly
    distinguished, hence
    $\widetilde{\ch}(\ov{\varepsilon})=0$. Choose hermitian metrics on
    $\Rd f_{\ast}\sigma^{>d}F$ and $\Rd f_{\ast} F^{d}[-d]$. We thus have a
    distinguished triangle in $\oDb(Y)$
    \begin{displaymath}
      (\ov{\tau})\qquad \ov{\Rd f_{\ast}\sigma^{>d}F}\to\ov{\Rd f_{\ast} F}
      \to\ov{\Rd f_{\ast} F^{d}}[-d]\to\ov{\Rd f_{\ast}\sigma^{>d}F}[1]\to\dots.
    \end{displaymath}
    We define
    \begin{equation} \label{eq:74}
      T(\ov{\xi})=T(\ov{N}_{X/Y},\sigma^{>d}\ov{F},\ov{\Rd
        f_{\ast}\sigma^{>d}F})
      +(-1)^{d}T(\ov{N}_{X/Y}, \ov{F}^{d},\ov{\Rd f_{\ast} F^{d}})
      -\widetilde{\ch}(\ov{\tau}).
    \end{equation}
    This does not depend on the choice of hermitian structures on $\Rd
    f_{\ast}\sigma^{>d}F$ and $\Rd f_{\ast} F^{d}$, by the analogue to
    \cite[Thm.~3.33~(vii)]{BurgosFreixasLitcanu:HerStruc} for
    $\widetilde{\ch}$ and because \eqref{eq:71} holds by assumption
    for $T(\ov{N}_{X/Y},\sigma^{>d}\ov{F},\ov{\Rd
      f_{\ast}\sigma^{>d}F})$ and $T(\ov{N}_{X/Y}, \ov{F}^{d},\ov{\Rd
      f_{\ast} F^{d}})$. Similarly, \eqref{eq:71} holds for
    $T(\ov{\xi})$ defined in \eqref{eq:74}.  The differential equation
    and compatibility with pull-back are proven as in the first
    case. This concludes the proof of the existence in case that
    $T_{\ov f}=\overline{N}_{X/Y}[-1]$.

    To conclude with the existence, we may now consider a general
    $\ov{\xi}$. Choose a hermitian metric on the normal bundle
    $N_{X/Y}$. Put
    $\ov{\xi}'=(\ov{N}_{X/Y}[-1],\ov{\mathcal{F}},\ov{\Rd
      f_{\ast}\mathcal{F}})$.
    We define
    \begin{equation}\label{eq:73}
      T(\ov{\xi})=T(\ov{\xi}')+
      \ov f_{\ttwist}[\ch(\ov{\mathcal{F}})\bullet
      \widetilde{\Td}_{m}(
      T_{\ov f}\dra \ov{N}_{X/Y}[-1])],
    \end{equation}
    where $\widetilde{\Td}_{m}$ is the multiplicative Todd secondary
    class defined in \cite[\S~5]{BurgosFreixasLitcanu:HerStruc}.
    It is straightforward from the definition that $T(\ov{\xi})$ satisfies
    the differential equation and is compatible with pull-back by
    morphisms transverse to $f$. We call $T$ the extended theory.

    We now proceed to prove that the extended theory $T$ is transitive
    and compatible with the projection formula. For the projection
    formula, it suffices by Remark \ref{rem:4}~\ref{item:53} to prove
    \begin{displaymath}
      T(\ov{\xi}\otimes\ov{G})=T(\ov{\xi})\bullet\ch(\ov{G})
    \end{displaymath}
    for $\ov{G}$ a hermitian vector bundle placed in degree 0. This
    readily follows from the inductive construction of the extended
    theory $T$ and the assumptions on the initial theory $T$. One
  similarly establishes the transitivity and the additivity

  We conclude by observing that, since Lemma \ref{lemm:uniquenes_phi}
  implies that the equations \eqref{eq:72}, \eqref{eq:71},
  \eqref{eq:74} and \eqref{eq:73} hold, the theory $T(\ov{\xi})$ thus
  constructed for arbitrary $\ov{\xi}$ is completely determined by the
  values $T(\ov{\xi}')$, with $\ov{\xi}'$ of the form $(\ov{N}_{X/Y},
  \ov{F},\overline{E})$ where $\ov{F}$ is a hermitian vector bundle
  and $E\to f_{\ast}F$ is a finite locally free resolution.

  Once we have seen that any theory of singular Bott-Chern classes as
  in \cite{BurgosLitcanu:SingularBC} can be uniquely extended, then
  statements \ref{item:56} and \ref{item:57} follow combining equation
  (7.3) and Corollary 9.43 in \cite{BurgosLitcanu:SingularBC}. Note
  that the minus sign in equation \eqref{eq:70} comes from the fact
  that $S(T_{f})=-S(N_{X/Y})$.
\end{proof}

In \cite[\S 6]{BurgosLitcanu:SingularBC} several
anomaly formulas are proved. We now indicate the translation of these
formulas to the current setting. Recall that we are using the notation
of \cite{BurgosFreixasLitcanu:HerStruc} with respect to secondary
characteristic classes.

\begin{proposition}\label{prop:1bis}
  Let $T$ be a theory of analytic torsion classes for closed immersions. Let
  \begin{math}
    \overline{\xi}= (\ov{f}\colon X\to Y,\ov{\mathcal{F}}, \ov{\Rd
      f_{\ast}\mathcal{F}})
  \end{math}
  be an embedded metrized complex.
  \begin{enumerate}
  \item \label{item:6bis}
    If $\ov{\mathcal{F}}'$ is another choice
    of metric on $\mathcal{F}$ and
    $\overline{\xi}_1=(\ov{f}\colon X\to Y,\ov{\mathcal{F}}', \ov{\Rd
      f_{\ast}\mathcal{F}})$, then
    \begin{displaymath}
      T(\overline{\xi}_1)=T(\overline{\xi})+\ov f
      _{\ttwist}[\cht(\ov{\mathcal{F}}',\ov{\mathcal{F}})].
    \end{displaymath}
  \item \label{item:7bis}
    If $\ov{f}'$ is another hermitian structure on
    $f$ and
    $\overline{\xi}_{2}=(\ov{f}'\colon X\to Y,\ov{\mathcal{F}}, \ov{\Rd
      f_{\ast}\mathcal{F}})$, then
    \begin{equation}\label{eq:14bis}
      T(\overline{\xi}_{2})=T(\overline{\xi})+\ov f'
      _{\ttwist}[\ch(\overline {\mathcal{F}})\bullet
      \widetilde{\Td}_{m}(\ov{f}',\ov{f})].
    \end{equation}
  \item \label{item:8bis}
    If $\ov{\Rd f_{\ast}\mathcal{F}}'$ is another
    choice of metric on $\Rd f_{\ast}\mathcal{F}$, and $\overline
    {\xi}_{3}=(\ov{f}\colon X\to Y,\ov{\mathcal{F}}, \ov{\Rd
      f_{\ast}\mathcal{F}}')$, then
    \begin{displaymath}
      T(\overline{\xi}_{3})= T(\overline{\xi})-
      \cht(\ov{\Rd f_{\ast}\mathcal{F}}',
      \ov{\Rd f_{\ast}\mathcal{F}}).
    \end{displaymath}
  \end{enumerate}
\end{proposition}
\begin{proof}
  We first prove the second assertion. Let $\ov{E}\dra
  T_{f}$ be a representative of the hermitian structure on
  $T_{\ov{f}}$. By
  \cite[Thm.~3.13~(ii)]{BurgosFreixasLitcanu:HerStruc}, we may assume
  the
  hermitian structure on $T_{\ov{f}'}$ is represented by the
  composition
  \begin{displaymath}
    \ov{E}\oplus\ov{A}\overset{\pr_{1}}{\longrightarrow }\ov{E}\dra T_{f}
  \end{displaymath}
  for some bounded acyclic complex $\ov A$ of hermitian vector
  bundles on $X$. For every
  smooth morphism $g\colon Y'\to Y$ of complex varieties, consider the
  cartesian diagram~\eqref{eq:9}.
  We introduce the assignment that, to every such $g$ and each bounded
  acyclic complex of hermitian vector bundles $\ov{A}$ on $X'$,
  assigns the class
  \begin{displaymath}
    \begin{split}
      \widetilde{\varphi}(\ov{A})=&T({g'}^{\ast}\ov{\xi})
      -T\left((f',\Ld {g'}^{\ast}T_{\ov{f}}+ [\ov{A}]),\Ld
        {g'}^{\ast}\ov{\mathcal{F}},
	\Ld g^{\ast}\ov{\Rd f_{\ast}\mathcal{F}}\right)\\
      &+{f'}_{\ast}\left[\ch(\Ld {g'}^{\ast}\ov{\mathcal{F}})
	\widetilde{\Td}_{m}\left((\Ld{g'}^{\ast}T_{\ov{f}} +
          [\ov{A}]),\Ld {g'}^{\ast}T_{\ov{f}}\right) \Td(\Ld
        {g'}^{\ast}T_{\ov{f}}+[\ov{A}])\right].
    \end{split}
  \end{displaymath}
  Here $[\ov{A}]$ stands for the class of $\ov{A}$ in
  $\KA(X')$ (\cite[Def.~2.31]{BurgosFreixasLitcanu:HerStruc}) and $+$
  denotes the action of $\KA(X')$ on $\oDb(X')$
  (\cite[Thm.~3.13]{BurgosFreixasLitcanu:HerStruc}). Since
  $\widetilde{\varphi}$ satisfies the hypothesis of Lemma
  \ref{lemm:uniquenes_phi}, we have
  $\widetilde{\varphi}=0$. This concludes the proof of
  \ref{item:7bis}.

  To prove \ref{item:6bis}, we observe that
  $\ov{\mathcal{F}}'=\ov{\mathcal{F}}+ [\ov A]$ for some bounded
  acyclic complex $\ov A$ of hermitian vector bundles on $X$. For each
  cartesian diagram as
  \eqref{eq:9}, we set $\ov f'=g^{\ast}
  \ov f$. Let $\widetilde {\varphi}_{1}$ be the assignment that,
  to each such diagram and each bounded
  acyclic complex of hermitian vector bundles $\ov{A}$ on $X'$,
  assigns the class
  \begin{displaymath}
    \widetilde{\varphi}_{1}(\ov{A})=T({g'}^{\ast}\ov{\xi})
    -T\left(\ov f',\Ld
      {g'}^{\ast}\ov{\mathcal{F}}+[A],
      \Ld g^{\ast}\ov{\Rd f_{\ast}\mathcal{F}}\right)
    -{\ov f'}_{\ttwist}[\cht(A)].
  \end{displaymath}
  The hypothesis of
  Lemma \ref{lemm:uniquenes_phi} are satisfied, hence
  $\widetilde{\varphi}_{1}=0$. This concludes the proof of~\ref{item:6bis}.

  Finally, to prove \ref{item:8bis}, to each morphism $g\colon Y'\to
  Y$, transverse to $f$,
  we associate the cartesian diagram \eqref{eq:9} and we consider the
  assignment
  $\widetilde{\varphi}_{2}$ that,
  to each bounded acyclic complex of hermitian vector bundles
  $\ov{B}$ on $Y'$, assigns the class
  \begin{displaymath}
    \widetilde{\varphi}_{2}(\ov{B})=T({g'}^{\ast}\ov{\xi})
    -T\left(\ov f',\Ld
      {g'}^{\ast}\ov{\mathcal{F}},
      \Ld g^{\ast}\ov{\Rd f_{\ast}\mathcal{F}}+[B]\right)
    +\cht(B).
  \end{displaymath}
  By Lemma \ref{lemm:uniquenes_phi} applied to $\Id_{Y}$, we have
  $\widetilde{\varphi}_{2}=0$. This concludes the proof of~\ref{item:8bis}.
\end{proof}

The following result provides a compatibility relation for analytic
torsion classes for closed immersions with respect to distinguished
triangles. The statement is valid for additive theories, in particular
the ones we are concerned with.
\begin{proposition}\label{prop:2bis}
  Let $T$ be an additive theory of analytic torsion classes for closed
  immersions. Let $f\colon X\to Y$ be a closed immersion of smooth
  complex varieties. Consider distinguished triangles in $\oDb(X)$ and
  $\oDb(Y)$ respectively,
\begin{displaymath}
	(\ov{\tau}):\
        \ov{\mathcal{F}}_{2}\to\ov{\mathcal{F}}_{1}
        \to\ov{\mathcal{F}}_{0}\to\ov{\mathcal{F}}_{2}[1],\quad
	(\ov{\Rd f_{\ast}\tau}):\
        \ov{\Rd f_{\ast}\mathcal{F}}_{2}\to
        \ov{\Rd f_{\ast}\mathcal{F}}_{1}
	\to\ov{\Rd f_{\ast}\mathcal{F}}_{0}\to
        \ov{\Rd f_{\ast}\mathcal{F}}_{2}[1],
\end{displaymath}
and the relative hermitian complexes $\ov{\xi}_{i}=(\ov{f},\ov{\mathcal{F}}_i,
\ov{\Rd f_{\ast}\mathcal{F}}_i),$ $i=0,1,2$.
Then we have:
\begin{displaymath}
	\sum_{j}(-1)^{j}T(\ov{\xi}_{j})
	=\widetilde{\ch}(\ov{\Rd f_{\ast}\tau})
	-\ov f_{\ttwist}(\widetilde{\ch}(\ov{\tau })).
\end{displaymath}
\end{proposition}
\begin{proof}
  We can assume that the distinguished triangles $\ov \tau $ and
  $\ov{\Rd f_{\ast}\tau }$ can be
  represented by short exact sequences of complexes of hermitian vector
  bundles
  \begin{align*}
    \ov \varepsilon :\quad &0\longrightarrow \ov E_{2}
    \longrightarrow \ov E_{1}
    \longrightarrow \ov E_{0} \longrightarrow 0,\\
    \ov \nu : \quad &0\longrightarrow \ov V_{2}
    \longrightarrow \ov V_{1}
    \longrightarrow \ov V_{0} \longrightarrow 0.
  \end{align*}
  Applying the explicit construction at the beginning of the proof of
  \cite[Theorem 2.3]{BurgosLitcanu:SingularBC} to each row of the
  above exact sequences, we obtain exact sequences
  \begin{align*}
    \widetilde \varepsilon^{i} :\quad &0\longrightarrow \widetilde  E_{2}^{i}
    \longrightarrow \widetilde E_{1}^{i}
    \longrightarrow \widetilde E_{0}^{i} \longrightarrow 0,\\
    \widetilde \nu^{i} :\quad  &0\longrightarrow \widetilde V_{2}^{i}
    \longrightarrow \widetilde V_{1}^{i}
    \longrightarrow \widetilde V_{0}^{i} \longrightarrow 0
  \end{align*}
  over $X\times \PP^{1}$ and $Y\times \PP^{1}$
  respectively. The restriction of $\widetilde \varepsilon ^{i} $
  (respectively $\widetilde \nu ^{i}$) to $X\times \{0\}$
  (respectively $Y\times \{0\}$) agrees with $\ov \varepsilon $
  (respectively $\ov \nu $). Whereas the restriction to
  $X\times \{\infty\}$ (respectively $Y\times \{\infty\}$)
  is orthogonally split. The sequences $\widetilde \varepsilon ^{i}$ and
  $\widetilde \nu ^{i}$ form exact sequences of complexes that we
  denote $\widetilde \varepsilon $ and
  $\widetilde \nu$. It is easy to verify that the restriction to
  $X\times \{\infty\}$ (respectively $Y\times \{\infty\}$)
  are orthogonally split as sequences of complexes. Moreover, there are
  isomorphisms $\widetilde V_{j}\dra \Rd f _{\ast}\widetilde E_{j}$,
  $j=0,1,2$. We denote
  \begin{math}
    \widetilde \xi _{j}=(\ov f \times \Id_{\PP^{1}},\widetilde
    E_{j},\widetilde V_{j}).
  \end{math}
  Then, in the group
  $\bigoplus_{p}\widetilde{\mathcal{D}}_{D}^{2p-1}(Y,N_{f},p)$, we
  have
  \begin{align*}
    0&=\dd_{\mathcal{D}}\frac{1}{2\pi i}\int_{ \PP^{1}}
    \frac{-1}{2}\log t\bar t\bullet\sum_{j}
    (-1)^{j}T(\widetilde{\xi}_{j})\\
    &=T(\overline {\xi}_{1})-T(\overline
    {\xi}_{0}\oplus \overline{\xi}_{2})-
    \frac{1}{2\pi i}\int_{ \PP^{1}}
    \frac{-1}{2}\log t\bar t\bullet\sum_{j}
    (-1)^{j}\ch(\widetilde V_{j})\\
    &+
  \frac{1}{2\pi i}\int_{ \PP^{1}}
    \frac{-1}{2}\log t\bar t\bullet
    \sum_{j}(-1)^{j}(f\times \Id_{\PP^{1}})
    _{\ast}(\ch(\widetilde{E}_{j})\bullet
    \Td(
    \overline{f}\times \Id_{\PP^{1}} ) )\\
    &=T(\overline {\xi}_{1})-T(\overline
    {\xi}_{0}\oplus \overline{\xi}_{2})
    +\widetilde{\ch}(\ov{\Rd f_{\ast}\tau})
    -f_{\ast}(\widetilde{\ch}(\ov{\tau })\Td(\ov{f})).
  \end{align*}
Thus the result follows from the additivity.
\end{proof}

We end this chapter with the relation between the singular Bott-Chern
classes of Bismut-Gillet-Soul\'e \cite{BismutGilletSoule:MR1086887}
and the theory of homogeneous analytic torsion classes. We draw
attention to the difference of normalizations. Let us momentarily
denote by $\tau$ the theory of singular Bott-Chern classes of
Bismut-Gillet-Soul\'e. By the anomaly formulas, it may be extended to
arbitrary embedded metrized complexes. Let $\ov{\xi}=(\ov{f}:X\to
Y,\ov{\mathcal{F}},\ov{\Rd f_{\ast}\mathcal{F}})$ be a relative
metrized complex, with $Y$ of dimension $d$. If $\tau^{(p-1,p-1)}$
denotes the component of degree $(p-1,p-1)$ of the current $\tau$, we
define
\begin{equation}\label{eq:88}
  T^{BGS}(\ov{\xi})^{(2p-1,p)}= -\frac{1}{2(2\pi
    i)^{d-(p-1)}}\tau^{(p-1,p-1)}
  \in\widetilde{\mathcal{D}}_{D}^{2p-1}(Y,N_{f},p).
\end{equation}
In the above equation, the factor $(2\pi i)^{(p-1)}$ comes from the
difference in the normalization of characteristic classes. In
\cite{BismutGilletSoule:MR1086887} the authors use real valued classes
while we use twisted coefficients. The factor $(2\pi i)^{d}$ comes
from our convention about the Deligne complex of currents. The factor
$2$ comes from the fact that the second order differential operator
that appears in the Deligne complex is $-2\partial\bar \partial=2(2\pi
i)d d^{c}$, while the second order differential operator that appears
in the differential equation considered by Bismut, Gillet and Soul\'e
is $dd^{c}$. The main reason behind this change is that we want the
Bott-Chern classes to be related to Beilinson's regulator and not to
twice Beilinson's regulator (see \cite{GilletSoule:ait} Theorem
3.5.4). Finally, the minus sign comes from the discrepancy of the
differential equations of the singular Bott-Chern forms of
Bismut-Gillet-Soul\'e and the analytic torsion forms of
Bismut-K\"ohler. Note that we are forced to change this sign because
we want to merge singular Bott-Chern forms and analytic torsion forms
on a single theory.
We put
\begin{displaymath}
	T^{BGS}(\ov{\xi})=\sum_{p\geq
          1}T^{BGS}(\ov{\xi})^{(2p-1,p)}\in
        \bigoplus_{p}\widetilde{\mathcal{D}}_{D}^{2p-1}(Y, N_{f},p).
\end{displaymath}
We have the following comparison theorem
\cite[Thm. 9.25]{BurgosLitcanu:SingularBC}.

\begin{theorem}\label{thm:18}
For every embedded metrized complex $\xi$ we have
\begin{displaymath}
	T^{BGS}(\ov{\xi})=T^{h}(\ov{\xi}).
\end{displaymath}
\end{theorem}

\section{Regular coherent sheaves}
\label{sec:regsheaf}

In this section we recall some properties of regular sheaves.  Let $X$
be a scheme and let $\PP^n_X=\PP_X(V)$ be the projective space of
lines of the trivial bundle $V$ of rank $n+1$ on $X$. Let $\pi\colon
\PP^n_X\rightarrow X$ be the natural projection. By abuse of notation,
if ${\mathcal{G}}$ is a sheaf on $X$, we will denote also by
${\mathcal{G}}$ the inverse image $\pi ^{\ast}{\mathcal{G}}$.

\begin{definition}[\cite{Mumford:Lectures_on_Curves}, Lecture
  14]\label{def:regsheaf}
  A quasi-coherent sheaf ${\mathcal{F}}$ on $\PP^n_X$ is called
  \textit{regular} if $R^q\pi_{\ast}{\mathcal{F}}(-q)=0$ for all
  $q>0$.
\end{definition}

Recall the following properties of regular sheaves (see \cite{Quillen:haKt}).
\begin{enumerate}
\item If ${\mathcal{G}}$ is a quasi-coherent sheaf on $X$, then
  ${\mathcal{G}}\otimes_X \mathcal{O}_{\PP^n_X}(k)$ is regular for
  $k\ge 0$.
\item If the scheme $X$ is noetherian and ${\mathcal{F}}$ is a
  coherent sheaf on $\PP^n_X$, then Serre's vanishing theorem implies
  that for $d$ large enough ${\mathcal{F}}(d)$ is regular.
\item Let $0\rightarrow {\mathcal{F}}_2\rightarrow {\mathcal{F}}_1
  \rightarrow {\mathcal{F}}_0 \rightarrow 0$ be an exact sequence of
  quasi-coherent sheaves on $\PP^n_X$ and $d$ be an integer. Then
  \begin{enumerate}
  \item if ${\mathcal{F}}_{2}(d)$ and ${\mathcal{F}}_{0}(d)$ are
    regular, then ${\mathcal{F}}_{1}(d)$ is regular;
  \item if ${\mathcal{F}}_{2}(d+1)$ and ${\mathcal{F}}_{1}(d)$ are
    regular, then ${\mathcal{F}}_{0}(d)$ is regular;
  \item if ${\mathcal{F}}_0(d)$ and ${\mathcal{F}}_1(d+1)$ are regular
    and the map $R^0\pi_{\ast}({\mathcal{F}}_1(d))\rightarrow
    R^0\pi_{\ast}({\mathcal{F}}_0(d))$ is surjective, then
    ${\mathcal{F}}_2(d+1)$ is regular.
  \end{enumerate}
\item If ${\mathcal{F}}$ is regular, then ${\mathcal{F}}(k)$ is regular for $k>0$.
\item If ${\mathcal{F}}$ is regular, then the canonical map $R^0\pi_{\ast}{\mathcal{F}}\otimes _X
  \mathcal{O}_{\PP^n_X}\rightarrow {\mathcal{F}}$ is surjective.
\end{enumerate}

\begin{theorem}[{\cite[\S
8.1]{Quillen:haKt}}]\label{thm:2}
  Let ${\mathcal{F}}$ be a regular quasi-coherent sheaf on
  $\PP^{n}_{X}$. Then there exists a \textit{canonical resolution}
\begin{equation*}
  \gamma_{\can}({\mathcal{F}})\ :\ 0\rightarrow {\mathcal{G}}_n(-n)\rightarrow
  {\mathcal{G}}_{n-1}(-n+1)\rightarrow \dots \rightarrow
  {\mathcal{G}}_0\rightarrow {\mathcal{F}}
  \rightarrow 0
\end{equation*}
where ${\mathcal{G}}_i$ ($i=0,\dots,n$) are quasi-coherent sheaves on
$X$.  Moreover, for every $k\geq 0$, the sequence
\begin{displaymath}
  0\rightarrow {\mathcal{G}}_k\rightarrow {\mathcal{G}}_{k-1}
  \otimes\Sym^1V^{\vee}
  \rightarrow \dots\rightarrow {\mathcal{G}}_{0}\otimes\Sym^kV^{\vee}
  \rightarrow R^0\pi_{\ast}({\mathcal{F}}(k))\rightarrow 0
\end{displaymath}
is exact. Hence the sheaves ${\mathcal{G}}_k$ are determined
by ${\mathcal{F}}$ up to unique isomorphism.
\end{theorem}

\begin{corollary} \label{cor:4}
Let $X$ be a noetherian scheme and ${\mathcal{F}}$ a coherent
sheaf on $\PP^n_X$.  Then, for $d$ large enough, we have a resolution
$$\gamma_d({\mathcal{F}})\ :\ 0\rightarrow {\mathcal{G}}_n(-n-d)\rightarrow
{\mathcal{G}}_{n-1}(-n-d+1)\rightarrow \dots \rightarrow
{\mathcal{G}}_0(-d)\rightarrow {\mathcal{F}} \rightarrow 0$$ where
${\mathcal{G}}_i$, $i=0,\dots,n$ are coherent sheaves on $X$.
\end{corollary}

\begin{example}\label{exm:1}
  The sheaf $\mathcal{O}(1)$ is regular. Its canonical resolution is
  \begin{displaymath}
    0\to
    \Lambda^{n+1}V^{\vee}(-n)\to\Lambda ^{n}V^{\vee}(-n+1)\to\dots\to
    \Lambda ^{2}V^{\vee}(-1)\to V^{\vee}\to \mathcal{O}(1)\to 0.
  \end{displaymath}
  Twisting this exact sequence by $\mathcal{O}(-1)$ we obtain the
  Koszul exact sequence
  \begin{displaymath}
    0\to
    \Lambda^{n+1}V^{\vee}(-n-1)\to\Lambda ^{n}V^{\vee}(-n)\to\dots\to
    \Lambda ^{2}V^{\vee}(-2)\to V^{\vee}(-1)\to \mathcal{O}\to 0,
  \end{displaymath}
  that we denote $K$. We will denote by $K(k)$ its twist by $\mathcal{O}(k)$.
\end{example}

\begin{theorem}[{\cite{Zha99:_rieman_roch}}]\label{thm:14}
\begin{enumerate}
\item \label{item:9} Let ${\mathcal{F}}$ be a regular coherent sheaf
  on $\PP^{n}_{X}$, and let $\gamma_{\can}({\mathcal{F}})$ be the canonical
  resolution of ${\mathcal{F}}$ as in Theorem \ref{thm:2}. Let
\begin{displaymath}
\varepsilon_1\ :\ 0\rightarrow {\mathcal{F}}_{n+k}(-n-k)\rightarrow \dots
\rightarrow {\mathcal{F}}_1(-1)\rightarrow {\mathcal{F}}_0\rightarrow
{\mathcal{F}}\rightarrow 0
\end{displaymath}
be an exact sequence of coherent sheaves, where the
${\mathcal{F}}_{i}$ are sheaves on $X$. Then there exist natural
surjective morphisms of sheaves ${\mathcal{F}}_i\rightarrow
{\mathcal{G}}_i$ on $X$, $0\leq i\leq n$ making commutative the
diagram
\begin{displaymath}
  \xymatrix{
    {\mathcal{F}}_{n+1}(-n-1)\ar[d] \ar[r] & {\mathcal{F}}_n(-n)
    \ar@{->>}[d] \ar[r] & \dots
    \ar[r] & {\mathcal{F}}_0 \ar@{->>}[d]
    \ar[r] & {\mathcal{F}} \ar@{=}[d] \ar[r] & 0\\
    0 \ar[r] & {\mathcal{G}}_n(-n) \ar[r] & \dots \ar[r] &
    {\mathcal{G}}_0 \ar[r] &
    {\mathcal{F}} \ar[r] & 0.
  }
\end{displaymath}
\item \label{item:11} Let ${\mathcal{F}}$ be a regular coherent sheaf on $X$, and
  $\gamma_{\can}({\mathcal{F}})$ the canonical resolution.
There exists a resolution of
${\mathcal{F}}(1)$ of the form
\begin{displaymath}
\varepsilon_{2}\ :\ 0 \rightarrow {\mathcal{S}}_{n+k}(-n-k)\rightarrow
\dots\rightarrow
       {\mathcal{S}}_1(-1) \rightarrow {\mathcal{S}}_0 \rightarrow {\mathcal{F}}(1)\rightarrow 0
\end{displaymath}
such that ${\mathcal{S}}_0\dots,{\mathcal{S}}_{n+k}$ are coherent
sheaves on $X$ and the
following diagram of exact sequences with surjective vertical arrows
is commutative:
\begin{displaymath}
    \xymatrix{
        {\mathcal{S}}_{n+1}(-n-1)\ar[d] \ar[r] &
        {\mathcal{S}}_n(-n) \ar@{->>}[d] \ar[r] & \dots
        \ar[r] &
        {\mathcal{S}}_0 \ar@{->>}[d] \ar[r] &
        {\mathcal{F}}(1) \ar@{=}[d] \ar[r] & 0\\
        0 \ar[r] & {\mathcal{G}}_n(-n+1) \ar[r] & \dots \ar[r] &
        {\mathcal{G}}_0(1) \ar[r] & {\mathcal{F}}(1) \ar[r] & 0.
    }
\end{displaymath}
\end{enumerate}
\end{theorem}
\begin{proof}
  We introduce the sheaves ${\mathcal{N}}_{j}$ and
  ${\mathcal{K}}_{j}$ defined as the kernels at each term of the sequences
  $\gamma_{\can}$ and $\varepsilon_{1}$, respectively. Hence, there
  are exact sequences
\begin{align*}
  &0\to {\mathcal{N}}_{j+1}(j+1)\to {\mathcal{G}}_{j+1}\to
  {\mathcal{N}}_{j}(j+1)\to 0,\\
  &0\to {\mathcal{K}}_{j+1}(j+1)\to {\mathcal{F}}_{j+1}\to
  {\mathcal{K}}_{j}(j+1)\to 0.
\end{align*}
With these notations, observe that
${\mathcal{N}}_{-1}={\mathcal{K}}_{-1}={\mathcal{F}}$. By induction,
starting from the left hand side of the long exact sequences, it
is easily checked that ${\mathcal{N}}_{j}(j+1)$ and
${\mathcal{K}}_{j}(j+1)$ are regular
sheaves, for $j\geq -1$. Also, by Theorem \ref{thm:2}, we find that
${\mathcal{G}}_{j+1}=\pi_{\ast}({\mathcal{N}}_{j}(j+1))$ for $j\geq
-1$. We next prove by
induction that, for each $k\geq -1$,
there is a commutative diagram of exact sequences
\begin{equation}\label{eq:78}
  \xymatrix{
    &0\ar[d] &0\ar[d] &0\ar[d] &\\
    0\ar[r] &{\mathcal{P}}_{k+1}\ar[r]\ar[d]
    &{\mathcal{H}}_{k+1}\ar[r]\ar[d] &{\mathcal{P}}_{k}(1)\ar[d]&\\
    0\ar[r] &{\mathcal{K}}_{k+1}(k+1)\ar[r]\ar[d]
    &{\mathcal{F}}_{k+1}\ar[r]\ar[d]
    &{\mathcal{K}}_{k}(k+1)\ar[r]\ar[d] &0\\
    0\ar[r] &{\mathcal{N}}_{k+1}(k+1)\ar[r]
    &{\mathcal{G}}_{k+1}\ar[r]\ar[d]
    &{\mathcal{N}}_{k}(k+1)\ar[r]\ar[d] &0\\
    & &0 &0, &
  }
\end{equation}
where ${\mathcal{H}}_{k+1}$, ${\mathcal{P}}_{k}$ and
${\mathcal{P}}_{k+1}$ are defined as the kernels of the corresponding
morphisms, and that ${\mathcal{P}}_{k}(1)$ is
regular. Assume that this is true for a fixed $k\ge 1$. In order to
proceed with the induction, we need to prove: (a)
the map ${\mathcal{K}}_{k+1}(k+2)\to {\mathcal{N}}_{k+1}(k+2)$ is
surjective, (b) the sheaf ${\mathcal{P}}_{k+1}(1)$ is regular, and (c)
there is an induced surjective map
${\mathcal{F}}_{k+2} \to {\mathcal{G}}_{k+2}$.

We first prove (a). If we apply $\pi_{\ast}$ to the last two
columns of diagram \eqref{eq:78}. Observing that
${\mathcal{F}}_{k+1}$, ${\mathcal{G}}_{k+1}$ and ${\mathcal{H}}_{j+1}$
are actually sheaves on $X$ and recalling that
${\mathcal{K}}_{k+1}(k+2)$ is regular (so that
$R^{1}\pi_{\ast}{\mathcal{K}}_{k+1}(k+1)=0$), we find a commutative
diagram of exact sequences
\begin{displaymath}
  \xymatrix{
    0\ar[r]		&
    {\mathcal{H}}_{k+1}\ar[r]\ar[d]	&
    {\mathcal{F}}_{k+1}\ar[r]\ar[d]	&
    {\mathcal{G}}_{k+1}\ar@{=}[d]\ar[r]	&
    0	\\
    0\ar[r]	&
    \pi_{\ast}({\mathcal{P}}_{k}(1))\ar[r]	&
    \pi_{\ast}({\mathcal{K}}_{k}(k+1))\ar[d]\ar[r] &
    \pi_{\ast}({\mathcal{N}}_{k}(k+1))\ar[r]	&
    0\\
    &&0&&
  }
\end{displaymath}
It follows that the map
${\mathcal{H}}_{k+1}\twoheadrightarrow\pi_{\ast}({\mathcal{P}}_{k}(1))$ is a
surjection. Since ${\mathcal{P}}_{k}(1)$ is
regular, we have that
$\pi_{\ast}({\mathcal{P}}_{k}(1))\otimes\OO_{\PP^{n}_{X}}\twoheadrightarrow
{\mathcal{P}}_{k}(1)$ is also a surjection. Thus the map
$\mathcal{H}_{k+1}\to \mathcal{P}_{k}(1)$ is surjective. The diagram
\eqref{eq:78} implies that the map
\begin{math}
	{\mathcal{K}}_{k+1}(k+1)\to {\mathcal{N}}_{k+1}(k+1)
\end{math}
 is also surjective. Twisting by ${\mathcal{O}}(1)$, we obtain (a).

 Now the regularity of ${\mathcal{H}}_{k+1}$ and
 ${\mathcal{P}}_{k}(1)$, and
 the surjectivity of
 ${\mathcal{H}}_{k+1}\twoheadrightarrow\pi_{\ast}({\mathcal{P}}_{k}(1))$
  ensure the regularity of ${\mathcal{P}}_{k+1}(1)$.
 In its turn, this shows that the sequence
\begin{equation}\label{eq:11}
  0\to\pi_{\ast}({\mathcal{P}}_{k+1}(1))\to\pi_{\ast}({\mathcal{K}}_{k+1}(k+2))
  \to\pi_{\ast}({\mathcal{N}}_{k+1}(k+2))\to 0
\end{equation}
is exact. Finally, we observe that there is a surjective map
\begin{equation}\label{eq:77}
  \xymatrix{
    {\mathcal{F}}_{k+2}\ar@{->>}[r]	&\pi_{\ast}({\mathcal{K}}_{k+1}(k+2)),
  }
\end{equation}
by the regularity of ${\mathcal{K}}_{k+2}(k+3)$. From
\eqref{eq:11} and \eqref{eq:77}, we obtain a surjection
\begin{displaymath}
  \xymatrix{
    {\mathcal{F}}_{k+2}\ar@{->>}[r]
    &\pi_{\ast}({\mathcal{N}}_{k+1}(k+2))={\mathcal{G}}_{k+2}.
  }
\end{displaymath}
This completes the proof of the inductive step. Note that the first
step of the induction ($k=-1$) is part of the data. From the existence
of the diagrams we deduce (i).

To prove the second item we construct the resolution
$\mathcal{S}_{\ast}$ inductively. We will denote by
${\mathcal{K}}_{k}$ the kernel of any map ${\mathcal{S}}_{k}(-k)\to
{\mathcal{S}}_{k-1}(-k+1)$ already defined and by ${\mathcal{N}}_{k}$ the
successive kernels of the canonical resolution of ${\mathcal{F}}$ as
in the proof of the first statement.

Assume that we have constructed the
sequence $\varepsilon _{2}$ up to ${\mathcal{S}}_{k}(-k)$ with the
further conditions that
${\mathcal{K}}_{k}(k+1)$ is regular and that there is an exact sequence
\begin{displaymath}
  0\to {\mathcal{P}}_{k}(1)\to {\mathcal{K}}_{k}(k+1)\to
  {\mathcal{N}}_{k}(k+2)\to 0
\end{displaymath}
with ${\mathcal{P}}_{k}(1)$
regular. We have to show that we can extend the resolution one step
satisfying the same conditions.
Recall that we already know that ${\mathcal{N}}_{k}(k+1)$ is regular.
We consider as well the surjection
\begin{math}
  \xymatrix{
    {\mathcal{G}}_{k+1}(1)\ar@{->>}[r] &{\mathcal{N}}_{k}(k+2).
  }
\end{math}
We form the fiber product
\begin{displaymath}
  {\mathcal{T}}_{k+1}:=\Ker({\mathcal{K}}_{k}(k+1)\oplus
  {\mathcal{G}}_{k+1}(1)\to {\mathcal{N}}_{k}(k+2)).
\end{displaymath}
Observe that ${\mathcal{T}}_{k+1}$ is regular, because both
${\mathcal{N}}_{k}(k+1)$,
${\mathcal{K}}_{k}(k+1)\oplus {\mathcal{G}}_{k+1}(1)$ are regular and
the morphism
\begin{displaymath}
  \xymatrix{
    \pi_{\ast}({\mathcal{K}}_{k}(k)\oplus
    {\mathcal{G}}_{k+1})\ar@{->>}[r]
    &{\mathcal{G}}_{k+1}=\pi_{\ast}({\mathcal{N}}_{k}(k+1))
    }
\end{displaymath}
is surjective. So are the arrows ${\mathcal{T}}_{k+1}\to
{\mathcal{G}}_{k+1}(1)$ and
${\mathcal{T}}_{k+1}\to {\mathcal{K}}_{k}(k+1)$. Therefore, if we define
${\mathcal{S}}_{k+1}=\pi_{\ast} ({\mathcal{T}}_{k+1})$,
we have a commutative diagram of exact sequences
\begin{displaymath}
  \xymatrix{
    &0\ar[d] &0\ar[d] &0\ar[d] &\\
    0\ar[r] &{\mathcal{P}}_{k+1}\ar[r]\ar[d]
    &{\mathcal{H}}_{k+1}\ar[r]\ar[d] &{\mathcal{P}}_{k}(1)\ar[d]&\\
    0\ar[r] &{\mathcal{K}}_{k+1}(k+1)\ar[r]\ar[d]
    &{\mathcal{S}}_{k+1}\ar[r]\ar[d]
    &{\mathcal{K}}_{k}(k+1)\ar[r]\ar[d] &0\\
    0\ar[r] &{\mathcal{N}}_{k+1}(k+2)\ar[r] &{\mathcal{G}}_{k+1}(1)\ar[r]\ar[d]
    &{\mathcal{N}}_{k}(k+2)\ar[r]\ar[d] &0\\
    & &0 &0, &
  }
\end{displaymath}
where ${\mathcal{H}}_{k+1}$ and ${\mathcal{P}}_{k+1}$ are defined as
the kernels of the corresponding morphisms. Thus we have been able to
extend the resolution one step further. We still need to show that
this extension satisfies the extra properties.
We observe that, by the definition of ${\mathcal{S}}_{k+1}$ and the left
exactness of direct images, the map $\pi _{\ast}({\mathcal{S}}_{k+1})\to \pi
_{\ast}({\mathcal{G}}_{k+1}(1))$ is surjective. Therefore
${\mathcal{H}}_{k+1}(1)$ is
regular. Moreover, one can check that ${\mathcal{S}}_{k+1}$ is the fiber product
\begin{displaymath}
  {\mathcal{S}}_{k+1}=\Ker\left(\pi
_{\ast}({\mathcal{G}}_{k+1}(1))\oplus \pi _{\ast}({\mathcal{K}}_{k}(k+1))\to \pi
_{\ast}({\mathcal{N}}_{k}(k+2))\right).
\end{displaymath}
This implies easily that
$\pi _{\ast}({\mathcal{H}}_{k+1})=\pi
_{\ast}({\mathcal{P}}_{k}(1))$. We also observe that, by definition of
fiber product,
 ${\mathcal{P}}_{k}(1)=\Ker({\mathcal{T}}_{k+1}\to
 {\mathcal{G}}_{k+1}(1))$. Since ${\mathcal{S}}_{k+1}$ surjects onto
 ${\mathcal{T}}_{k+1}$, we deduce that the morphism
 ${\mathcal{H}}_{k+1}\to {\mathcal{P}}_{k}(1)$ is
 surjective. From this we conclude that the morphism
 ${\mathcal{K}}_{k+1}(k+2)\to
 {\mathcal{N}}_{k+1}(k+3)$ is surjective and that
the sheaf ${\mathcal{P}}_{k+1}(1)$ is regular. Since
${\mathcal{N}}_{k+1}(k+3)$ is regular, we
deduce that ${\mathcal{K}}_{k+1}(k+2)$ is regular. Therefore
$\mathcal{S_{k+1}}$ satisfies all the required properties, concluding
the proof of \ref{item:11}.
\end{proof}

We end this section recalling the notion of generating class of a
triangulated category.

\begin{definition}\label{def:gen_class}
  Let $\mathbf{D}$ be a triangulated category. A \emph{generating
    class} is a subclass $\mathbf{C}$ of $\mathbf{D}$ such that the
  smallest triangulated subcategory of $\mathbf{D}$ that contains
  $\mathbf{C}$ is equivalent to $\mathbf{D}$ via the inclusion.
\end{definition}

A well-known direct consequence of Theorem \ref{thm:2} is the following result.

\begin{corollary} \label{cor:6}
  The class of objects of the form $\mathcal{G}(k)$, with
  $\mathcal{G}$ a coherent sheaf in $X$ and $-n\leq k\leq 0$, is a
  generating class of
  $\Db(\PP^{n}_{X})$.
\end{corollary}

\section{Analytic torsion for projective spaces}
\label{sec:projsp}

Let $n$ be a non-negative integer, $V$ the $n+1$ dimensional
vector space
  $\CC^{n+1}$ and
$\PP^n:=\PP^n(V)$ the projective space of lines in $V$. We write
$\overline V$ for the vector space $V$ together with the trivial
metric.
We
will denote by $V$ the trivial vector bundle of fiber $V$ over any
base.

We may
construct natural relative hermitian complexes that arise by
considering the sheaves $\mathcal{O}(k)$, their
cohomology and the Fubini-Study metric.
If we endow the trivial sheaf with the trivial metric and
$\mathcal{O}(1)$ with the Fubini-Study metric, then the tangent bundle
$T_{\pi}$ carries a quotient hermitian structure via the short exact
sequence
\begin{equation} \label{eq:4}
	0\to\mathcal{O}_{\PP^{n}_{\CC}}\to
	\mathcal{O}(1)^{n+1}\to T_{\pi}\to 0.
\end{equation}
We will denote the resulting hermitian vector bundle by
$\ov{T}_{\pi}^{\FS}$ and call it the Fubini-Study metric of $T_{\pi
}$.  The arrow $(\pi,
\ov{T}_{\pi}^{\FS})$ in $\ov\Sm_{\ast/\CC}$ will be written
$\ov{\pi}^{\FS}$.
We endow the invertible sheaves $\OO(k)$ with the $k$-th tensor power
of the Fubini-Study metric on $\OO(1)$. We refer to them by
$\overline{\OO(k)}$.

We now describe natural hermitian structures on the complexes $\Rd \pi
_{\ast} \mathcal{O}(k)$.
First assume $k\geq 0$. The sheaf $\OO(k)$ is $\pi$-acyclic, hence
\begin{math}
  \Rd\pi_{\ast}\mathcal{O}(k)=\text{H}^{0}(\PP^{n}_{\CC},\mathcal{O}(k))
\end{math}
as a complex concentrated in degree 0. This space is naturally
equipped with the $L^{2}$ metric with respect to the Fubini-Study
metric on $\OO(k)$ and the volume form
$\mu=c_{1}(\overline{\OO(1)})^{\land n}/n!$ on $\PP^{n}_{\CC}$. Namely, given
global sections $s,t$ of $\OO(k)$,
\begin{displaymath}
  \langle s,t\rangle_{L^{2}}=\int_{\PP^{n}_{\CC}} \langle
  s(x),t(x)\rangle_{x}\mu(x).
\end{displaymath}

If $-n\leq k<0$, then $\Rd\pi_{\ast}\mathcal{O}(k)=0$
and we put the trivial metric on it.
Finally, let $k\leq -n-1$. Then the cohomology of
$\Rd\pi_{\ast}\mathcal{O}(k)$ is concentrated in degree $n$ and there
is an isomorphism,
\begin{displaymath}
  \Rd\pi_{\ast}\mathcal{O}(k)\cong
  \text{H}^{0}(\PP^{n}_{\CC},\OO(-k-n-1))^{\vee}[-n].
\end{displaymath}
Notice that this
isomorphism is canonical due to Grothendieck duality and to the
natural identification
$\omega_{\PP^{n}_{\CC}}=\mathcal{O}(-n-1)$. Hence we may endow
$\Rd\pi_{\ast}\mathcal{O}(k)$ with the dual of the $L^{2}$ metric on $\text{H}^{0}(\PP^{n}_{\CC},\OO(-k-n-1))$.

\begin{notation} \label{def:7}
For every integer $k$, we introduce the relative metrized complex
\begin{equation}\label{eq:5}
  \ov{\xi_{n}}(k)=(\ov{\pi}^{\FS}, \ov{\mathcal{O}(k)},
  \ov{\Rd\pi_{\ast}\mathcal{O}(k)}).
\end{equation}
If $X$ is a smooth complex variety, we will also denote by
$\ov{\xi}_{n}(k)$ its
pull-back to $\PP^{n}_{X}$.
  Let $\overline {\mathcal{F}}$ be a metrized coherent sheaf on
  $X$. Then we define $\overline {\mathcal{F}}(k)$ and $\ov {\Rd \pi
    _{\ast}{\mathcal{F}}(k)}$ by the equality
  \begin{displaymath}
    \ov \xi_{n}(k)\otimes \ov {\mathcal{F}}=(\ov \pi ^{\FS}, \overline
    {\mathcal{F}}(k), \ov {\Rd \pi
    _{\ast}{\mathcal{F}}(k)}).
  \end{displaymath}
\end{notation}

\begin{definition}\label{def:2}
  Let $X$ be a complex smooth variety and $\pi\colon \PP^{n}_{X}\to X$ the
  projection. An \emph{analytic torsion class} for the relative
  hermitian complex $\ov{\xi}=(\ov{\pi},
  \ov{\mathcal{F}},\ov{\Rd\pi_{\ast}\mathcal{F}})$ is a class
  $\widetilde \eta\in \bigoplus_{p} \widetilde
  {\mathcal{D}}^{2p-1}(X,p)$ such that
  \begin{equation}\label{eq:1}
    \dd_{\mathcal{D}} \widetilde
    \eta=\ch(\ov{\Rd\pi_{\ast}\mathcal{F}})-\ov
    \pi_{\ttwist}
    [\ch(\ov{\mathcal{F}})].
  \end{equation}
\end{definition}

The existence of this class is guaranteed by the
Grothendieck-Riemann-Roch theorem, which implies that the two currents
at the right hand side of equation \eqref{eq:1} are
cohomologous. Since the map $\pi $ is smooth, the analytic torsion
class is the class of a smooth form.

\begin{definition}\label{def:3}
  Let $n$ be a non-negative integer.  A \emph{theory of analytic
    torsion classes for projective spaces of dimension $n$} is an
  assignment that, to each relative metrized complex
\begin{math}
  \overline{\xi}=(\ov{\pi}\colon \PP^n_X\rightarrow
  X,\ov{\mathcal{F}},\ov{\Rd\pi_{\ast}\mathcal{F}})
\end{math}
of relative dimension $n$, assigns a class of differential forms
\begin{displaymath}
  T(\overline \xi)\in \bigoplus_{p} \widetilde
  {\mathcal{D}}^{2p-1}(X,p),
\end{displaymath}
satisfying the following properties.
\begin{enumerate}
\item \label{item:1} (Differential equation)
  \begin{math}
      \dd_{\mathcal{D}} T(\overline
      {\xi})=\ch(\ov{\Rd\pi_{\ast}\mathcal{F}})-\ov \pi_{\ttwist}
      [\ch(\ov{\mathcal{F}})].
    \end{math}
  \item \label{item:2}(Functoriality) Given a morphism
    $f\colon Y\longrightarrow X$, we have
    \begin{math}
    	T(f^{\ast}\ov{\xi})=f^{\ast} T(\ov{\xi}).
    \end{math}
  \item \label{item:3}(Additivity and normalization) If
    $\overline{\xi}_1$ and $\overline{\xi}_2$ are relative metrized
    complexes on $X$, then
    $T(\overline{\xi}_1\oplus\overline{\xi}_2)
    =T(\overline{\xi}_1)+T(\overline{\xi}_2).$
  \item \label{item:4}(Projection formula) For any
    hermitian vector bundle $\ov{G}$ on $X$, and any integer $k\in[-n,0]$, we have
    \begin{math}
      T(\ov{\xi}_{n}(k)\otimes\ov{G})=
      T(\ov{\xi}_{n}(k))\bullet\ch(\ov{G}).
    \end{math}
  \end{enumerate}

  A \emph{theory of analytic torsion classes for projective spaces} is
  an assignment as
  before, for all non-negative integers $n$.
\end{definition}

\begin{definition}\label{def:char_numbers}
  Let $T$ be a theory of analytic torsion classes for projective spaces
  of dimension $n$. Fix as base space the point $\Spec \CC$. The
  \emph{characteristic numbers} of $T$ are
  \begin{equation}\label{eq:3}
    t_{n,k}(T):=T(\overline{\xi}_{n}(k))\in \widetilde
    {\mathcal{D}}^{1}(\Spec \CC, 1)=\RR, \
    k\in\mathbb{Z}.
  \end{equation}
  The numbers $t_{n,k}(T)$, $-n\le k \le 0$ will be
  called the \emph{main characteristic numbers} of $T$.
\end{definition}

The central result of this section is the following classification theorem.

\begin{theorem}\label{thm:1}
Let $n$ be a non-negative integer and let
$\mathfrak{t}=(t_{n,k})_{k=-n,\dots,0}$ be a family of arbitrary real
numbers. Then there exists a unique theory $T_{\mathfrak{t}}$ of
analytic torsion classes for projective spaces of  dimension $n$, such that
$t_{n,k}(T_{\mathfrak{t}})=t_{n,k}$.
\end{theorem}

Before proving Theorem \ref{thm:1}, we show some
consequences of the definition of the analytic torsion classes.
First we state some anomaly formulas
that determine the dependence of the analytic torsion classes with
respect to different choices of metrics.
\begin{proposition} \label{prop:1}
  Let $T$ be a theory of analytic torsion classes for projective
  spaces of dimension $n$. Let
  \begin{math}
    \overline{\xi}= (\ov{\pi}\colon \PP^n_X\rightarrow
    X,\ov{\mathcal{F}}, \ov{\Rd\pi_{\ast}\mathcal{F}})
  \end{math}
  be a relative metrized complex.
  \begin{enumerate}
  \item \label{item:6}
    If $\ov{\mathcal{F}}'$ is another choice
    of metric on $\mathcal{F}$ and $\overline{\xi}_1=(\ov{\pi}\colon
    \PP^n_X\rightarrow
    X,\ov{\mathcal{F}}', \ov{\Rd\pi_{\ast}\mathcal{F}})$, then
    \begin{displaymath}
      T(\overline{\xi}_1)=T(\overline{\xi})+\ov \pi
      _{\ttwist}[\cht(\ov{\mathcal{F}}',\ov{\mathcal{F}})].
    \end{displaymath}
  \item \label{item:7}
    If $\ov{\pi}'$ is another hermitian structure on $\pi$ and
    $\overline{\xi}_{2}= (\ov{\pi}'\colon \PP^n_X\rightarrow
    X,\ov{\mathcal{F}}, \ov{\Rd\pi_{\ast}\mathcal{F}})$, then
    \begin{equation}\label{eq:14}
      T(\overline{\xi}_{2})=T(\overline{\xi})+\ov \pi'
      _{\ttwist}[\ch(\overline {\mathcal{F}})\bullet
      \widetilde{\Td}_{m}(\ov{\pi}', \ov{\pi})].
    \end{equation}
  \item \label{item:8}
    If $\ov{\Rd\pi_{\ast}\mathcal{F}}'$ is another
    choice of metric on $\Rd\pi_{\ast}\mathcal{F}$, and $\overline
    {\xi}_{3}= (\ov{\pi}\colon \PP^n_X\rightarrow
    X,\ov{\mathcal{F}}, \ov{\Rd\pi_{\ast}\mathcal{F}}')$, then
    \begin{displaymath}
      T(\overline{\xi}_{3})= T(\overline{\xi})-
      \cht(\ov{\Rd\pi_{\ast}\mathcal{F}}',
      \ov{\Rd\pi_{\ast}\mathcal{F}}).
    \end{displaymath}
  \end{enumerate}
\end{proposition}
\begin{proof}
The proof is the same as the proof of Proposition
\ref{prop:1bis}.
\end{proof}

Next we state the behavior of analytic torsion classes for projective
spaces with respect to
distinguished triangles.

\begin{proposition} \label{prop:2}
  Let $T$ be a theory of analytic torsion classes for projective spaces of
  dimension $n$. Let $X$ be a smooth complex variety and
  $\pi\colon \PP^{n}_{X}\to X$ the projection. Consider distinguished
  triangles in $\oDb(\PP^{n}_{X})$ and $\oDb(X)$ respectively:
\begin{displaymath}
	(\ov{\tau}):\ 
        \ov{\mathcal{F}}_{2}\to\ov{\mathcal{F}}_{1}
        \to\ov{\mathcal{F}}_{0}\to\ov{\mathcal{F}}_{2}[1]\ \text{ and
        }\ 
	(\ov{\Rd\pi_{\ast}\tau}):\ 
        \ov{\Rd\pi_{\ast}\mathcal{F}}_{2}\to
        \ov{\Rd\pi_{\ast}\mathcal{F}}_{1}
	\to\ov{\Rd\pi_{\ast}\mathcal{F}}_{0}\to
        \ov{\Rd\pi_{\ast}\mathcal{F}}_{2}[1],
\end{displaymath}
and define relative metrized complexes
$\ov{\xi}_{i}=(\ov{\pi},\ov{\mathcal{F}}_i,
\ov{\Rd\pi_{\ast}\mathcal{F}}_i),$ $i=0,1,2$.
Then
\begin{displaymath}
	\sum_{j}(-1)^{j}T(\ov{\xi}_{j})
	=\widetilde{\ch}(\ov{\Rd\pi_{\ast}\tau})
	-\ov \pi_{\ttwist}(\widetilde{\ch}(\ov{\tau })).
\end{displaymath}
\end{proposition}
\begin{proof}
 The proof is similar to that of \ref{prop:2bis}.
\end{proof}

In view of this proposition, we see that the additivity axiom is
equivalent to the apparently stronger statement of the next corollary.

\begin{corollary} \label{cor:2}
  With the assumptions of Proposition \ref{prop:2}, if $\ov \tau $ and
  $\ov{\Rd\pi_{\ast}\tau}$ are tightly distinguished, then
  \begin{math}
    T(\ov \xi_{1})=T(\ov \xi _{0})+T(\ov \xi _{2}).
  \end{math}
\end{corollary}

\begin{corollary}
  Let $\overline{\xi}= (\ov{\pi},\ov{\mathcal{F}},
  \ov{\Rd\pi_{\ast}\mathcal{F}})$ be a relative metrized complex and
  let $\overline{\xi}[i]= (\ov{\pi},\ov{\mathcal{F}}[i],
  \ov{\Rd\pi_{\ast}\mathcal{F}}[i])$ be the shifted relative metrized
  complex. Then
  \begin{math}
    T(\ov{\xi})=(-1)^{i}T(\ov{\xi}[i]).
  \end{math}
\end{corollary}
\begin{proof}
  It is enough to treat the case $i=1$. We consider the tightly
  distinguished triangle
  \begin{displaymath}
    \ov{\mathcal{F}}\dra \ocone(\Id_{\ov{\mathcal{F}}})\dra \ov{\mathcal{F}}[1]\dra
  \end{displaymath}
  and the analogous triangle for direct images. Since
  $\ocone(\Id_{\ov{\mathcal{F}}})$ and
  $\ocone(\Id_{\ov{\Rd \pi _{\ast}\mathcal{F}}})$ are meager,
  we have, by the anomaly formulas and the additivity axiom,
  \begin{displaymath}
    T(\ov \pi , \ocone(\Id_{\ov{\mathcal{F}}}), \ocone(\Id_{\ov{\Rd
        \pi _{\ast}\mathcal{F}}}))=
    T(\ov \pi ,\ov 0,\ov 0)=0.
  \end{displaymath}
  Hence, the result follows from Corollary \ref{cor:2}.
\end{proof}

Next we rewrite Proposition \ref{prop:2} in the language of
complexes of metrized coherent sheaves.
Let
\begin{displaymath}
  \overline{\varepsilon }:\quad
  0\to \overline{\mathcal{F}}_{m}\to \dots \to \overline{\mathcal{F}}_{l}\to 0
\end{displaymath}
be a bounded complex of metrized coherent sheaves on
$\PP^{n}_{X}$ and assume that
hermitian structures on the complexes
$\Rd \pi _{\ast}\mathcal{F}_{j}$, $j=l,\dots,m$ are chosen. Let $[\ov \varepsilon
],
[\overline{\Rd \pi _{\ast}\varepsilon }]\in \Ob \oDb(\PP^{n}_{X})$ be
the associated objects as in
\cite[Def. 3.37, Def. 3.39]{BurgosFreixasLitcanu:HerStruc}.

\begin{corollary} \label{cor:1}
  With the above hypothesis,
  \begin{displaymath}
    T(\ov \pi ,[\ov \varepsilon ],[\overline{\Rd \pi
      _{\ast}\varepsilon }]) =\sum_{j=l}^{m}(-1)^{j}T(\ov \pi ,\ov
    {\mathcal{F}}_{j},\ov{\Rd \pi _{\ast}\mathcal{F}_{j}}).
  \end{displaymath}
Moreover, if $\varepsilon $ is acyclic, then
\begin{math}
  T(\ov \pi ,[\ov \varepsilon ],[\overline{\Rd \pi
      _{\ast}\varepsilon }]) =\widetilde{\ch}(\overline{\Rd \pi
      _{\ast}\varepsilon }) -\ov \pi _{\ttwist}[\widetilde{\ch}
    (\ov \varepsilon )].
\end{math}
\end{corollary}

Finally, we show that the projection formula holds in greater generality:
\begin{proposition}\label{prop:3}
  Let $T$ be a theory of analytic torsion classes for projective
  spaces of dimension $n$. Let $X$ be a smooth complex variety, let
  $\ov \xi=(\ov \pi ,\ov{\mathcal{F}},\ov{\Rd \pi _{\ast}
    \mathcal{F}})$ be a relative metrized complex and let
  $\ov{\mathcal{G}}$ be an object in $\oDb(X)$. Then
\begin{equation} \label{eq:20}
  T(\ov{\xi}\otimes\ov{\mathcal{G}})=T(\ov{\xi})\bullet\ch(\ov{\mathcal{G}}).
\end{equation}
\end{proposition}
\begin{proof}
  By the anomaly formulas, if equation \eqref{eq:20} holds for a
  particular choice of hermitian structures on $\pi $, $\mathcal{F}$
  and $\Rd \pi _{\ast} \mathcal{F}$ then it holds for any other
  choice. Moreover, if we are in the situation of Proposition
  \ref{prop:2} and equation \eqref{eq:20} holds for two of $\ov \xi
  _{0}$, $\ov \xi _{1}$, $\ov \xi _{2}$, then it holds for the third.
  Using that the objects of the form $\mathcal{H}(k)$, where
  $\mathcal{H}$ is a coherent sheaf on $X$ and $k=-n,\dots,0$,
  constitute a
  generating class of $\Db(\PP^{n}_{X})$, we are reduced to prove that
  \begin{displaymath}
    T(\ov{\xi}_{n}(k)\otimes\ov{\mathcal{G}})=
    T(\ov{\xi}_{n}(k))\bullet\ch(\ov{\mathcal{G}}).
  \end{displaymath}
  for $k=-n,\dots,0$.
 Now, if
  \begin{displaymath}
    \ov{\mathcal{G}}_{2}\dra \ov{\mathcal{G}}_{1} \dra \ov{\mathcal{G}}_{0} \dra
  \end{displaymath}
  is a distinguished triangle in $\oDb(X)$ and equation \eqref{eq:20}
  is satisfied for two of $\ov{\mathcal{G}}_{2}$,
  $\ov{\mathcal{G}}_{1}$, $\ov{\mathcal{G}}_{0}$, then it is satisfied
  also by the
  third. Therefore, since the complexes of vector bundles concentrated in
  a single degree constitute a generating class of $\Db(X)$, the
  projection formula axiom implies the proposition.
\end{proof}

\begin{proof}[Proof of Theorem \ref{thm:1} ]
  To begin with, we prove the uniqueness assertion.
 Assume a theory of
  analytic torsion classes $T$, with main characteristic numbers
  $t_{n,k}$, $-n\le k\le 0$, exists. Then, the anomaly formulas (Proposition \ref{prop:1}) imply that, if
$T(\ov{\pi}, \ov{\mathcal{F}},
\ov{\Rd\pi_{\ast}\mathcal{F}})$ is known for a particular choice of hermitian
structures on $\pi $, $\mathcal{F}$ and $\Rd \pi _{\ast} \mathcal{F}$
then the value of $T(\ov{\pi}', \ov{\mathcal{F}}',
\ov{\Rd\pi_{\ast}\mathcal{F}}')$ for any other choice of hermitian
structures is fixed.
By Proposition \ref{prop:2}, if we know the value of $T(\ov{\pi},
\ov{\mathcal{F}},
\ov{\Rd\pi_{\ast}\mathcal{F}})$, for $\mathcal{F}$ in a generating
class, then $T$ is determined.
By the projection formula (Proposition \ref{prop:3}), the
characteristic numbers determine the values of $T(\ov \xi(k)\otimes
\mathcal{G})$, $k=-n,\dots,0$. Finally, since by Corollary \ref{cor:6},
the objects of the form $\mathcal{G}(k)$, $k=-n,\dots,0$ form a
generating class, we
deduce that the characteristic numbers determine the theory $T$.
Thus, if it exists,
the theory $T_{\mathfrak{t}}$ is unique.

In particular, from the above discussion we see that
  the main characteristic numbers determine all the characteristic
  numbers. We now derive an explicit inductive formula for them.
 Consider the metrized Koszul resolution
\begin{equation}\label{eq:38}
  \overline{K}: 0\to
  \Lambda ^{n+1}\overline V^{\vee}(-n-1)\to \dots \to
  \Lambda ^{1} \overline V^{\vee}(-1)\to \ov{\mathcal{O}}_{\PP^{n}_{\CC}}\to 0,
\end{equation}
where $\mathcal{O}(k)$, for $k\not =0$, has the Fubini-Study metric
and $\ov{\mathcal{O}}_{\PP^{n}_{\CC}}$ has the trivial metric.  We
will denote by $\overline{K}(k)$ the above exact sequence twisted by
$\ov{\mathcal{O}(k)}$, $k\in\mathbb{Z}$, again with the Fubini-Study
metric. Recall the definition of the relative metrized complexes
$\ov{\xi}_{n}(k)$ \eqref{eq:5}. In particular, for every $k$, we have
fixed natural hermitian structures on the objects
$\Rd\pi_{\ast}\OO(k-j)$. According to \cite[Def. 3.37,
Def. 3.39]{BurgosFreixasLitcanu:HerStruc}, we may consider the classes
$[\ov{K}(k)]$ and
$[\ov{\Rd\pi_{\ast}K(k)}]$ in $\oDb(\PP^{n}_{\CC})$ and
$\oDb(\Spec\CC)$, respectively. By Corollary \ref{cor:1}, for each
$k\in \mathbb{Z}$ we find
\begin{displaymath}
    \sum_{j=0}^{n+1}(-1)^{j}T(\ov{\xi}_{n}(k-j)
    \otimes\Lambda^{j}\ov{V}^{\vee})=
    \widetilde{\ch}(\ov{\Rd\pi_{\ast}K(k)})
    -\ov\pi_{\ttwist} ^{\FS}[\widetilde{\ch}(\ov{K}(k))].
\end{displaymath}
Because $\Lambda^{j}\ov{V}^{\vee}$ is isometric to
$\CC^{\binom{n+1}{j}}$ with the trivial metric, the additivity axiom
for the theory $T$ and the definition of the characteristic numbers
$t_{n,k-j}$ provide
\begin{displaymath}
  T(\ov{\xi}_{n}(k-j)\otimes\Lambda^{j}\ov{V}^{\vee})=
  t_{n,k-j}\binom{n+1}{j}.
\end{displaymath}
Therefore we derive
\begin{equation}\label{eq:79}
  \sum_{j=0}^{n+1}(-1)^{j}\binom{n+1}{j}t_{n,k-j}=
  \widetilde{\ch}(\ov{\Rd\pi_{\ast}K(k)})
  -\ov \pi^{\FS}_{\ttwist}[\widetilde{\ch}(\ov{K}(k))].
\end{equation}
This equation gives us an inductive formula for all the characteristic
numbers $t_{n,k}$ once we have fixed $n+1$ consecutive characteristic
numbers and, in particular, once we have fixed the main characteristic
numbers.

To prove the existence, we follow the proof of the uniqueness to obtain a
formula for $T(\overline \xi)$. We start with the main characteristic
numbers $\mathfrak{t}=(t_{n,k})_{-n\le k \le 0}$. We define the characteristic
numbers $t_{n,k}$ for $k\in \mathbb{Z}$ inductively using equation
(\ref{eq:79}).

We will need the following results.
\begin{lemma}\label{lemm:6}
  Let
  \begin{displaymath}
    \overline \eta:
    0\to \overline {\mathcal{F}}_{2}\to
    \overline {\mathcal{F}}_{1}\to
    \overline {\mathcal{F}}_{0}\to
    0
  \end{displaymath}
  be a short exact sequence of metrized coherent sheaves on $X$. Let
  $k$ be an integer, and $\overline {\mathcal{F}}(k)$ and $\ov {\Rd \pi
    _{\ast}{\mathcal{F}}(k)}$ be as in Notation \ref{def:7}. Thus we
  have an exact sequence $\overline
  \eta(k)$ of metrized coherent sheaves on $\PP^{n}_{X}$ and a
  distinguished triangle $\overline {\Rd \pi _{\ast}\eta(k)}$. Then
  \begin{equation}
    \label{eq:44}
    \cht(\overline {\Rd \pi _{\ast}\eta(k)})=\ov \pi^{\FS}
    _{\ttwist}(\cht(\overline \eta(k))).
  \end{equation}
\end{lemma}
\begin{proof}
  By
  the Riemann-Roch theorem for the map $\PP^{n}_{\mathbb{C}}\to
  \Spec\mathbb{C}$ we have
  \begin{equation}
    \label{eq:42}
    \ch(\overline{\Rd\pi _{\ast}\mathcal{O}(k)})
    =\pi _{\ast}(\ch(\overline{\mathcal{O}(k)})\Td(\ov \pi^{\FS} )).
  \end{equation}
  Hence, by the properties of Bott-Chern classes and the choice of metrics
\begin{align*}
  \cht(\overline {\Rd\pi _{\ast}\eta(k)})&=
  \cht(\overline \eta)\bullet \ch(\overline{\Rd\pi_{\ast}
  \mathcal{O}(k)})\\
  &=\cht(\overline \eta)\bullet \pi
    _{\ast}(\ch(\ov{\mathcal{O}(k)})\Td(\ov \pi ^{\FS}))\\
    &=\pi
    _{\ast}\left(\cht(\overline \eta(k))\bullet \Td(\ov \pi
      ^{\FS})\right)\\
    &=\ov \pi^{\FS}
    _{\ttwist}(\cht(\overline \eta(k))).\qedhere
\end{align*}
\end{proof}

\begin{lemma}\label{lemm:5}
  Let
  \begin{equation}
    \label{eq:36}
    \overline \mu:
    0\to \overline {\mathcal{M}}_{m}(-m-d)\to\dots\to \overline
    {\mathcal{M}}_{l}(-l-d)
    \to 0
  \end{equation}
be an exact sequence of metrized coherent sheaves on $\PP^{n}_{X}$, where, for
each $i=l,\dots, m$, $\ov {\mathcal{M}}_{i}$ is a
metrized coherent sheaf on $X$, and $\ov {\mathcal{M}}_{i}(k)$ is as
in Notation \ref{def:7}. On $\Rd \pi _{\ast}{\mathcal{M}}_{i}(k)$ we
consider the hermitian structures given also by Notation
\ref{def:7}. Then
\begin{equation}
  \label{eq:37}
  \sum_{i=l}^{m} (-1)^{i}t_{n,-d-i}\ch(\overline {\mathcal{M}}_{i})=
  \cht(\overline{\Rd \pi
  _{\ast}\mu} )-  \ov \pi^{\FS} _{\ttwist}(\cht(\overline \mu )).
\end{equation}
\end{lemma}
\begin{proof}
  We consider a commutative diagram of exact
  sequences
  \begin{displaymath}
    \xymatrix{
      & & &0\ar[d] & &0\ar[d] & \\
      &\ov \mu' &0\ar[r] &\ov{\mathcal{M}}_{m}'(-m-d)\ar[r]\ar[d] &\dots\ar[r]
      &\ov{\mathcal{M}}_{l}'(-l-d)\ar[r]\ar[d] &0
      \\
      &\ov \mu &0\ar[r] &\ov{\mathcal{M}}_{m}(-m-d)\ar[r]\ar[d] &\dots\ar[r]
      &\ov{\mathcal{M}}_{l}(-l-d)\ar[r]\ar[d] &0
      \\
      &\ov \mu'' &0\ar[r] &\ov{\mathcal{M}}_{m}''(-m-d)\ar[r]\ar[d] &\dots\ar[r]
      &\ov{\mathcal{M}}_{l}''(-l-d)\ar[r]\ar[d] &0
      \\
      & & &0 & &0 & \\
      & & &\ov \xi_{m} &\dots &\ov \xi_{l}. &
    }
  \end{displaymath}
  \emph{Claim.} If equation \eqref{eq:37} holds for two of $\ov \mu
  $, $\ov \mu
  '$ and $\ov \mu ''$, then it holds for the third.\\
  \emph{Proof of the claim.} On the one hand we have
  \begin{displaymath}
    \sum_{i=l}^{m} (-1)^{i}t_{n,-d-i}\left(\ch(\overline
      {\mathcal{M}}'_{i})-\ch(\overline {\mathcal{M}}_{i})
      +\ch(\overline {\mathcal{M}}''_{i})\right)=
    \sum_{i=l}^{m} (-1)^{i}t_{n,-d-i}\dd_{\mathcal{D}}\cht(\ov \xi _{i}).
  \end{displaymath}
  But, if $t\in \mathcal{D}^{1}(\Spec
  \mathbb{C},1)=\mathbb{R}$ is a real number, in the group
  $\bigoplus_{p}\widetilde{\mathcal{D}}^{2p-1}(X,p)$, we have
  $$ t \dd_{\mathcal{D}}\cht(\ov \xi _{i})=
  -\dd_{\mathcal{D}}(t\bullet \cht(\ov \xi _{i})
  )=0.$$
On the other hand, by Lemma \ref{lemm:6}
\begin{displaymath}
  \cht(\overline{\Rd \pi
  _{\ast}\mu'} )- \cht(\overline{\Rd \pi
  _{\ast}\mu} )+ \cht(\overline{\Rd \pi
  _{\ast}\mu''} )=
\pi _{\ttwist}^{\FS}(\cht(\overline \mu ')) -\pi
_{\ttwist}^{\FS}(\cht(\overline \mu ))+\pi
_{\ttwist}^{\FS}(\cht(\overline \mu'' )).
\end{displaymath}

  The proof of the lemma is done by induction on the length
  $r=m-l$ of the complex. If $r\le n$ then $\mu(d+l)$
  has the same shape as the canonical resolution of the zero
  sheaf. By the uniqueness of the canonical resolution,
  we have ${\mathcal{M}}_{i}=0$,  for $i=l,\dots,m$. Using
  the above claim when ${\mathcal{M}}_{i}=0$ has a non-trivial metric,
  we obtain the lemma for~$r\le n$.

Assume now that $r>n$. Let $K$ be
the Koszul exact sequence (\ref{eq:38}). Then $K(1)\otimes
{\mathcal{M}}_{l}$
is the canonical resolution of the regular coherent sheaf
${\mathcal{M}}_{l}(1)$. By Theorem \ref{thm:14}~\ref{item:9} there is a
surjection of exact sequences $\mu \to K(-l-d)\otimes {\mathcal{M}}_{l}$ whose
kernel is an exact sequence
 \begin{equation*}
    \mu':
    0\to {\mathcal{M}}'_{m}(-m-d)\to\dots\to {\mathcal{M}}'_{l+1}(-d-l-1)
    \to 0.
  \end{equation*}
We consider on $K$ the metrics of \eqref{eq:38}, for  $i=l+1,\dots,m$,
we choose arbitrary metrics on ${\mathcal{M}}'_{i}$ and denote by $\overline \mu
'$ the corresponding exact sequence of metrized coherent sheaves.

By induction hypothesis, $\ov \mu '$ satisfies equation
\eqref{eq:37}. Moreover, since the characteristic
numbers $t_{n,k}$ for $k\not \in [0,n]$ are defined using equation
\eqref{eq:79}, the exact sequence $\ov K(-l-d)\otimes \ov
{\mathcal{M}}_{l}$ also satisfies equation \eqref{eq:37}. Hence the
lemma follows from the previous claim.
\end{proof}

We now treat the case of complexes concentrated in a single
degree. Let $\ov{\mathcal{F}}$ be a coherent sheaf on $\PP^{n}_{X}$
with a hermitian structure and let $\ov{\Rd \pi _{\ast} \mathcal{F}}$
be a choice of a hermitian structure on the direct image complex. Write
$\overline{\xi}=(\ov \pi^{\FS},\overline{\mathcal{F}}, \ov
{\Rd\pi_{\ast}\mathcal{F}})$
for the corresponding
relative metrized complex.

Choose an integer $d$ such
that $\mathcal{F}(d)$ is regular. Then we have the
resolution $ \gamma _{d}(\mathcal{F})$ of Corollary \ref{cor:4}. More
generally, let $\mu $ be an exact sequence of the form
\begin{displaymath}
  0\to \mathcal{S}_{m}(-d-m)\to \dots \to \mathcal{S}_{1}(-d-1)\to
  \mathcal{S}_{0}(-d)\to \mathcal{F}\to 0,
\end{displaymath}
where the $\mathcal{S}_{i}$, $i=0,\dots,m$ are coherent sheaves on
$X$. Assume that we have chosen hermitian structures on the sheaves
$\mathcal{S}_{i}$. Using Notation \ref{def:7} and \cite[Def. 3.37,
Def. 3.39]{BurgosFreixasLitcanu:HerStruc} we have objects $[\ov \mu ]$ in
$\KA(\PP^{n}_{X})$ and $[\ov{\Rd
\pi _{\ast}\mu }]$ in $\KA(X)$.
Then we write
\begin{equation}\label{eq:33}
  T_{\mathfrak{t},\ov \mu
  }(\overline{\xi})=\sum_{j=0}^{m}(-1)^{j}t_{n,j-d}\ch(\overline{\mathcal{S}}_{j})
  -\cht(\ov{\Rd \pi _{\ast}\mu })
  +\ov \pi ^{\FS}_{\ttwist}(\cht{(\ov \mu )})
\end{equation}

\begin{lemma}\label{lemm:2}   Given any choice of metrics
  on the sheaves $\mathcal{G}_{i}$, (respectively $
  {\mathcal{G}}_{i}'$) $i=0,\dots,n$, that
  appear in the resolution $\gamma _{d}(\mathcal{F})$
  (respectively $\gamma _{d+1}(\mathcal{F})$), denote by $\ov \gamma
  _{d}$ and $\ov \gamma_{d+1} $ the corresponding exact sequences of
  metrized coherent sheaves. Then
  \begin{displaymath}
    T_{\mathfrak{t},\ov \gamma
      _{d+1}}(\overline{\xi})=T_{\mathfrak{t},\ov \gamma
      _d}(\overline{\xi}).
  \end{displaymath}
  In particular, $T_{\mathfrak{t},\ov \gamma _d}(\overline{\xi})$ does not depend
  on the choice of metrics on the sheaves $\mathcal{G}_{i}$.
\end{lemma}
\begin{proof}
By Theorem \ref{thm:14}~\ref{item:11}, there is an exact sequence
\begin{equation}
  \label{eq:35}
  \overline \mu:
  0\to \overline {\mathcal{S}}_{n+k}(-n-k-d-1)\to\dots\to \overline {\mathcal{S}}_{0}(-d-1)\to
  \overline {\mathcal{F}} \to 0,
\end{equation}
and a surjection of exact sequences $f\colon\overline {\mu}\to
\overline \gamma _{d}$ extending the identity on $\overline {\mathcal{F}}$.
Here
$\overline {\mathcal{S}}_{i}$, $i=0,\dots,n+k$ are coherent sheaves on
$X$ with hermitian structures.

By Theorem \ref{thm:14}~\ref{item:9} there is a surjection of
exact sequences $\overline \mu\longrightarrow \overline
\gamma _{d+1}$ extending the identity on $\overline {\mathcal{F}}$, whose kernel
is an exact sequence
\begin{equation}\label{eq:55}
  \overline \varepsilon:
  0\to \overline {\mathcal{M}}_{n+k}(-n-k-d-1)\to\dots\to \overline
  {\mathcal{M}}_{0}(-d-1)
  \to 0,
\end{equation}
where $\overline {\mathcal{M}}_{i}$, $i=0,\dots,n+k$ are coherent sheaves
on $X$, and we have chosen arbitrarily an hermitian structure on them.
Denote by $\overline \eta_{i}$ the rows of the exact sequence
\begin{displaymath}
  0\to \overline \varepsilon \to
  \overline \mu\to
  \overline \gamma _{d+1}\to 0.
\end{displaymath}
Observe that $\overline \eta_{i}=\overline \eta'_{i}(-i-d-1)$ for some
short exact sequence $\overline \eta_{i}'$ on $X$.
When $j\ge n$ we denote $\overline{\mathcal{G}}'_{j}=\overline 0$.
Then, we have
\begin{multline}
  \label{eq:45}
  \sum_{j=0}^{n+k}(-1)^{j}t_{n,j-d-1}
  \left(\ch(\overline{\mathcal{G}}'_{j})-\ch(\overline{\mathcal{S}}_{j})
    +\ch(\overline {\mathcal{M}}_{j})\right)=\\
  \sum_{j=0}^{n+k}(-1)^{j}t_{n,j-d-1}\dd_{\mathcal{D}}\cht(\overline \eta_{i}')
=0.
\end{multline}
By \cite[Prop.~3.41]{BurgosFreixasLitcanu:HerStruc}, we have
\begin{align}
  \label{eq:46}
  \cht(\overline {\Rd \pi _{\ast}\gamma_{d+1}})-\cht(\overline {\Rd \pi
    _{\ast}\mu})
  +\cht(\overline {\Rd \pi
    _{\ast}\varepsilon })&=\sum_{j=0}^{n+k}(-1)^{j}\cht(\overline
  {\Rd \pi _{\ast} \eta_{j}}),\\
  \label{eq:47}
  \cht(\overline{\gamma }_{d+1})-
  \cht(\overline \mu)+
  \cht(\overline \varepsilon)&=
  \sum_{j=0}^{n+k}(-1)^{j}\cht(\overline
  \eta_{j}).
\end{align}
Combining equations (\ref{eq:45}), (\ref{eq:46})
and (\ref{eq:47}) and lemmas \ref{lemm:6} and \ref{lemm:5} we obtain
\begin{equation}
  \label{eq:48}
  T_{\mathfrak{t},\ov \mu }(\ov \xi )= T_{\mathfrak{t},\ov \gamma _{d+1}}(\ov \xi ).
\end{equation}

We consider now $\cone(\mu,\gamma _{d})$. On it we put the obvious
hermitian structure induced by $\ov{\mu}$ and $\ov{\gamma_{d}}$,
$\ov{\cone(\mu ,\gamma _{d})}$. On $\Rd\pi _{\ast}\cone(\mu,\gamma
_{d})$, we put the obvious family of hermitian metrics induced by
$\ov{\Rd\pi_{\ast}\mu}$ and $\ov{\Rd\pi_{\ast}\gamma_{d}}$, and denote
it as $\ov{\Rd\pi _{\ast}\cone(\mu,\gamma _{d})}$. By
\cite[Cor.~3.42]{BurgosFreixasLitcanu:HerStruc} we have
\begin{align}
  \label{eq:7}
  \cht(\ov{\cone(\mu,\gamma _{d})})&=\cht(\ov \gamma _{d})-\cht(\ov
  \mu ),\\
  \label{eq:8}
  \cht(\ov {\Rd \pi _{\ast}\cone(\mu,\gamma _{d})})&=\cht(\ov{\Rd\pi
    _{\ast} \gamma _{d}})-\cht(\ov
  {\Rd \pi _{\ast}\mu }).
\end{align}
Observe that $\ov{\cone(\mu,\gamma _{d})}^{i}= \overline
  {\mathcal{S}}_{-i-1}(i-d) \oplus \overline
{\mathcal{G}}_{-i}(i-d) $. Combining Lemma \ref{lemm:5} for
$\ov{\cone(\mu,\gamma _{d})}$ with equations \eqref{eq:7} and
\eqref{eq:8}, we obtain
\begin{equation}
  \label{eq:10}
  T_{\mathfrak{t},\ov \mu }(\ov \xi )=T_{\mathfrak{t},\ov \gamma _{d}}(\ov \xi ),
\end{equation}
Together with equation \eqref{eq:48} this proves the lemma.
\end{proof}

Now we are in position to prove the existence of
$T_{\mathfrak{t}}$. Let $n$ and
$\mathfrak{t}$ be as in Theorem~\ref{thm:1}. We define the numbers
$t_{n,k}$, for $k< 0$
and $k> n$ by equation \eqref{eq:79}.
Let $\ov \xi
  =(\ov \pi^{\FS}, \ov {\mathcal{F}},\ov {\Rd \pi_{ \ast}
    \mathcal{F}})$ be a relative metrized complex. We construct
  $T_{\mathfrak{t}}(\ov \xi)$ by induction on the length of the
  cohomology of $\mathcal{F}$.
  If it has at most a single non zero coherent sheaf
  $\mathcal{H}$ sitting at degree $j$, then $\ov {\mathcal{F}}$ and
  $\ov {\Rd \pi_{ \ast} \mathcal{F}}$ determine hermitian structures
  on $\mathcal{H}[-j]$ and $\Rd \pi _{\ast}\mathcal{H}[-j]$
  respectively. We choose an integer $d$ such that $\mathcal{H}(d)$ is
  regular and we write
  \begin{equation}\label{eq:29}
    T_{\mathfrak{t}}(\ov \xi)=
    (-1)^{j}T_{\mathfrak{t},\ov \gamma _{d}(\mathcal{H})}(\ov
    \pi^{\FS}, \ov {\mathcal{H}},\ov
    {\Rd \pi_{ \ast} \mathcal{H}}).
  \end{equation}
  By Lemma \ref{lemm:2}, this does not depend on the choice of $d$ nor
  on the choice of metrics on~$\ov \gamma _{d}(\mathcal{H})$.

  Assume
  that we have defined the analytic torsion classes for all
  complexes whose cohomology has length less than $l$ and that the
  cohomology of
  $\mathcal{F}$ has length $l$. Let $\mathcal{H}$ be the
  highest cohomology sheaf of $\mathcal{F}$, say of degree
  $j$. Choose
  auxiliary hermitian structures on $\mathcal{H}[-j]$ and $\Rd \pi
  _{\ast}\mathcal{H}[-j]$. There is a unique natural map
  $\mathcal{H}[-j]\dra
  \mathcal{F}$. Then we define
  \begin{multline} \label{eq:30}
    T_{\mathfrak{t}}(\ov \xi)=
    T_{\mathfrak{t}}(\ov \pi ^{\FS},\ov{ \mathcal{H}[-j]},\ov
    {\Rd \pi_{ \ast} \mathcal{H}[-j]})\\+
    T_{\mathfrak{t}}(\ov \pi ^{\FS},\ocone(\ov{ \mathcal{H}[-j]},\ov{
      \mathcal{F}}),
    \ocone(\ov
    {\Rd \pi_{ \ast} \mathcal{H}[-j]},\ov
    {\Rd \pi_{ \ast} \mathcal{F}})).
  \end{multline}
  It follows from \cite[Thm.~2.27~(iv)]{BurgosFreixasLitcanu:HerStruc}
  that the right hand side of this equality
  does not depend on the choice of the auxiliary hermitian
  structures.

  Finally, we consider the case when $\ov{\pi}$ has a metric
  different from the Fubini-Study
  metric. Thus, let $\ov \xi=(\ov \pi, \ov {\mathcal{F}},\ov {\Rd \pi_{
      \ast} \mathcal{F}})$ and write $\ov \xi'=(\ov \pi^{\FS}, \ov
  {\mathcal{F}},\ov {\Rd \pi_{
      \ast} \mathcal{F}})$. Then we put
  \begin{equation} \label{eq:6}
      T_{\mathfrak{t}}(\overline{\xi})=T_{\mathfrak{t}}(\overline{\xi}')+\ov
      \pi
      _{\ttwist}[\ch(\overline F)\bullet
      \widetilde{\Td}_{m}(\ov{\pi},\ov{\pi}^{\FS})].
    \end{equation}
\begin{definition}\label{def:8} Let $n$ and
$\mathfrak{t}$ be as in Theorem \ref{thm:1}. Then $T_{\mathfrak{t}}$
is the assignment that to each
relative metrized complex $\ov \xi$ associates  $T_{\mathfrak{t}}(\ov
\xi )$ given by equations \eqref{eq:29}, \eqref{eq:30} and \eqref{eq:6}.
\end{definition}

It remains to prove that $T_{\mathfrak{t}}$ satisfies axioms
(i) to (iv).
Axiom (i) follows from the differential equations satisfied by the
Bott-Chern classes.
Axiom (ii) follows from the functoriality of the canonical resolution,
the Chern forms and the Bott-Chern classes. Axiom (iii) follows from the
additivity of the canonical resolution and of the Chern
character. Finally Axiom (iv) follows from the multiplicativity of the Chern
character. This concludes the proof of Theorem \ref{thm:1}.
\end{proof}

We finish this section showing the compatibility of analytic torsion classes
with the composition of projective bundles.
Let $X$ be a smooth complex variety. Consider the commutative diagram
with cartesian square
\begin{displaymath}
  \xymatrix{& \PP^{n_{1}}_{X}\underset{X}{\times}\PP^{n_{2}}_{X}
    \ar[dl]_{p_{1}}\ar[dr]^{p_{2}}\ar[dd]^{p}&\\
    \PP^{n_{1}}_{X} \ar[dr]_{\pi _{1}} &&  \PP^{n_{2}}_{X}\ar[dl]^{\pi _{2}}\\
    &X&
  }
\end{displaymath}
On $\pi _{1}$ and $\pi_{2}$ we introduce arbitrary hermitian
structures and on $p_{1}$ and $p_{2}$ the hermitian structures induced
by the cartesian diagram.
\begin{proposition}\label{prop:comp_proj}
  Let $\ov{\mathcal{F}}$ be an object of
  $\oDb(\PP^{n_{1}}_{X}\underset{X}{\times}\PP^{n_{2}}_{X})$. Put
  arbitrary hermitian structures on $\Rd (p_{1})_{\ast} \mathcal{F}$,
  $\Rd (p_{2})_{\ast} \mathcal{F}$, and $\Rd p_{\ast} \mathcal{F}$. Then
  \begin{equation} \label{eq:15}
    T(\ov \pi _{1})+
    (\ov \pi _{1})_{\ttwist}(
      T(\ov p_{1}))
    =T(\ov \pi _{2})+
    (\ov \pi _{2})_{\ttwist}(
      T(\ov p_{2})),
  \end{equation}
   where we are using the convention at the end of Definition \ref{def:19}.
\end{proposition}
\begin{proof}
  By the anomaly formulas (Proposition \ref{prop:1}), if equation
  \eqref{eq:15} holds for a particular choice of hermitian structures
  on $\mathcal{F}$, $\Rd (p_{1})_{\ast} \mathcal{F}$, $\Rd
  (p_{2})_{\ast} \mathcal{F}$, and $\Rd p_{\ast} \mathcal{F}$, then it
  holds for any other choice.
  Let
  \begin{displaymath}
    \ov {\mathcal{F}}_{2}\dra
    \ov {\mathcal{F}}_{1}\dra
    \ov {\mathcal{F}}_{0}\dra
  \end{displaymath}
  be a distinguished triangle and put hermitian structures on the
  direct images as before. Then Proposition \ref{prop:2} implies that,
  if equation \eqref{eq:15} holds for two of them, then it also holds for the
  third. Since the objects of the form
  $\mathcal{G}(k,l):=p^{\ast}\mathcal{G}\otimes
  p_{1}^{\ast}\mathcal{O}(k)\otimes p_{2}^{\ast}\mathcal{O}(l)$ are a
  generating class of
  $\Db(\PP^{n_{1}}_{X}\underset{X}{\times}\PP^{n_{2}}_{X})$, the
  previous discussion shows that it is enough to prove the case
  $\mathcal{F}=\mathcal{G}(k,l)$, with the hermitian structure of
  $\mathcal{F}$ induced
  by a hermitian structure of $\mathcal{G}$ and the Fubini-Study
  metric on $\mathcal{O}(k)$ and $\mathcal{O}(l)$, and the hermitian
  structures on the direct images defined as in \eqref{eq:5}. In this case the
  result follows easily from the functoriality and the projection formula.
\end{proof}

\section{Compatible analytic torsion classes}
\label{sec:compat}

In this section we study the compatibility between analytic torsion
classes for closed immersions and analytic torsion classes for
projective spaces. It turns out that, once the compatibility between
the diagonal embedding of $\PP^{n}$ into $\PP^{n}\times \PP^{n}$ and
the second projection of $\PP^{n}\times \PP^{n}$ onto $\PP^{n}$ is
established, then all the other possible compatibilities
follow. Essentially this
observation can be traced back to~\cite{Bost:immersion}.

Let $n$,  $V$, $\ov V$  and
$\PP^n(V)$ be as in the previous section.
We
consider the diagram
\begin{displaymath}
    \xymatrix{
        \PP^n \ar[rd]_{\Id} \ar[r]^{\Delta\ \ \ }   & \PP^n
        \times\PP^n \ar[r]^{p_1}\ar[d]^{p_2} & \PP^n \ar[d]^{\pi}\\
        & \PP^n \ar[r]_{\pi_{1}} & \Spec\CC\ .}
\end{displaymath}
On $\PP^n$ we have the tautological short exact sequence
\begin{displaymath}
0\rightarrow \mathcal{O}(-1)\rightarrow V\rightarrow Q
\rightarrow 0\ .
\end{displaymath}
This induces on $\PP^n \times\PP^n$ the exact sequence
\begin{displaymath}
0\rightarrow p_2^{\ast}\mathcal{O}(-1) \rightarrow
V \rightarrow p_2^{\ast}Q\rightarrow 0\ .
\end{displaymath}
By composition with the injection
\begin{math}
p_1^{\ast}\mathcal{O}(-1)\hookrightarrow V
\end{math},
we obtain a morphism
\begin{math}
p_1^{\ast}\mathcal{O}(-1)\to p_2^{\ast}Q,
\end{math}
hence a section of $p_2^{\ast}Q\otimes p_1^{\ast}\mathcal{O}(1)$.
The zero locus of this section is the image of the diagonal. Moreover,
the associated Koszul complex is quasi-isomorphic to
$\Delta_{\ast}\mathcal{O}_{\PP^n}$. That is, the sequence
\begin{multline}\label{eq:15bis}
0\rightarrow \Lambda^n(p_2^{\ast}Q^\vee)\otimes
p_1^{\ast}\mathcal{O}_{\PP^n}(-n)\rightarrow\dots\\ \dots \rightarrow
\Lambda^1(p_2^{\ast}Q^\vee)\otimes
p_1^{\ast}\mathcal{O}_{\PP^n}(-1)\rightarrow
\mathcal{O}_{\PP^n\times\PP^n} \rightarrow
\Delta_{\ast}\mathcal{O}_{\PP^n}\rightarrow 0
\end{multline}
is exact.

On $T_{\PP^{n}}$ and $T_{\PP^{n}\times \PP^{n}}$ we consider the
Fubini-Study metrics. We denote by $\ov \Delta $ and $\ov p_{2}$ the
morphisms of $\oSm_{\ast/\CC}$ determined by these metrics. As in
\cite[Ex.~5.7]{BurgosFreixasLitcanu:HerStruc}, we have that $\ov
{p_{2}}\circ \ov \Delta =\ov \Id_{\PP^{n}}$, where
$T_{\ov{\Id_{\PP^{n}}}}=\ov 0$.
The Fubini-Study metric on $\mathcal{O}(-1)$ and the metric induced by
the tautological exact sequence on $Q$ induce a metric $\ov K(\Delta )$ on the Koszul
complex. This is a hermitian structure
on $\Rd \Delta _{\ast}\mathcal{O}_{\PP^{n}}$.
Finally on $\mathcal{O}_{\PP^{n}}$ we consider the trivial
metric. This is a hermitian structure on~$\Rd (p_{2})_{\ast}K(\Delta)$.

Fix a real additive genus $S$ and denote by $T_S$ the theory of
analytic torsion classes for closed immersions that is compatible with
the projection formula and transitive, associated to $S$
(
Theorem \ref{thm:17}). Moreover, fix a family of real numbers
$\mathfrak{t}=\{t_{nk}\ |\ n\geq 0,\ -n\leq k\leq 0\}$ and denote
$T_{\mathfrak{t}}$ the theory of generalized analytic torsion classes
for projective spaces associated to this family.

Compatible analytic torsion classes for
closed immersions and for projective spaces should combine to provide
analytic torsion classes for arbitrary projective morphisms, and these
classes should be transitive. The transitivity condition for
the composition $\Id_{\PP^{n}}=p_{2}\circ \Delta $ should give us
\begin{displaymath}
  0=T(\ov \Id_{\PP^{n}},\ov {\mathcal{O}}_{\PP^{n}},\ov
  {\mathcal{O}}_{\PP^{n}}) =
  T_{\mathfrak{t}}(\ov{p_{2}},\ov K(\Delta ),\ov{\mathcal{O}}_{\PP^{n}})+
  (\ov p_{2})_{\ttwist}(T_{S}(\ov \Delta ,\ov{\mathcal{O}}_{\PP^{n}}, \ov
  K(\Delta ))).
\end{displaymath}
In general we define

\begin{definition}\label{compatprojsp}
The theories of analytic torsion classes $T_S$ and  $T_{\mathfrak{t}}$
are called
\emph{compatible} if
\begin{equation}
  \label{eq:12}
  T_{\mathfrak{t}}(\ov{p_{2}},\ov K(\Delta ),\ov{\mathcal{O}}_{\PP^{n}})
  +(\ov p_{2})_{\ttwist}(T_{S}(\ov \Delta ,\ov{\mathcal{O}}_{\PP^{n}}, \ov
  K(\Delta )))=0.
\end{equation}
\end{definition}

\begin{theorem} \label{thm:15}
Let $S$ be a real additive genus. Then there exists a unique family
of real numbers $\mathfrak{t}=\{t_{n,k}\ |\ n\geq 0,\
-n\leq k\leq 0\}$ such that the theories of analytic torsion classes
$T_S$ and $T_{\mathfrak{t}}$
are compatible. The theory  $T_{\mathfrak{t}}$ will also be denoted $T_{S}$.
\end{theorem}
\begin{proof}
  The first step is to make explicit equation \eqref{eq:12} in terms
  of the main characteristic numbers $\mathfrak{t}$. To this end,
  first observe that, since the exact sequence
  \begin{equation}\label{eq:17}
    0\to T_{p_{2}}\to T_{\PP^{n}\times \PP^{n}}\to
    p_{2}^{\ast}T_{\PP^{n}}\to 0
  \end{equation}
  is split and the hermitian metric on $T_{\PP^{n}\times \PP^{n}}$ is
  the orthogonal direct sum metric, $\ov {p_{2}}=\pi
  ^{\ast}_{1}(\ov \pi ^{\FS})$. Next, we denote by $\ov K(\Delta
  )_{i}$ the component of degree $i$ of the Koszul complex, and we
  define
  \begin{displaymath}
    \ov {\Rd (p_{2})_{\ast}K(\Delta
      )_{i}}=
    \begin{cases}
      \ov {\mathcal{O}}_{\PP^{n}},&\text{for }i=0, \\
      \ov 0,&\text{for }i>0.
    \end{cases}
  \end{displaymath}
  Finally using Corollary \ref{cor:1}, functoriality and the
  compatibility with the projection formula, we derive
  \begin{align*}
    T_{\mathfrak{t}}(\ov{p_{2}},\ov K(\Delta
    ),\ov{\mathcal{O}}_{\PP^{n}})&=
    \sum_{i=0}^{n}(-1)^{i}T_{\mathfrak{t}}(\ov{p_{2}},\ov K(\Delta )_{i},\ov {\Rd (p_{2})_{\ast}K(\Delta
      )_{i}})\\
    &=\sum_{i=0}^{n}(-1)^{i}T_{\mathfrak{t}}(\pi _{1}^{\ast}\ov \xi
    _{n}(-i)\otimes \Lambda ^{i}\ov Q^{\vee})\\
    &=\sum_{i=0}^{n}(-1)^{i}t_{n,-i}\ch(\Lambda ^{i}\ov Q^{\vee}).
  \end{align*}

  Thus, the second and last step is to solve the equation
  \begin{equation}
    \label{eq:13}
    \sum_{i=0}^{n}(-1)^{i}t_{n,-i}\ch(\Lambda ^{i}\ov Q^{\vee})=
    -(p_{2})_{\ast}(T_{S}(\ov \Delta ,\ov{\mathcal{O}}_{\PP^{n}}, \ov
  K(\Delta ))\bullet \Td(\ov{p_{2}})).
  \end{equation}
  Since the left hand side of equation \eqref{eq:13} is closed, in
  order to be able to solve this equation we have to show that the
  right hand side is also closed. We compute
  \begin{align*}
    \dd_{\mathcal{D}}(p_{2})_{\ast}&(T_{S}(\ov \Delta
    ,\ov{\mathcal{O}}_{\PP^{n}}, \ov
  K(\Delta ))\bullet \Td(\ov{p_{2}}))\\
  &=(p_{2})_{\ast}\left(
    \sum_{i=0}^{n}(-1)^{i}\ch(\ov K(\Delta )_{i})\Td(\ov p_{2})-
    \Delta _{\ast}(\ch(\ov {\mathcal{O}}_{\PP^{n}}) \Td(\ov \Delta ))\Td(\ov p_{2})
    \right)\\
    &=(p_{2})_{\ast}\left(
    \sum_{i=0}^{n}(-1)^{i}p_{2}^{\ast}(\ch(\Lambda ^{i}\ov Q^{\vee}))
    p_{1}^{\ast}(\ch(\ov{\mathcal{O}}(-i)))\Td(\ov p_{2})\right)-1\\
    &=
    \sum_{i=0}^{n}(-1)^{i}\ch(\Lambda ^{i}\ov Q^{\vee})(p_{2})_{\ast}\left(
    p_{1}^{\ast}(\ch(\ov{\mathcal{O}}(-i)))\Td(\ov p_{2})\right)-1\\
    &=
    \sum_{i=0}^{n}(-1)^{i}\ch(\Lambda ^{i}\ov Q^{\vee})\pi _{1}^{\ast}\pi _{\ast}\left(
    \ch(\ov{\mathcal{O}}(-i))\Td(\ov \pi )\right)-1\\
  &=1-1=0.
  \end{align*}
  In the first equality we have used the differential equation of
  $T_{S}$. In the second one we have used the definition of the
  Koszul complex, the equation $\ch(\ov
  {\mathcal{O}}_{\PP^{n}})=1$ and the fact that, by the choice of
  hermitian structures on $T_{\ov \Delta } $ and $T_{\ov p_{2}}$ we
  have $\Td(\ov \Delta )\bullet \Delta ^{\ast}(\Td(\ov p_{2}))=1$. The third
  equality is the projection formula and the fourth is base change for
  cohomology. For the last equality we have used
  equation~\eqref{eq:42}.

  Both sides of equation \eqref{eq:13} are closed and
  defined up to boundaries, hence this is an equation in
  cohomology classes. The tautological exact sequence induces
  exact sequences
  \begin{displaymath}
    0\to \Lambda ^{k}Q^{\vee}\to \Lambda ^{k}V^{\vee}\to
    \Lambda ^{k-1}Q^{\vee}\otimes \mathcal{O}(1)\to 0,
  \end{displaymath}
  that give us equations
  \begin{displaymath}
    \ch(\Lambda ^{k}Q^{\vee})=\binom{n+1}{k}-\ch(\Lambda
    ^{k-1}Q^{\vee})\ch(\mathcal{O}(1)).
  \end{displaymath}
  Hence
  \begin{displaymath}
    \ch(\Lambda
    ^{k}Q^{\vee})=\sum^{k}_{i=0}(-1)^{i}\binom{n+1}{k-i}\ch(\mathcal{O}(i)).
  \end{displaymath}
  Since the classes $\ch(\mathcal{O}(i))$, $i=0,\dots,n$, form a
  basis of $\bigoplus _{p}H^{2p}_{\mathcal{D}}(\PP^{n},\RR(p))$, the
  same is true for the classes $\ch(\Lambda ^{i}Q^{\vee})$,
  $i=0,\dots,n$. Therefore, if $\mathbf{1}_{1}\in
  H^{1}_{\mathcal{D}}(\PP^{n},\RR(1))$ is the class represented by the
  constant function $1$, the classes $\mathbf{1}_{1}\bullet \ch(\Lambda
  ^{i}Q^{\vee})$, $i=0,\dots,n$ form a basis of $\bigoplus
  _{p=1}^{n+1}H^{2p-1}_{\mathcal{D}}(\PP^{n},\RR(p))$. This implies
  that equation \eqref{eq:13} has a unique solution.
 \end{proof}

 \begin{remark}
   Given a theory $T$ of analytic torsion classes for projective spaces,
   obtained from an arbitrary choice of
   characteristic numbers, in general, it does not exist an additive
   genus such that the associated theory of
   singular Bott-Chern classes is compatible with $T$. It would be
   interesting to characterize the collections of characteristic
   numbers that arise from Theorem \ref{thm:15}.
 \end{remark}

 By definition, compatible analytic torsion classes for closed immersions and
 projective spaces satisfy a compatibility condition for the
 trivial vector bundle and the diagonal embedding. When adding the
 functoriality and the projection formula, we obtain compatibility
 relations for arbitrary sections of the trivial projective bundle and
 arbitrary objects.

 Let $X$ be a smooth complex variety, let $\pi \colon \PP_{X}^{n}\to
 X$ be the projective space over $X$ and let $s\colon X\to
 \PP_{X}^{n}$ be a section. Choose any hermitian structure on $T_{\pi
 }$. Since we have an isomorphism $T_{s}\dra \Ld s^{\ast} T_{\pi
 }[-1]$, this hermitian structure induces a hermitian structure on
 $s$. Denote by $\ov \pi $ and $\ov s$ the corresponding morphisms in
 $\ov\Sm_{\ast/\CC}$. With this choice of hermitian structures, we
 have
 \begin{displaymath}
  \ov \pi \circ \ov s=(\pi\circ s,\ocone(\Ld s^{\ast} T_{\ov \pi
 }[-1],\Ld s^{\ast} T_{\ov \pi
 }[-1]))=(\Id_{X},\ov 0),
 \end{displaymath}
because the cone of the identity is meager.

 \begin{proposition}
   Let $S$ be a real additive genus. Let $T_{S}$ denote both, the
   theory of analytic torsion classes for closed immersions determined
   by $S$, and the theory of analytic torsion classes for projective
   spaces compatible with it.
   Let $\ov{\mathcal{F}}$ be an object of $\oDb(X)$. Put a hermitian
   structure on $\Rd s_{\ast} \mathcal{F}$. Then
   \begin{equation}
     \label{eq:19}
     T_{S}(\ov \pi , \ov{\Rd s_{\ast} \mathcal{F}}, \ov {\mathcal{F}})
     +\ov \pi _{\ttwist}(T_{S}(\ov s,\ov {\mathcal{F}}, \ov{\Rd s_{\ast}
         \mathcal{F}}))=0.
   \end{equation}
 \end{proposition}
 \begin{proof}
   By the anomaly formulas Proposition \ref{prop:1bis} and Proposition
   \ref{prop:1}, if equation \eqref{eq:19} holds for a particular
   choice of hermitian structure on $\Rd s_{\ast} \mathcal{F}$ then it
   holds for any other choice. Therefore we can assume that the
   hermitian structure on $\Rd s_{\ast} \mathcal{F}$ is given by $\ov
   K(s)\otimes \pi ^{\ast}\ov{\mathcal{F}}$, where $\ov K(s)$ is the
   Koszul complex associated to the section~$s$.
   By the projection formulas, if \eqref{eq:19} holds for the
   trivial bundle $\mathcal{O}_{X}$ then it holds for arbitrary
   objects of~$\oDb(X)$.

   We now prove that, if equation \eqref{eq:19} holds for a particular
   choice of hermitian structure $\ov \pi $, then it holds for any
   other choice. Thus, assume that equation \eqref{eq:19} is satisfied
   for $\ov \pi $ and $\ov s$. Let $\ov \pi '$ be another choice of
   hermitian structure on $\pi $ and $\ov s'$ be the hermitian
   structure induced on $s$. On one hand, we have
   \begin{equation}
     \label{eq:21}
     T_{S}(\ov \pi ',\ov K(s),\ov {\mathcal{O}}_{X})=
     T_{S}(\ov \pi ,\ov K(s),\ov  {\mathcal{O}}_{X})+
     \pi _{\ast}\left(
       \ch(\ov K(s)\bullet \widetilde \Td_{m}(\ov \pi ',\ov \pi
       )\bullet \Td(\ov \pi '))
     \right).
   \end{equation}
   On the other hand, we have
   \begin{align}
     &T_{S}(\ov s',\ov {\mathcal{F}}, \ov{\Rd s_{\ast}
         \mathcal{F}})
       \bullet \Td(\ov \pi' )\\
       &\phantom{A}=
     \left(T_{S}(\ov s,\ov {\mathcal{F}}, \ov{\Rd s_{\ast}
         \mathcal{F}}) +s_{\ast}(\widetilde{\Td}_{m}(\ov s',\ov
       s)\Td(\ov s'))\right)\bullet
       \left(\Td(\ov \pi )-\dd_{\mathcal{D}}(\widetilde \Td_{m}(\ov \pi ',\ov \pi
       )\bullet \Td(\ov \pi '))
     \right) \notag\\
     &\phantom{A}= T_{S}(\ov s,\ov {\mathcal{F}}, \ov{\Rd s_{\ast}
         \mathcal{F}})\bullet \Td(\ov \pi )+s_{\ast}(\widetilde{\Td}_{m}(\ov s',\ov
       s)\Td(\ov s'))\bullet \Td(\ov \pi ')\notag \\
       &\phantom{AAAAAAAAAAAAAAA}
       -T_{S}(\ov s,\ov {\mathcal{F}}, \ov{\Rd s_{\ast}
         \mathcal{F}})\bullet \dd_{\mathcal{D}}(\widetilde \Td_{m}(\ov \pi ',\ov \pi
       )\bullet \Td(\ov \pi '))\notag
   \end{align}
   In the group $\bigoplus_{p}\widetilde
   D_{D}^{2p-1}(\PP^{n}_{X},N_{s},p)$ we have
   \begin{multline}
     \label{eq:23}
     T_{S}(\ov s,\ov {\mathcal{F}}, \ov{\Rd s_{\ast}
         \mathcal{F}})\bullet \dd_{\mathcal{D}}(\widetilde \Td_{m}(\ov \pi ',\ov \pi
       )\bullet \Td(\ov \pi '))\\=
       \left(\ch(\ov K(s))-s_{\ast}(\Td(\ov s))\right)\bullet
       \left(\widetilde \Td_{m}(\ov \pi ',\ov \pi
       )\bullet \Td(\ov \pi ')\right).
   \end{multline}
   Observe that, by the definition of the hermitian structure
   of $\ov s$ and $\ov s'$ we have
   \begin{equation}
     \label{eq:24}
     \Td(\ov s)\bullet s^{\ast}\widetilde {\Td}_{m}(\ov \pi ',\ov \pi
     )=
     -\widetilde {\Td}_{m}(\ov s',\ov s)\bullet\Td(\ov s').
   \end{equation}
   By combining equations \eqref{eq:19}, \eqref{eq:21}, \eqref{eq:23}
   and \eqref{eq:24} we obtain
   \begin{equation}
     \label{eq:25}
     T_{S}(\ov \pi' , \ov{\Rd s_{\ast} \mathcal{F}}, \ov {\mathcal{F}})=
     -\pi _{\ast}\left(T_{S}(\ov s',\ov {\mathcal{F}}, \ov{\Rd s_{\ast}
         \mathcal{F}})
       \bullet \Td(\ov \pi' )
     \right).
   \end{equation}

   We now prove \eqref{eq:19} for a particular choice of
   hermitian structures. Let $f\colon X\to \PP^{n}$ denote
   the composition of
   $s$ with the projection $\PP^{n}_{X}\to \PP^{n}$. Then
   we have a commutative diagram with cartesian squares
   \begin{displaymath}
     \xymatrix{
     \PP^{n}\times X\ar[rrr]^{\Id\times f}\ar[dd]_{\pi }
     &&& \PP^{n}\times \PP^{n} \ar[dd]^{p_{2}}\\
     & X \ar[ul]^{s}\ar[dl]_{\Id}\ar[r]_{f}&
     \PP^{n}\ar[ru]_{\Delta }\ar[rd]^{\Id} &\\
     X\ar[rrr]_{f} &&& \PP^{n}
     }
   \end{displaymath}
   Let $\ov \Delta $ and $\ov p_{2}$ be as in Definition
   \ref{compatprojsp}.
   On $\ov \pi $ and $\ov s$ we put the hermitian structures induced
   by $\ov \Delta $. Since the Koszul complex $\ov
   K(s)=(\Id_{\PP^{n}}\times f )^{\ast}\ov K(\Delta )$, by
   Proposition \ref{prop:11} and functoriality, equation
   \eqref{eq:19} in this case follows from equation \eqref{eq:12}.
 \end{proof}

 We now study another compatibility between analytic torsion classes
 for closed immersions and projective spaces. Let $\iota\colon X\to Y$
 be a closed immersion of smooth complex varieties. Consider the
 cartesian square
 \begin{displaymath}
   \xymatrix{
     \PP^{n}_{X} \ar[d]_{\pi _{1}}\ar[r]^{\iota _{1}}&
     \PP^{n}_{Y} \ar[d]^{\pi }\\
     X\ar[r]_{\iota }& Y
}
 \end{displaymath}
 Choose hermitian structures on $\pi $ and $\iota $ and put on $\pi
 _{1}$ and $\iota _{i}$ the induced ones.

 \begin{proposition}
   \label{prop:5}
   Let $S$ be a real additive genus. Let $T_{S}$ denote both, the
   theory of analytic torsion classes for closed immersions determined
   by $S$, and  the theory of analytic torsion classes for projective
   spaces compatible with it.
   Let $\ov{\mathcal{F}}$ be and object of $\oDb(\PP^{n}_{X})$. Put hermitian
   structures on $\Rd (\pi _{1})_{\ast} \mathcal{F}$, $\Rd (\iota
   _{1})_{\ast} \mathcal{F}$ and $\Rd (\pi \circ \iota _{1})_{\ast}
   \mathcal{F}$. Then
   \begin{equation}
     \label{eq:26}
     T_{S}(\ov \pi )+ \ov \pi _{\ttwist}(T_{S}(\ov {\iota _{1}}))
     =
     T_{S}(\ov \iota)+ \ov \iota  _{\ttwist}(T_{S}(\ov {\pi _{1}})).
   \end{equation}
 \end{proposition}
 \begin{proof}
   By the anomaly formulas, if equation \eqref{eq:26} holds for a
   particular choice of metrics on $\Rd (\pi _{1})_{\ast}
   \mathcal{F}$, $\Rd (\iota _{1})_{\ast} \mathcal{F}$ and $\Rd (\pi
   \circ \iota _{1})_{\ast} \mathcal{F}$, then it holds for any
   choice. Because the sheaves $\mathcal{G}(k)$, with $\mathcal{G}$ a
   coherent sheaf on $X$, constitute a generating class of
   $\Db(\PP^{n}_{X})$ and by propositions \ref{prop:2bis} and
   \ref{prop:2}, we reduce to the case when $\mathcal{F}$ is of the
   form $\mathcal{G}(k)$. We choose arbitrary hermitian structures on
   $\mathcal{G}$ and $\Rd\iota_{\ast}\mathcal{G}$. Furthermore, we
   assume $\OO(k)$, $\Rd(\pi_{1})_{\ast}\OO(k)$ and
   $\Rd\pi_{\ast}\OO(k)$ endowed with the hermitian structures of
   Notation \ref{def:7}. From these choices and the projection
   formula, the objects $\Rd(\pi_{1})_{\ast}\mathcal{F}$,
   $\Rd(\iota_{1})_{\ast}\mathcal{F}$ and $\Rd (\pi \circ \iota
   _{1})_{\ast}\mathcal{F}$ automatically inherit hermitian
   structures. Indeed, it is enough to observe the natural
   isomorphisms
\begin{align}
  &\Rd(\pi_{1})_{\ast}\mathcal{F}\cong\mathcal{G}\otimes\Rd(\pi_{1})_{\ast}
  \OO(k)\label{eq:26_1}\\
  &\Rd(\iota_{1})_{\ast}(\pi_{1}^{\ast}\mathcal{G}\otimes\iota_{1}^{\ast}\OO(k))
  \cong\Ld\pi^{\ast}(\Rd\iota_{\ast}\mathcal{G})\otimes\OO(k)\label{eq:26_2}\\
  &\Rd(\pi\circ\iota_{1})_{\ast}\mathcal{F}\cong\Rd\pi_{\ast}(\Ld\pi^{\ast}\Rd\iota_{\ast}\mathcal{G}\otimes\OO(k))\cong\Rd\iota_{\ast}\mathcal{G}\otimes\Rd\pi_{\ast}\OO(k).\label{eq:26_3}
\end{align}
We now work out the left hand side of equation \eqref{eq:26}. Using
the projection formula for the theory $T_{S}$ for projective spaces,
and equations \eqref{eq:26_1}--\eqref{eq:26_3}, we find
\begin{equation}\label{eq:26_4}
	T_{S}(\ov{\pi})=t_{n,k}\bullet\ch(\ov{\Rd\iota_{\ast}\mathcal{G}}).
\end{equation}
Using the functoriality of $T_{S}$ for closed immersions and the
projection formula we have
\begin{align}
  T_{S}(\ov{\iota}_{1})=&\pi^{\ast}T_{S}(\iota,\ov{\mathcal{G}},\ov{\Rd\iota_{\ast}\mathcal{G}})\bullet\ch(\ov{\OO(k)})
  \notag\\
  \ov{\pi}_{\ttwist}(T_{S}(\ov\iota_1))=&T_{S}(\iota,\ov{\mathcal{G}},\ov{\Rd\iota_{\ast}\mathcal{G}})
  \bullet\pi_{\ast}(\ch(\ov{\OO(k)})\bullet\Td(\ov{\pi})).\label{eq:26_5}
\end{align}
Now for the right hand side of \eqref{eq:26}.
The projection formula for $T_{S}$ for closed immersions implies
\begin{equation}\label{eq:26_6}
	T_{S}(\ov\iota)=T_{S}(\iota,\ov{\mathcal{G}},\ov{\Rd\iota\mathcal{G}})
	\bullet\ch(\ov{\Rd\pi_{\ast}\OO(k)}).
\end{equation}
Similarly, we obtain
\begin{math}
	T_{S}(\ov\pi_1)=t_{n,k}\bullet\ch(\ov{\mathcal{G}}),
\end{math}
and hence
\begin{equation}\label{eq:26_7}
	\ov\iota_{\ttwist}(T_{S}(\ov\pi_{1}))=t_{n,k}\bullet\iota_{\ast}(\ch(\ov{\mathcal{G}})\bullet\Td(\ov\iota)).
\end{equation}
Using \eqref{eq:26_4}--\eqref{eq:26_7}, the difference of the two sides of \eqref{eq:26} becomes
\begin{displaymath}
	t_{n,k}\bullet\dd_{\mathcal{D}}T_{S}(\iota,\ov{\mathcal{G}},\ov{\Rd\iota_{\ast}\mathcal{G}})
	-T_{S}(\iota,\ov{\mathcal{G}},\ov{\Rd\iota_{\ast}\mathcal{G}})\bullet\dd_{\mathcal{D}}t_{n,k}=
	-\dd_{\mathcal{D}}(t_{n,k}\bullet T_{S}(\iota,\ov{\mathcal{G}},\ov{\Rd\iota_{\ast}\mathcal{G}}))=0
\end{displaymath}
in the group $\oplus_{p}\widetilde{\mathcal{D}}_{D}^{2p-1}(Y,N_{\iota},p)$.
\end{proof}

\section{Generalized analytic torsion classes}
\label{sec:gener-analyt-tors}

In this section we will extend the definition of analytic torsion
classes to arbitrary morphisms of smooth complex varieties. Our
construction is based on the construction of analytic torsion classes
by Zha in \cite{Zha99:_rieman_roch}.

\begin{definition}\label{def:genAT}
A \textit{theory of generalized analytic torsion classes} is an
assignment that, to each morphism $\ov f\colon X\to Y$ in
$\ov\Sm_{\ast/\CC}$ and relative metrized complex
$\overline{\xi}=(\ov f,\ov{\mathcal{F}},\ov{\Rd
    f_{\ast}\mathcal{F}}),$
 assigns a class of currents
$$T(\overline{\xi})\in \bigoplus_{p=1}^{n+1} \widetilde
  {\mathcal{D}}^{2p-1}_{D}(Y,N_{f},p)$$
satisfying the following properties:
\begin{enumerate}
  \item \label{item:1AT} (Differential equation) For any current $\eta\in T(\overline{\xi})$, we have
    \begin{equation}\label{eq:16}
      \dd_{\mathcal{D}} \eta=
    \ch(\overline {\Rd
    f_{\ast}\mathcal{F}})-\ov f_{\ttwist}
    [\ch(\overline {\mathcal{F}})].
  \end{equation}

  \item \label{item:2AT}(Functoriality) If $g\colon
    Y'\to Y$ is a morphism transverse to $f$, then
    \begin{displaymath}
      g^{\ast} T(\overline \xi)=T(g^{\ast} \overline
      \xi).
    \end{displaymath}
  \item \label{item:3AT}(Additivity and normalization) If $\overline{\xi}_1$,
    $\overline{\xi}_2$ are relative metrized complexes on $X$,
    then
    $$T(\overline{\xi}_1\oplus\overline{\xi}_2)
    =T(\overline{\xi}_1)+T(\overline{\xi}_2).$$
\item \label{item:4AT}(Projection formula) If $\overline{\xi}$ is a
relative metrized complex, and $\overline{\mathcal{G}}\in \Ob \oDb(Y)$,
then
$$T(\overline{\xi}\otimes
\overline{\mathcal{G}})=T(\overline{\xi})\bullet
\ch(\overline{\mathcal{G}}).
$$
\item \label{item:5AT}(Transitivity) If $\ov f\colon X\to Y$, $\ov
  g\colon Y\to Z$ are morphisms in $\ov\Sm_{\ast/\CC}$, and
  $(\ov f,\ov{\mathcal{F}},\ov{\Rd f_{\ast}\mathcal{F}})$
  and $(\ov g,\ov{\Rd f_{\ast}\mathcal{F}},\ov{\Rd
    (g\circ f)_{\ast}\mathcal{F}})$ are relative metrized complexes, then
  \begin{equation}
    \label{eq:18}
    T(\ov g\circ \ov f) =
  T(\ov g)+\ov g_{\ttwist}(T(\ov
  f)).
  \end{equation}
  \end{enumerate}
\end{definition}

Propositions \ref{prop:1tris} and \ref{prop:2tris} below contain several
anomaly and compatibility formulas satisfied by an arbitrary theory of generalized analytic torsion classes.
They follow from properties \ref{item:1AT}--\ref{item:3AT} and are analogous
to those in propositions \ref{prop:1bis} and \ref{prop:1}, \ref{prop:2bis} and
\ref{prop:2} respectively. The proofs are omitted, as they are similar to those of
the analogous statements referred to hereinbefore.

\begin{proposition}\label{prop:1tris}
  Let $T$ be a theory of generalized analytic torsion classes. Let
  \begin{math}
    \overline{\xi}= (\ov{f},\ov{\mathcal{F}}, \ov{\Rd f_{\ast}\mathcal{F}})
  \end{math}
  be a relative metrized complex.
  \begin{enumerate}
  \item \label{item:6tris}
    If $\ov{\mathcal{F}}'$ is another choice
    of metric on $\mathcal{F}$ and
    $\overline{\xi}_1= (\ov{f},\ov{\mathcal{F}}', \ov{\Rd
      f_{\ast}\mathcal{F}})$, then
    \begin{displaymath}
      T(\overline{\xi}_1)=T(\overline{\xi})+\ov f
      _{\ttwist}[\cht(\ov{\mathcal{F}}',\ov{\mathcal{F}})].
    \end{displaymath}
  \item \label{item:7tris}
    If $\ov{f}'$ is another choice of hermitian structure on $f$ and
    $\overline{\xi}_{2}= (\ov{f}',\ov{\mathcal{F}}, \ov{\Rd f_{\ast}\mathcal{F}})$, then
    \begin{equation}\label{eq:14tris}
      T(\overline{\xi}_{2})=T(\overline{\xi})+\ov f'
      _{\ttwist}[\ch(\overline {\mathcal{F}})\bullet
      \widetilde{\Td}_{m}(\ov{f}', \ov{f})].
    \end{equation}
  \item \label{item:8tris}
    If $\ov{\Rd f_{\ast}\mathcal{F}}'$ is a different
    choice of metric on $\Rd f_{\ast}\mathcal{F}$, and $\overline
    {\xi}_{3}= (\ov{f},\ov{\mathcal{F}}, \ov{\Rd f_{\ast}\mathcal{F}}')$, then
    \begin{displaymath}
      T(\overline{\xi}_{3})= T(\overline{\xi})-
      \cht(\ov{\Rd f_{\ast}\mathcal{F}}',
      \ov{\Rd f_{\ast}\mathcal{F}}).
    \end{displaymath}
  \end{enumerate}
\end{proposition}

\begin{proposition} \label{prop:2tris}
  Let $T$ be a theory of generalized analytic torsion classes. Let
  $\ov f\colon X\to Y$ be a morphism in $\ov\Sm_{\ast/\CC}$. Consider
  the distinguished
  triangles in $\oDb(X)$ and $\oDb(Y)$ respectively:
\begin{displaymath}
	(\ov{\tau}):\
        \ov{\mathcal{F}}_{2}\to\ov{\mathcal{F}}_{1}
        \to\ov{\mathcal{F}}_{0}\to\ov{\mathcal{F}}_{2}[1],\ \text{ and
        }\  
	(\ov{\Rd f_{\ast}\tau}):\
        \ov{\Rd f_{\ast}\mathcal{F}}_{2}\to
        \ov{\Rd f_{\ast}\mathcal{F}}_{1}
	\to\ov{\Rd f_{\ast}\mathcal{F}}_{0}\to
        \ov{\Rd f_{\ast}\mathcal{F}}_{2}[1],
\end{displaymath}
and the relative metrized complexes
$\ov{\xi}_{i}=(\ov{f},\ov{\mathcal{F}}_i,
        \ov{\Rd f_{\ast}\mathcal{F}}_i)$, $i=0,1,2$.
Then we have:
\begin{displaymath}
	\sum_{j=0,1,2}(-1)^{j}T(\ov{\xi}_{j})
	=\widetilde{\ch}(\ov{\Rd\pi_{\ast}\tau})
	-\ov f_{\ttwist}(\widetilde{\ch}(\ov{\tau })).
\end{displaymath}
\end{proposition}

The main result of this section is the following classification
theorem.

\begin{theorem}\label{thm:gen_anal_tor}
  Let $S$ be a real additive genus. Then there exists a unique theory
  of generalized analytic torsion classes that agrees with $T_{S}$
  when restricted to the class of closed immersions. Moreover, if $T$
  is a theory of generalized analytic torsion classes, then there
  exists a real additive genus $S$ such that $T=T_{S}$.
\end{theorem}

We will denote the theory associated to the additive genus $S$, whose existence is guaranteed by the preceding theorem, by
  $T_{S}$. In particular, there is a unique theory of generalized
  analytic torsion classes that agrees with $T^{h}$ when restricted to
  the class of closed immersions. This theory will be called
  homogeneous.

\begin{proof}
  We first prove the uniqueness. Let $T$ be a theory of analytic
  torsion classes that agrees with $T_{S}$ for
  the class of closed immersions. Since the
  restriction of $T$ to projective spaces, by the transitivity axiom,
  is compatible with $T_{S}$,  by Theorem \ref{thm:15}, it also agrees
  with $T_{S}$. Finally, the transitivity axiom implies that $T$ is determined
  by its values for closed immersions and projective spaces.

  We now prove the existence. For the moment, let $T_{S}$ be the
  theory of analytic torsion classes for closed immersions and
  projective spaces determined by $S$.
  Let $\ov f\colon X\to Y$ be a morphism
  in $\ov\Sm_{\ast/\CC}$, and let $\ov \xi =(\ov
  f,\ov{\mathcal{F}},\ov{\Rd f_{\ast}\mathcal{F}})$ be a relative
  metrized complex. Since $f$ is assumed to be projective, there is a
  factorization $f=\pi \circ \iota $, where $\iota \colon X\to
  \PP^{n}_{Y}$ is a closed immersion and $\pi \colon \PP_{Y}^{n}\to Y$
  is the projection. Choose auxiliary hermitian structures on $\iota
  $, $\pi $ and $\Rd \iota _{\ast} \mathcal{F}$. Then we define
  \begin{equation}
    \label{eq:27}
    T_{S}(\ov \xi )= T_{S}(\ov \pi )
    + \ov \pi_{\ttwist}(
      T_{S}(\ov \iota ))
      +\ov f_{\ttwist}\left[
        \ch(\ov{\mathcal{F}})\bullet \widetilde {\Td}_{m}(\ov f,\ov
        \pi \circ \ov \iota )
        \right]
  \end{equation}
 To simplify the notations, in the sequel we will also refer to it
 simply by $T(\ov\xi)$. The anomaly formulas easily imply that this
 definition does not
  depend on the choice of hermitian structures on $\iota
  $, $\pi $ and $\Rd \iota _{\ast} \mathcal{F}$. We next show that
  this definition is independent of the factorization of $f$. Let
  $f=\pi _{1}\circ \iota _{1}=\pi _{2}\circ \iota
  _{2}$ be two different factorizations, being $\PP^{n_{i}}$, the
  target of $\iota _{i}$, $i=1,2$. Since equation  \eqref{eq:27} is
  independent of
  the choice of auxiliary hermitian structures, by
  \cite[Lem.~5.12]{BurgosFreixasLitcanu:HerStruc},
  we may assume that $
  \ov f=\ov \pi _{1}\circ \ov \iota _{1}=\ov \pi _{2}\circ \ov \iota
  _{2}$.

  We consider the commutative diagram with cartesian square
  \begin{displaymath}
    \xymatrix{
      X \ar[r]^{j_{1}}\ar[dr]_{\Id_{X}}&
      X\underset{Y}{\times} \PP_{Y}^{n_{2}} \ar[r]^{k_{1}}\ar[d]^{q_{1}} &
      \PP_{Y}^{n_{1}}\underset{Y}{\times} \PP_{Y}^{n_{2}} \ar[d]^{p_{1}}\\
      & X  \ar[r]^{\iota _{1}}\ar[dr]_{f}&
      \PP_{Y}^{n_{1}}\ar[d]^{\pi _{1}}\\
      & & Y
    }
  \end{displaymath}
  where $j_{1}(x)=(x,\iota _{2}(x))$, $p_{1}$ is the first projection
  and $q_{1}$ and $k_{1}$ are defined by the cartesian square. The
  hermitian structure of $\ov \pi _{2} $ induces a hermitian structure
  on $p_{1}$ that, in turn, induces a hermitian structure on
  $q_{1}$. The hermitian structure of $\iota _{1}$ induces a hermitian
  structure on $k_{1}$ and the hermitian structure of $\iota _{2}$
  induces one on $j_{1}$. We will denote the corresponding morphisms
  of $\ov\Sm_{\ast/\CC}$ by $\ov p_{1}$, $\ov q_{1}$, $\ov k_{1}$ and
  $\ov j_{1}$. We consider also the analogous diagram obtained
  swapping 1 and 2. Finally, we write
  $\ov p=\ov \pi _{1}\circ \ov p_{1}=\ov \pi _{2}\circ \ov p_{2}$ and
  $\ov j=\ov k_{1}\circ \ov j_{1}= \ov k_{2}\circ \ov j_{2}$.
  Then we have
  \begin{align*}
    T(\ov \pi _{1})+(\ov \pi _{1})_{\ttwist}(T(\ov \iota _{1}))&=
    T(\ov \pi _{1})+(\ov \pi _{1})_{\ttwist}(T(\ov \iota _{1}))+
    \ov f_{\ttwist}\left(T(\ov q_{1})+(\ov q_{1})_{\ttwist}(T(\ov
    j_{1}))\right)\\
    &=
    T(\ov \pi _{1})+(\ov \pi _{1})_{\ttwist}\big(T(\ov \iota _{1})+
    (\ov \iota_{1}) _{\ttwist}(T(\ov q_{1}))\big)+\ov p_{\ttwist}(\ov
    k_{1})_{\ttwist}(T(\ov j_{1}))\\
    &=
    T(\ov \pi _{1})+(\ov \pi _{1})_{\ttwist}\big(T(\ov p _{1})+
    (\ov p_{1}) _{\ttwist}(T(\ov k_{1}))\big)+\ov p_{\ttwist}(\ov
    k_{1})_{\ttwist}(T(\ov j_{1}))\\
    &=T(\ov p)+\ov p_{\ttwist}(T(\ov j)).
  \end{align*}
  Analogously, we obtain
  \begin{displaymath}
    T(\ov \pi _{2})+(\ov \pi _{2})_{\ttwist}(T(\ov \iota _{2}))=
    T(\ov p)+\ov p_{\ttwist}(T(\ov j)).
  \end{displaymath}
  Hence $T_{S}$ is well defined for all relative metrized
  complexes.

  It remains to prove that it satisfies the properties of
  a theory of analytic torsion classes. The properties \ref{item:1AT}
  to \ref{item:4AT} are clear. We thus focus on property \ref{item:5AT}.
  Let $\ov{f}\colon X\to Y$ and $\ov{g}\colon Y\to Z$ be morphisms in
  $\ov{\Sm}_{\ast/\CC}$. We choose factorizations of
  $\ov{g}\circ\ov{f}$ and $\ov{g}$:
  \begin{align*}
    &\xymatrix{
      X\ar@{^{(}->}[r]^{\ov{i}}\ar[rd]_{\ov{g}\circ\ov{f}}
      &\PP^{m}_{Z}\ar[d]^{\ov{p}}\\
      &Z
    }
    &\xymatrix{
      Y\ar@{^{(}->}[r]^{\ov{\ell}}\ar[rd]_{\ov{g}}
      &\PP^{n}_{Z}\ar[d]^{\ov{r}}\\
      &Z,
    }
  \end{align*}
  where the hermitian
  structures on
  $\ov{p}$ and $\ov{r}$ come from fixed hermitan structures on the
  tangent bundles
  $T_{\PP^{m}_{\CC}}$ and $T_{\PP^{n}_{\CC}}$, and the hermitian
  structures $\ov i$ and $\ov {\ell }$ are obtained by using
  \cite[Lem.~5.12]{BurgosFreixasLitcanu:HerStruc}.
  We define
  $\varphi:X\to\PP^{m}_{\CC}$ to be the morphism obtained from $i$ by
  composing with the projection to $\PP^{m}_{\CC}$. Then we see that
  the morphism $j:=(\varphi,f)\colon X\to \PP^{m}_{Y}$ is a closed
  immersion. Indeed, it is
  enough to realize that the composition
\begin{displaymath}
	\xymatrix{
		X\ar[r]^{\hspace{-0.1cm}(\varphi,f)}	&\PP^{m}_Y\ar[r]^{(\Id,g)}	&\PP^{m}_Z
	}
\end{displaymath}
agrees with the closed immersion $i$ and that $G:=(\Id,g)$ is separated (since proper). We can thus decompose $\ov{f}$ as
\begin{displaymath}
	\xymatrix{
			X\ar@{^{(}->}[r]^{\ov{j}}\ar[rd]_{\ov{f}}		&\PP^{m}_{Y}\ar[d]^{\ov{q}}\\
				&Y.
		}
\end{displaymath}
Again, in this factorization the hermitian structure $\ov{q}$ comes
from the previously fixed hermitian structure on $T_{\PP^{m}_{\CC}}$
and the hermitian structure $\ov j$ is obtained by using
\cite[Lem.~5.12]{BurgosFreixasLitcanu:HerStruc}. Because
$\ov{g}\circ\ov{f}=\ov{p}\circ\ov{i}$ and by the very construction of
$T$ for arbitrary projective morphisms \eqref{eq:27}, we have
\begin{equation}\label{eq:ma_1}
	T(\ov{g}\circ\ov{f})=T(\ov{p})+\ov{p}_{\ttwist}(T(\ov{i})).
\end{equation}
We proceed to work on $T(\ov{i})$. For this we write the commutative diagram
\begin{displaymath}
	\xymatrix{
		X\ar@{^{(}->}[r]^{j}\ar@{^{(}->}[rd]_{i}	&\PP^{m}_{Y}\ar@{^{(}->}[r]^{k}\ar[d]	^{G}	 &\PP^{m}_{\PP^{n}_{Z}}\ar@{=}[r]		 &\PP^{n}_{\PP^{m}_{Z}}\ar[d]^{\pi}\\
		&\PP^{m}_{Z}\ar[rr]^{\Id}	&	&\PP^{m}_{Z}.
	}
\end{displaymath}
We recall that $G=(\Id,g)$ and $k=(\Id,\ell)$. Below, $G$, $k$ and
$\pi$ will be endowed with the obvious hermitian structures. With
these choices, we observe that $\ov{i}=\ov{G}\circ\ov{j}$ and
$\ov{G}=\ov{\pi}\circ\ov{k}$. Taking also into account the
construction of~$T$ and the fact that $T=T_{S}$ is transitive for
compositions of closed immersions, we find
\begin{equation}\label{eq:ma_2}
		T(\ov{i})=T(\ov{\pi}\circ\ov{k}\circ\ov{j})
		=T(\ov{\pi})+\ov{\pi}_{\ttwist}(T(\ov{k}))+\ov{G}_{\ttwist}(T(\ov{j}))
		=T(\ov{G})+\ov{G}_{\ttwist}(T(\ov{j})).
\end{equation}
Therefore, from equations \eqref{eq:ma_1}, \eqref{eq:ma_2} and the identity $\ov{p}_{\ttwist}\ov{G}_{\ttwist}=\ov{g}_{\ttwist}\ov{q}_{\ttwist}$ we derive
\begin{equation}\label{eq:ma_3}
		T(\ov{g}\circ\ov{f})=T(\ov{p})+\ov{p}_{\ttwist}(T(\ov{G}))+\ov{g}_{\ttwist}\ov{q}_{\ttwist}(T(\ov{j})).
\end{equation}
We claim that
\begin{equation}\label{eq:ma_4}
		T(\ov{p})+\ov{p}_{\ttwist}(T(\ov{G}))=T(\ov{g})+\ov{g}_{\ttwist}(T(\ov{q})).
\end{equation}
Assuming this for a while, we combine \eqref{eq:ma_3} and \eqref{eq:ma_4} into
\begin{equation}
		T(\ov{g}\circ\ov{f})=T(\ov{g})+\ov{g}_{\ttwist}(T(\ov{q})+\ov{q}_{\ttwist}(T(\ov{j})))\\
		=T(\ov{g})+\ov{g}_{\ttwist}(T(\ov{f})).
\end{equation}
Hence we are lead to prove \eqref{eq:ma_4}. For this we construct the commutative diagram with cartesian squares
\begin{displaymath}
	\xymatrix{
		\PP^{m}_{Y}\ar@{^{(}->}[r]^{\hspace{-0.5cm}\tilde{\ell}}\ar[d]^{\ov{q}}	 &\PP^{m}_{Z}\times_{Z}\PP^{n}_{Z}\ar[r]^{\hspace{0.5cm}\tilde{r}}\ar[d]^{\tilde{p}}	&\PP^{m}_{Z}\ar[d]^{\ov{p}}\\
		Y\ar@{^{(}->}[r]^{\ov{\ell}}	&\PP^{n}_{Z}\ar[r]^{\ov{r}}	&Z.
	}
\end{displaymath}
Observe that $\ov{G}=\tilde{r}\circ\tilde{\ell}$. Recall now
Proposition \ref{prop:comp_proj} and Proposition \ref{prop:5}. We then
have the chain of equalities
\begin{multline*}
  T(\ov{p})+\ov{p}_{\ttwist}(T(\ov{G}))
  =T(\ov{p})+\ov{p}_{\ttwist}(T(\tilde{r})+
  \tilde{r}_{\ttwist}(T(\tilde{\ell})))
  =T(\ov{r})+\ov{r}_{\ttwist}(T(\tilde{p})+
  \tilde{p}_{\ttwist}(T(\tilde{\ell}))\\ 
  =T(\ov{r})+\ov{r}_{\ttwist}(T(\ov{\ell})+
  \ov{\ell}_{\ttwist}(T(\ov{q}))) 
  =T(\ov{g})+\ov{g}_{\ttwist}(T(\ov{q})).
\end{multline*}
This proves the claim.

The last assertion of the statement of the theorem follows from the uniqueness.
\end{proof}

\begin{theorem}\label{thm:comparison_genus}
\begin{enumerate}
\item Let $T$ be a theory of generalized analytic torsion
  cla\-sses. Then there is a unique real additive genus $S$ such that,
  for any relative metrized
    complex $\ov \xi :=(\ov f, \ov{\mathcal{F}}, \ov{\Rd
      f_{\ast}\mathcal{F}})$, we have
    \begin{equation}
      \label{eq:70bis}
      T(\ov \xi )-T^{h}(\ov \xi )=-f_{\ast}[\ch(\mathcal{F})\bullet
      \Td(T_{f}) \bullet S(T_{f})\bullet \mathbf{1}_{1}].
    \end{equation}
  \item
    Conversely, any real additive
    genus $S$ defines, by means of equation \eqref{eq:70bis}, a
    unique
    theory of generalized analytic torsion classes $T_S$.
\end{enumerate}
\end{theorem}
\begin{proof}
  We prove the first item, the second being immediate. Let $S$ be the
  real additive genus corresponding to $T$, provided by Theorem
  \ref{thm:gen_anal_tor}. Then \eqref{eq:70bis} holds for embedded
  metrized complexes. Because $T$ and $T^{h}$ are both transitive, it
  suffices to show that \eqref{eq:70bis} holds whenever
  $f\colon\PP^{n}_{X}\to X$ is a trivial projective bundle. Observe
  $T$ and $T^{h}$ satisfy the same anomaly formulas. Then, since the
  sheaves $\mathcal{G}(k)$, $k=-n,\dots,0$ form a generating class for
  $\Db(\PP^{n}_{X})$, and by the projection formula for $T$ and
  $T^{h}$, we easily reduce to the case $\ov{\xi}=\ov{\xi}(k)$. Let
  $t_{n,k}$, $t_{n,k}^{h}$ be the characteristic numbers of $T$,
  $T^{h}$ respectively. We have to establish the equality
  \begin{equation}\label{eq:70_1}
    t_{n,-i}-t^{h}_{n,-i}=-\pi_{\ast}(\ch(\ov{\OO}(-i))\Td(\ov{\pi})
    S(T_{\ov{\pi}})),\quad
    i=-n,\dots,0.
  \end{equation}
  This is an equation of real numbers. By functoriality, this equation
  is equivalent to the analogous equation in $\oplus_{p}
  H^{2p-1}_{\mathcal{D}}(\PP^{n}_{\CC},\RR(p))$, for the second
  projection $p_{2}:\PP^{n}_{\CC}\times\PP^{n}_{\CC}\to\PP^{n}_{\CC}$
  instead of $\pi$. Because the classes $\ch(\Lambda^{i}Q^{\vee})$
  constitute a basis for $\oplus_{p}
  H^{2p-1}_{\mathcal{D}}(\PP^{n}_{\CC},\RR(p))$, \eqref{eq:70_1} is
  equivalent to the equation in cohomology
  \begin{multline}\label{eq:70_2}
    \sum_{i}(-1)^{i}(t_{n,-i}-t^{h}_{n,-i})\ch(\Lambda^{i}\ov{Q}^{\vee})
    =\\
    - p_{2\ast}(\sum_{i}(-1)^{i}\ch(p_{1}^{\ast}\ov{\OO}(-i)\otimes
    \Lambda^{i}p_{2}^{\ast}\ov{Q}^{\vee})\Td(\ov{p}_{2})
    S(T_{\ov{p}_{2}})\bullet\mathbf{1}_{1}).
  \end{multline}
  Recalling the exact sequence \eqref{eq:15bis}, minus the right hand side
  of \eqref{eq:70_2} becomes
  \begin{multline*}
    p_{2\ast}(\ch(\ov{\Delta_{\ast}\OO_{\PP^{n}}})\Td(\ov{p}_{2})
    S(T_{\ov{p}_{2}})\bullet\mathbf{1}_{1})
    =\\
    p_{2\ast}(\Delta_{\ast}(\ch(\ov{\OO}_{\PP^{n}})\Td(\ov{\Delta}))\Td(\ov{p}_{2})
    S(T_{\ov{p}_{2}})\bullet\mathbf{1}_{1})
    =S(T_{\PP^{n}})\bullet\mathbf{1}_{1}.
  \end{multline*}
  On the other hand, using the compatibility condition (Definition
  \ref{compatprojsp}), the left hand side of \eqref{eq:70_2} can be
  equivalently written as
  \begin{multline}\label{eq:70_3}
    T(\ov{p}_{2},\ov{\Delta_{\ast}\OO_{\PP^{n}}},\ov{\OO_{\PP^n}})
    -T^{h}(\ov{p}_{2},\ov{\Delta_{\ast}\OO_{\PP^{n}}},\ov{\OO_{\PP^n}})
    =\\
    -p_{2\ttwist}(T(\Delta,\ov{\OO}_{\PP^{n}},\ov{\Delta_{\ast}\OO_{\PP^{n}}})
    -T^{h}(\Delta,\ov{\OO}_{\PP^{n}},\ov{\Delta_{\ast}\OO_{\PP^{n}}})).
  \end{multline}
  The genus $S$ is additive, so in Deligne cohomology
  we have the relation
  \begin{displaymath}
    S(T_{\ov{\Delta}})=S(T_{\PP^{n}})-\Delta^{\ast}S(T_{\PP^{n}\times\PP^{n}})
    =
    S(T_{\PP^{n}})-\Delta^{\ast}p_{1}^{\ast}S(T_{\PP^{n}})
    -\Delta^{\ast}p_{2}^{\ast}S(T_{\PP^{n}}) 
    =-S(T_{\PP^{n}}).
  \end{displaymath}
  Hence, since the statement is known for closed immersions, the right hand
  side of \eqref{eq:70_3} becomes
  \begin{displaymath}
    p_{2\ast}(\Delta_{\ast}(\ch(\ov{\OO_{\PP^{n}}})\Td(T_{\ov{\Delta}})
    S(T_{\ov{\Delta}})\bullet\mathbf{1}_{1})\Td(\ov{p}_{2}))=
    -S(T_{\PP^{n}})\bullet\mathbf{1}_{1}.
  \end{displaymath}
This concludes the proof.
\end{proof}

\section{Higher analytic torsion forms of Bismut and K\"ohler}
\label{sec:high-analyt-tors}

We now explain the relationship  between the theory of analytic torsion
forms of Bismut-K\"ohler \cite{Bismut-Kohler} and the theory of
generalized analytic torsion classes developed so far.

Let $\pi \colon X\to Y$ be a smooth projective morphism (a projective
submersion) of smooth
complex varieties. Let
$\omega $ be a closed $(1,1)$-form on $X$ that induces a K\"ahler metric
on the fibers of $\pi $. Then $(\pi,\omega )$ is called a K\"ahler
fibration. The form $\omega $ defines a hermitian structure on
$T_{\pi}$, and we will abusively write $\ov{\pi}=(\pi,\omega)$ for the corresponding morphism in $\ov{\Sm}_{\ast/\CC}$.

Let $\ov F$ be a hermitian vector bundle on $X$ such that
for every $i\ge 0$,
$R^{i}\pi _{\ast} F$ is locally free. We consider on $R^{i}\pi _{\ast}
F$ the $L^{2}$ metric obtained using Hodge theory on the fibers of
$\pi $. Using \cite[Def.~3.47]{BurgosFreixasLitcanu:HerStruc} we
obtain a hermitian structure on $\pi _{\ast}F$,
denoted by $\ov{\pi _{\ast} F}_{L^{2}}$. Then $\ov{\xi}=(\ov \pi
,\ov F, \ov{\pi _{\ast} F}_{L^{2}})$  is a relative
metrized complex.  The relative metrized complexes that arise in this
way will be said to be \emph{K\"ahler}.

In the paper \cite{Bismut-Kohler}, Bismut and K\"ohler associate to
every K\"ahler relative metrized complex $\ov{\xi}$ a
differential form, that we temporarily denote by $\tau (\ov \xi
)$. Since in \cite{Bismut-Kohler} the authors use real valued
characteristic classes, while we use characteristic classes in the
Deligne complex, we have to change the normalization of this form.
To this end,
if $\tau(\ov{\xi})^{(p-1,p-1)}$ is the component of
degree $(p-1,p-1)$ of $\tau(\ov{\xi})$, then we put
\begin{displaymath}
  T^{BK}(\ov{\xi})^{(2p-1,p)}= \frac{1}{2}(2\pi
  i)^{p-1}[\tau(\ov{\xi})^{(p-1,p-1)}]
  \in\widetilde{\mathcal{D}}_{D}^{2p-1}(Y,\emptyset,p).
\end{displaymath}
We recall that $[\cdot]$ converts differential forms into currents
according with the conventions in \cite[\S1]{BurgosLitcanu:SingularBC}
(compare with equation \eqref{eq:88}). We
define
\begin{displaymath}
  T^{BK}(\ov{\xi})=\sum_{p\geq 1}T^{BK}(\ov{\xi})^{(2p-1,p)}.
\end{displaymath}
The first main result of \cite{Bismut-Kohler} is that this class
satisfies the differential equation
\begin{displaymath}
  d_{\mathcal{D}}T^{BK}(\ov{\xi})=\ch(\ov{\pi _{\ast} F}_{L^{2}})-
  \ov {\pi}_{\ttwist}[\ch(\ov F)].
\end{displaymath}
Thus, $T^{BK}(\ov \xi )$ is an example of analytic torsion class.

Let now $\omega '$ be another closed $(1,1)$-form on $X$ that induces
a K\"ahler metric on the fibers of $\pi $. We denote $\ov \pi
'=(\pi,\omega ')$. Let $\ov F'$ be the vector bundle $F$ with another
choice of metric and define $\ov{\pi _{\ast}F}'_{L^{2}}$ to be the object
$\pi _{\ast}F$ with the $L^{2}$ metric induced by $\omega '$ and $\ov
F'$. We write $\ov \xi '$ for the K\"ahler relative metrized complex
$(\ov \pi ',\ov F ',\ov{\pi
  _{\ast}F}'_{L^{2}})$.

The second main result of \cite{Bismut-Kohler} is the following anomaly formula.

\begin{theorem}[\cite{Bismut-Kohler} Theorem 3.10]\label{thm:22}  The following
  formula holds:
  \begin{displaymath}
    T^{BK}(\ov \xi ')-T^{BK}(\ov \xi )= \cht(\ov{\pi
  _{\ast}F}_{L^{2}},\ov{\pi
  _{\ast}F}'_{L^{2}}) +\ov \pi' _{\ttwist}\left[\ch(\ov F)\bullet \widetilde
{\Td}_{m}(\ov \pi' ,\ov \pi ) -\cht(\ov F,\ov
F')\right].
  \end{displaymath}
\end{theorem}

In the book \cite{Bismut:Asterisque}, Bismut studies the compatibility
of higher analytic torsion forms with complex immersions. Before
stating his result we have to recall the definition of the $R$-genus
of Gillet and Soul\'e \cite{GSATAT}. It is the additive genus
  attached to the power series
  \begin{equation}\label{eq:81}
    R(x)=\sum_{\substack{m \text{ odd}\\ m\geq
        1}}\left(2\zeta'(-m)+\left(1+\frac{1}{2}+\dots+\frac{1}{m}\right)
      \zeta(-m)\right)\frac{x^{m}}{m!}.
  \end{equation}
Let $T_{-R/2}$ be the theory of analytic torsion classes for closed
immersions associated to~$\frac{-1}{2}R$.

\begin{remark}\label{rem:10} The fact that we obtain the additive
  genus $-R/2$ instead of $R$ is due to two facts. The signs
  comes from the minus sign in equation \eqref{eq:70}, while the
  factor $1/2$ comes from the difference of the
  normalization of Green forms used in this paper and the one used in
  \cite{GilletSoule:ait}. Note however that the arithmetic
  intersection numbers computed using both normalizations agree,
  because the definition of arithmetic degree in \cite[\S
  3.4.3]{GilletSoule:ait} has a factor $1/2$ while the definition of
  arithmetic degree in \cite[(6.24)]{BurgosKramerKuehn:cacg} does not.
\end{remark}

Consider a commutative diagram of smooth complex varieties
\begin{displaymath}
  \xymatrix {X\ar[dr]_{f} \ar[r]^{\iota} & Y\ar[d]^{g} \\
    & Z}
\end{displaymath}
where $f$ and $g$ are projective submersions and $\iota $ is a closed immersion. Let
$\ov F$ be a hermitian vector bundle on $X$ such that the sheaves
$R^{i}f_{\ast}F$ are locally free and let
\begin{displaymath}
  0\to \ov E_{n}\to \dots \to \ov E_{0}\to \iota _{\ast} F\to 0
\end{displaymath}
be a resolution of $\iota_{\ast}F$ by hermitian vector bundles. We assume
that for all $i,j$, $R^{i}g_{\ast}E_{j}$ is locally free.  We will
denote by $E$ the complex $E_{n}\to\dots \to E_{0}$.
Let
$\omega ^{X}$ and $\omega ^{Y}$ be closed $(1,1)$ forms that define a
structure of K\"ahler fibration on $f$ and $g$ respectively.  As
before we write $\ov f=(f,\omega ^{X})$ and $\ov g=(g,\omega
^{Y})$. The exact sequence
\begin{displaymath}
  0\longrightarrow T_{f}\longrightarrow  f^{\ast}T_{g}\longrightarrow
  N_{X/Y}\longrightarrow 0
\end{displaymath}
induces a hermitian structure on $N_{X/Y}$.
We will denote $\ov \iota
$ the inclusion $\iota $ with this hermitian structure. Finally we
denote by $\ov{f_{\ast}F}_{E}$ the hermitian structure on $f_{\ast}F$
induced by the hermitian structures $\ov{g_{\ast} E_{j}}_{L^{2}}$,
$j=0,\dots,n$.

Then, adapted
to our language, the main result of \cite{Bismut:Asterisque} can be stated as
follows.

\begin{theorem}[\cite{Bismut:Asterisque} Theorems 0.1 and
  0.2] \label{thm:21} The
  following equation holds in the group $\bigoplus
  _{p}\widetilde{\mathcal{D}}_{D}^{2p-1}(Z,\emptyset,p)$:
  \begin{multline*}
    T^{BK}(\ov f,\ov F,\ov{f_{\ast} F}_{L_{2}})=
    \sum_{j=0}^{n}(-1)^{j}T^{BK}(\ov g,\ov E_{j},\ov{f_{\ast} E_{j}}_{L_{2}})\\
    +\ov g_{\ttwist}(T_{-R/2}(\ov \iota ,\ov F, \ov E))
    +\cht(\ov{f_{\ast}F}_{E},\ov{f_{\ast}F}_{L^{2}}).
  \end{multline*}
\end{theorem}

We can particularize the previous result to the case when $F=0$. Then
$E$ and
$g_{\ast}E$ are acyclic objects. The hermitian structures of
$\ov E_{j}$ and
$\ov {g_{\ast}E_{j}}_{L^{2}}$ induce hermitian structures on them. We
denote these hermitian structures as $\ov E$ and~$\ov
{g_{\ast}E}_{L^{2}}$.

\begin{corollary} \label{cor:8}
Let $\ov E$ be a bounded acyclic complex of hermitian vector
bundles on $Y$ such that the direct images $R^{i}g_{\ast}E_{j}$ are
locally free on $Z$. Then
\begin{displaymath}
  \sum_{j=0}^{n}(-1)^{j}T^{BK}(\ov{g},\ov{E_{j}},\ov{g_{\ast}E_{j}}_{L^{2}})=
  \cht(\ov{g_{\ast}E}_{L^{2}})
  -\ov{g}_{\ttwist}(\cht(\ov{E}))
\end{displaymath}
in $\bigoplus
  _{p}\widetilde{\mathcal{D}}_{D}^{2p-1}(Z,\emptyset,p)$.
\end{corollary}

We will also need a particular case of functoriality and projection
formula for the higher analytic torsion forms of Bismut-K\"ohler
proved by R\"ossler \cite{Roessler:ARR}.

The relative metrized complexes $\ov \xi _{n}(k)$ of
Notation \ref{def:7} are K\"ahler. Therefore we can apply the
construction of Bismut-K\"ohler to them. We denote
\begin{equation}
  \label{eq:80}
  t^{BK}_{n,k}=T^{BK}(\ov \xi _{n}(k)).
\end{equation}
By Corollary \ref{cor:8}, the numbers $t^{BK}_{n,k}$ satisfy the
relation \eqref{eq:79}. Hence they are determined by the main
characteristic numbers $t^{BK}_{n,k}$ for $-n\le k \le 0$.

\begin{theorem}[\cite{Roessler:ARR} Lemma 7.15]\label{thm:20} Let
  $\pi\colon
  \PP^{n}_{X}\to X$ be a trivial projective bundle. Let $\ov G$ be a
  hermitian vector bundle on $X$. Then
  \begin{displaymath}
    T^{BK}(\ov \xi _{n}(k)\otimes \ov G)=t^{BK}_{n,k}\bullet \ch(\ov
    G).
  \end{displaymath}
\end{theorem}
\begin{proof}
  In \cite{Roessler:ARR} this result is proved for $k\gg 0$. Using
  Corollary \ref{cor:8} and the Koszul resolution \eqref{eq:38} one
  can extend the result to all $k\in \ZZ$.
\end{proof}

We have all the ingredients we need to prove the main result of this
section.

\begin{theorem}\label{thm:19}
  Let $T_{-R/2}$ be the theory of generalized analytic torsion classes
  associated to the additive genus $\frac{-1}{2}R$. Then, for every K\"ahler
  relative metrized complex $\ov \xi $, we have
  \begin{displaymath}
    T^{BK}(\ov \xi )=T_{-R/2}(\ov \xi ).
  \end{displaymath}
  In particular $T_{-R/2}$ extends the construction of Bismut-K\"ohler
  to arbitrary projective morphisms of smooth complex varieties and
  arbitrary smooth metrics.
\end{theorem}
\begin{proof}
  Let $\mathfrak{t}^{BK}=\{t^{BK}_{n,k}\mid n\ge 0, -n\le k \le
  0\}$ and let $T_{\mathfrak{t}^{BK}}$ be the theory of analytic
  torsion classes for projective spaces associated to it.
  Let $\pi\colon \PP^{n}_{X}\to X$ be a relative projective space and
  let $\ov \xi =(\ov \pi ,\ov E,\ov{\pi _{\ast}E}_{L^{2}})$ be a
  K\"ahler relative metrized complex. By choosing $d\gg 0$ we may
  assume that all the coherent sheaves of the resolution $\gamma
  _{d}(F)$  of Corollary
  \ref{cor:4} are locally free. Using theorems \ref{thm:20} and \ref{thm:22},
  Proposition \ref{prop:1}
  and corollaries \ref{cor:1} and \ref{cor:8} we obtain that
  \begin{displaymath}
    T^{BK}(\ov \xi)=T_{\mathfrak{t}^{BK}}(\ov \xi ).
  \end{displaymath}

  By Theorem \ref{thm:21}, the theories
  $T_{\mathfrak{t}^{BK}}$ and $T_{-R/2}$ are compatible in the sense of
  Definition \ref{compatprojsp}. Therefore, $T^{BK}=T_{-R/2}$ when
  restricted to projective spaces.

  Finally, by factoring a smooth projective morphism as a closed
  immersion followed by the projection of a relative projective space,
  Theorem \ref{thm:21} implies that $T^{BK}=T_{-R/2}$ for all smooth
  projective morphisms.
\end{proof}

\begin{remark}\label{rem:bismut_kohler}
  \begin{enumerate}
  \item The construction of Bismut-K\"ohler applies to a wider class
    of varieties and morphisms: complex analytic manifolds and proper
    K\"ahler submersions. However for the comparison we have to
    restrict to smooth algebraic varieties and smooth projective
    morphisms.
  \item The results of Bismut and his coworkers are more precise. Here
    the class $T^{BK}(\ov \xi )$ is well defined up to the image of
    $\dd_{\mathcal{D}}$. In contrast, the higher analytic torsion form
    of Bismut and K\"ohler is a well defined differential form, local
    on the base and whose class modulo $\dd_{\mathcal{D}}$ agrees with
    $T^{BK}(\ov \xi )$.
 \end{enumerate}
\end{remark}

As a consequence of Theorem \ref{thm:19}, we obtain the following
results that, although they should follow from the definition of
higher
analytic torsion classes, we have not been able to find them explicitly in
the literature.

\begin{corollary}\label{cor:9} Let $f\colon X\to Y$ be a
  smooth projective morphism of smooth complex varieties, and let $\ov \xi
  =(\ov f, \ov E, \ov{f_{\ast}E}_{L^{2}})$ be a K\"ahler relative
  metrized complex.
  \begin{enumerate}
  \item Let $g\colon Y'\to Y$ be a morphism of smooth
    complex varieties. Then
    \begin{displaymath}
      T^{BK}(g^{\ast}\ov \xi )=g^{\ast}T^{BK}(\ov \xi ).
    \end{displaymath}
  \item Let $\ov G$ be a hermitian vector bundle on $Y$. Then
    \begin{displaymath}
      T^{BK}(\ov \xi \otimes \ov G)=T^{BK}(\ov \xi )\bullet \ch(\ov G).
    \end{displaymath}
  \end{enumerate}
\end{corollary}

The last consequence we want to discuss generalizes results already
proved by Ber\-thomieu-Bismut \cite[Thm 3.1]{BerthomieuBismut} and Ma
\cite[Thm. 0.1]{Ma:MR1765553}, \cite[Thm. 0.1]{Ma:MR1796698}. However
we note that while we stay within the algebraic category and
work with projective morphisms, these authors deal with proper
K\"ahler holomorphic submersions of complex manifolds. Let $\ov{g}\colon X\to
Y$ and $\ov{h}\colon Y\to Z$ be morphisms in the category
$\ov{\Sm}_{\ast/\CC}$, such that the composition
$\ov{f}=\ov{h}\circ\ov{g}$ is a smooth morphism. We choose a structure
of K\"ahler fibration on $f$, that we denote $\ov{f}'$. Let $\ov{E}$
be a hermitian vector bundle on $X$ and assume that the higher direct
images $R^{i}f_{\ast}E$ are locally free. Then we may consider the
analytic torsion $T^{BK}(\ov{f}')$ attached to the K\"ahler relative
metrized complex $(\ov{f}',\ov{E},\ov{f_{\ast}E}_{L^{2}})$. Also, we
choose an auxiliary hermitian structure on $g_{\ast}E$. We can
consider the torsion classes $T_{R/2}(\ov{g})$ and $T_{R/2}(\ov{h})$
of the relative metrized complexes $(\ov{g},\ov{E},\ov{g_{\ast}E})$
and $(\ov{h},\ov{g_{\ast}E},\ov{f_{\ast}E}_{L^{2}})$.

We make the
following additional assumption in some particular
situations:
\begin{description}
\item [({\bf *})] The morphisms $g$ and $h$ are K\"ahler fibrations, the higher
  direct images $R^{i}g_{\ast}E$ and $R^{j}h_{\ast}R^{i}g_{\ast}E$ are
  locally free and the auxiliary
  hermitian structure on $g_{\ast}E$ is the $L^{2}$ hermitian
  structure.
\end{description}

When the hypothesis {\bf(*)} is satisfied we denote by
$\ov{h_{\ast}g_{\ast}E}_{L^{2}}$ the $L^{2}$ hermitian structure
attached to the K\"ahler structure on $\ov{h}$ and the $L^{2}$ metric
on $\ov{g_{\ast}E}_{L^{2}}$. Observe that this last structure may
differ from the $L^{2}$ structure on $\ov{f_{\ast}E}_{L^{2}}$.
In this situation we can consider the
  torsion classes $T^{BK}(\ov{g})$ and
  $T^{BK}(\ov{h}')$ attached to
  $(\ov{g},\ov{E},\ov{g_{\ast}E}_{L^{2}})$ and
  $(\ov{h},\ov{g_{\ast}E}_{L^{2}},\ov{h_{\ast}g_{\ast}E}_{L^{2}})$.
  By Proposition \ref{prop:1}, we have the relation
  \begin{displaymath}
    T^{BK}(\ov{h}')=
    T_{R/2}(\ov{h})-
    \widetilde{\ch}(\ov{h_{\ast}g_{\ast}E}_{L^{2}},\ov{f_{\ast}E}_{L^{2}}).
  \end{displaymath}

  The properties of the generalized analytic torsion classes imply immediately:
\begin{corollary} \label{cor:comp_sub}
Under these assumptions, we have the equality
\begin{displaymath}
  T^{BK}(\ov{f}')=T_{-R/2}(\ov{h})
  +\ov{h}_{\ttwist}(T_{-R/2}(\ov{g}))+
  \ov{f}^{\prime}_{\ttwist}(\ch(\ov{E})
  \bullet\widetilde{\Td}_{m}(\ov{f}',\ov{f})).
\end{displaymath}
If in addition the hypothesis {\rm \bf (*)} is satisfied, then we have
\begin{displaymath}
  T^{BK}(\ov{f}')=T^{BK}(\ov{h}')+\ov{h}_{\ttwist}(T^{BK}(\ov{g}))
  +\ov{f}^{\prime}_{\ttwist}(\ch(\ov{E})\bullet
  \widetilde{\Td}_{m}(\ov{f}',\ov{f}))
  +\widetilde{\ch}(\ov{h_{\ast}g_{\ast}E}_{L^{2}},\ov{f_{\ast}E}_{L^{2}}).
\end{displaymath}
\end{corollary}

Since $T_{-R/2}$ extends the theory of analytic torsion classes $T^{BK}$,
we will denote $T_{-R/2}$ by $T^{BK}$ for
arbitrary relative metrized complexes.

\section{Grothendieck duality and analytic torsion}
\label{sec:groth-dual-analyt}

We will study now the compatibility of the analytic
torsion with Grothendieck duality.

\begin{definition}
  Let $\ov {\mathcal{F}}=(\mathcal{F},\ov E\dra \mathcal{F})$ be an
  object of $\oDb(X)$. Then the \emph{rank} of $\ov
  {\mathcal{F}}$ is
  \begin{displaymath}
    \rk(\ov{\mathcal{F}})=\sum_{i}(-1)^{i}\dim(E_{i}).
  \end{displaymath}
  This is just the Euler characteristic of the complex. \emph{The
    determinant}
  of $\ov {\mathcal{F}}$ is the complex
  \begin{displaymath}
    \det (\ov {\mathcal{F}})= \bigotimes _{i} \left(\Lambda ^{\dim E^{i}} \ov
      E^{i}\right)^{(-1)^{i}}[-\rk(\ov{\mathcal{F}})].
  \end{displaymath}
  It consists of a single line bundle concentrated in degree
  $\rk(\ov{\mathcal{F}})$.
\end{definition}

\begin{definition} \label{def:21}
  Let $\ov f\colon X\to Y$ be a morphism in $\ov\Sm_{\ast/\CC}$ of
  relative dimension $e$. The
  \emph{metrized dualizing complex}, is the
  complex  given by
  \begin{displaymath}
    \bomega _{\ov f} =(\det T_{\ov f})^{\vee}.
  \end{displaymath}
  This complex is concentrated in degree $-e$. The underlying object
  of $\Db(X)$ will be denoted by $\bomega _{f}$. If we are interested
  in the dualizing sheaf as a sheaf and not as an element of $\Db(X)$
  we will denote it by $\omega _{f}$ or $\omega _{X/Y}$. Finally, if
  $Y=\Spec \CC$, we will denote $\bomega_{f}$ (respectively
  $\omega _{f}$) by $\bomega_{X}$ (respectively
  $\omega _{X}$).
\end{definition}

\begin{definition} \label{def:20}
  Let $\mathcal{D}^{\ast}(\ast)$ be the Deligne complex associated to
  a Dolbeault complex. The \emph{sign operator} is
  \begin{displaymath}
    \sigma \colon \mathcal{D}^{\ast}(\ast)\longrightarrow
    \mathcal{D}^{\ast}(\ast),\quad \omega \in
  \mathcal{D}^{n}(p) \mapsto \sigma (\omega )=(-1)^{p}\omega.
  \end{displaymath}
 \end{definition}

 The sign operator satisfies the following compatibilities.
 \begin{proposition}\label{prop:13}
   \begin{enumerate}
   \item Let $(\mathcal{D}^{\ast}(\ast),\dd_{\mathcal{D}})$ be a
     Deligne algebra. Then the sign operator is a morphism of differential
     algebras. That is
     \begin{displaymath}
       \dd_{\mathcal{D}}\circ \,\sigma =\sigma \circ
     \dd_{\mathcal{D}},\quad
     \sigma (\omega \bullet \eta)=\sigma (\omega )\bullet\sigma ( \eta).
     \end{displaymath}
   \item Let $\ov {\mathcal{F}}$ be an object of $\oDb(X)$. Then the
     following equalities are satisfied
     \begin{align}
       \label{eq:39}
       \sigma \ch(\ov {\mathcal{F}})&= \ch(\ov
       {\mathcal{F}}^{\vee}),\\
       \label{eq:32}
       \sigma \ch(\det (\ov {\mathcal{F}}))&= \ch(\det (\ov
       {\mathcal{F}})^{\vee})= \ch(\det (\ov {\mathcal{F}}))^{-1},\\
       \label{eq:34}
       \sigma \Td(\ov{\mathcal{F}})&=
       (-1)^{\rk(\ov{\mathcal{F}})}\Td(\ov{\mathcal{F}})\bullet \ch(\det (\ov
       {\mathcal{F}})^{\vee}.
     \end{align}
   \end{enumerate}
 \end{proposition}
 \begin{proof}
   The first statement is clear because if $\omega \in
   \mathcal{D}^{n}(p)$ and $\eta\in \mathcal{D}^{m}(q)$ then
   $\dd_{\mathcal{D}}\omega \in
   \mathcal{D}^{n+1}(p)$ and $\omega \bullet \eta \in
   \mathcal{D}^{n+m}(p+q)$.

   For the second statement, let $\ov E\dra \mathcal{F}$ be the
   hermitian structure of $\mathcal{F}$. Write
   \begin{displaymath}
     \ov E^{+}=\bigoplus_{i\text{ even}} \ov E^{i},\qquad
     \ov E^{-}=\bigoplus_{i\text{ odd}} \ov E^{i}.
   \end{displaymath}
   Since this statement is local on $X$, we can chose trivializations
   of $\ov E^{+}$ and $\ov E^{-}$ over an open subset $U$. Let $H^{+}$
   and $H^{-}$ be the
   matrices of the hermitian metrics on $\ov E^{+}$ and $\ov
   E^{-}$. The curvature matrices of $\ov E^{+}$ and $\ov
   E^{-}$, whose entries are elements of
   $\mathcal{D}^{2}(U,1)$, are
   \begin{displaymath}
     K^{\pm}=K^{\pm}(\ov {\mathcal{F}})= - \ov \partial
     (H^{\pm})^{-1}\partial H^{\pm}.
   \end{displaymath}
   The characteristic forms can be computed
   from the curvature matrix:
   \begin{align*}
     \ch(\ov {\mathcal{F}})&=\tr (\exp(K^{+}))-\tr (\exp(K^{-})),\\
     \ch(\det (\ov{\mathcal{F}}))&=(-1)^{\rk(\ov{\mathcal{F}})}
     \det(\exp(K^{+}))\bullet \det(\exp(K^{-}))^{-1},\\
     \Td(\ov{\mathcal{F}})&=
     \det\left(\frac {K^{+}}{1-\exp(-K^{+})}\right)\bullet
     \det\left(\frac {K^{-}}{1-\exp(-K^{-})}\right)^{-1}.
   \end{align*}
   The sign in the second equation comes from the fact that $\det
   (\ov{\mathcal{F}})$ is concentrated in degree $\rk(\ov{\mathcal{F}})$.
   Therefore, since $\sigma (K^{\pm})=-K^{\pm}=K^{\pm}(\ov
   {\mathcal{F}}^{\vee})$, we have
   \begin{align*}
      \sigma \ch(\ov {\mathcal{F}})&=\sigma \tr (\exp(K^{+}))-
      \sigma \tr (\exp(K^{-}))\\
      &=\tr
      (\exp(K^{+}(\ov{\mathcal{F}}^{\vee})))
      -\tr (\exp(K^{-}(\ov{\mathcal{F}}^{\vee})))=
      \ch(\ov {\mathcal{F}}^{\vee}),\\
     \sigma \ch(\det (\ov{\mathcal{F}}))&=
     \det(\exp(-K^{+}))\bullet \det(\exp(-K^{-}))^{-1}=
     \ch(\det (\ov{\mathcal{F}}))^{-1},\\
     \sigma \Td(\ov{\mathcal{F}})&=
     \det\left(\frac {-K^{+}}{1-\exp(K^{+})}\right)\bullet
     \det\left(\frac {-K^{-}}{1-\exp(K^{-})}\right)^{-1}\\
     &=     \det\left(\frac {K^{+}}{1-\exp(-K^{+})}\right)\bullet
     \det(\exp(-K^{+}))\\
     &\phantom{AAA}\bullet
     \det\left(\frac {K^{-}}{1-\exp(-K^{-})}\right)^{-1} \bullet
     \det(\exp(-K^{-}))^{-1}\\
     &=
       \Td(\ov{\mathcal{F}})\bullet \ch(\det (\ov
       {\mathcal{F}}))^{-1}.
   \end{align*}
 \end{proof}
 \begin{corollary}\label{cor:7}
 Let $[\ov{E}]\in\KA(X)$. Then
 \begin{math}
 	\widetilde{\ch}(\ov{E}^{\vee})=\sigma\widetilde{\ch}(\ov{E}).
 \end{math}
 \end{corollary}
\begin{proof}
Due to Proposition \ref{prop:13}, the assignment sending $[\ov{E}]$ to
$\sigma\widetilde{\ch}(\ov{E})$ satisfies the characterizing
properties of $\widetilde{\ch}$.
\end{proof}

In the particular case of a projective
morphism between smooth complex varieties or, more generally, smooth
varieties over a field, Grothendieck duality takes a very simple form
(see for instance
\cite[\S 3.4]{Huybrechts:FM} and the references therein). If
$\mathcal{F}$ is an object of
$\Db(X)$ and $f\colon X\to Y$ is a projective morphism of  smooth
complex varieties, then there is a natural functorial isomorphism
\begin{equation}\label{eq:31}
  \Rd f_{\ast}(\mathcal{F}^{\vee}\Lotimes \bomega _{f})\cong(\Rd f_{\ast}
  \mathcal{F}) ^{\vee}.
\end{equation}

The compatibility between analytic torsion and Grothendieck duality is
given by the following result.

\begin{theorem-definition}\label{thm-def:T_dual}
  Let $T$ be a theory of generalized analytic torsion classes. Then
  the assignment
  that, to a relative metrized complex $\ov \xi =(\ov
  f,\ov{\mathcal{F}},\ov{\Rd f_{\ast} \mathcal{F}})$, associates the
  class
  \begin{displaymath}
    T^{\vee}(\ov \xi)=
    \sigma  T(\ov f,\ov{\mathcal{F}}^{\vee}\Lotimes \bomega _{\ov f},
    \ov {\Rd f_{\ast} \mathcal{F}}^{\vee})
  \end{displaymath}
  is a theory of generalized analytic torsion classes that we call the
  \emph{theory dual to} $T$.
\end{theorem-definition}
\begin{proof}
  We have to show that, if  $T$ satisfies the conditions of Definition
  \ref{def:genAT}, then the same is true for $T^{\vee}$.  We start with
  the differential equation. Let $e$ be the relative dimension of~$f$.
  \begin{align*}
    \dd_{\mathcal{D}} T^{\vee}(\ov \xi )&=
    \dd_{\mathcal{D}} \sigma  T(\ov
    f,\ov{\mathcal{F}}^{\vee}\Lotimes
    \bomega _{\ov f},
    \ov {\Rd f_{\ast} \mathcal{F}}^{\vee})\\
    &=\sigma  \dd_{\mathcal{D}}  T(\ov
    f,\ov{\mathcal{F}}^{\vee}\Lotimes \bomega _{\ov f},
    \ov {\Rd f_{\ast} \mathcal{F}}^{\vee})\\
    &=\sigma \ch(\ov {\Rd f_{\ast} \mathcal{F}}^{\vee})
    -\sigma f_{\ast}\left[\ch(\ov{\mathcal{F}}^{\vee}\Lotimes
      \bomega _{\ov
        f})\bullet \Td(\ov f)\right]\\
    &=\ch(\ov {\Rd f_{\ast} \mathcal{F}})-
    (-1)^{e}f_{\ast}\left[\sigma \ch(\ov{\mathcal{F}}^{\vee})\bullet
      \sigma (\ch(\det (T_{\ov f})^{\vee})\bullet \Td(T_{\ov
        f}))\right]\\
    &=\ch(\ov {\Rd f_{\ast} \mathcal{F}})-
    f_{\ast}\left[\ch(\ov{\mathcal{F}})\bullet
      \Td(\ov
        f)\right]
  \end{align*}
  The functoriality and the additivity are clear. We next check the
  projection formula. Let $\ov{\mathcal{G}}$ be an object of
  $\oDb(Y)$. Then
  \begin{multline*}
    T^{\vee}(\ov \xi \otimes \ov {\mathcal{G}}) =
    \sigma  T(\ov f,\ov{\mathcal{F}}^{\vee}\Lotimes\Ld
    f^{\ast}\ov
    {\mathcal{G}}^{\vee} \Lotimes \bomega _{\ov f},
    \ov {\Rd f_{\ast} \mathcal{F}}^{\vee}\Lotimes \ov
    {\mathcal{G}}^{\vee})\\
    =\sigma \left( T(\ov f,\ov{\mathcal{F}}^{\vee}\Lotimes
      \bomega _{\ov f},
    \ov {\Rd f_{\ast} \mathcal{F}}^{\vee})\bullet \ch(\ov{\mathcal{G}}^{\vee})
    \right)
    =T^{\vee}(\ov \xi )\bullet \ch(\ov {\mathcal{G}}).
  \end{multline*}
  Finally we check the transitivity. Let $\ov g\colon Y\to Z$ be
  another morphism in $\ov\Sm_{\ast/\CC}$.  By
  the definition of $\ov g\circ \ov f$ we have
  \begin{math}
    \bomega _{\ov g\circ \ov f}=\Ld f^{\ast}\bomega _{\ov g}\Lotimes
    \bomega _{\ov f}.
  \end{math}
  Therefore,
  \begin{multline*}
    \Rd f_{\ast}\left( \mathcal{F}^{\vee}\Lotimes
      \bomega _{g\circ f}
    \right)=
    \Rd f_{\ast}\left(  \mathcal{F}^{\vee}\Lotimes\Ld
      f^{\ast}\bomega _{g}\Lotimes
      \bomega _{f})\right)\\
    =\Rd f_{\ast}\left(  \mathcal{F}^{\vee}\Lotimes
      \bomega _{f})\right)\Lotimes\bomega _{g}
    =(\Rd f_{\ast} \mathcal{F})^{\vee}\Lotimes\bomega
    _{g}.
  \end{multline*}
  On $\Rd f_{\ast}\left( \mathcal{F}^{\vee}\Lotimes
      \bomega _{g\circ f}
    \right)$ we put the hermitian structure of $\ov{\Rd f_{\ast}
      \mathcal{F}}^{\vee}\Lotimes\bomega
    _{\ov g}$. Then we have
    \begin{align*}
      T^{\vee}(\ov g\circ \ov f)&=
      \sigma T(\ov g\circ \ov f,\ov {\mathcal{F}}^{\vee}\Lotimes
      \bomega _{\ov g\circ \ov f},\ov {\Rd (g\circ
        f)_{\ast} \mathcal{F}}^{\vee})\\
      &=\sigma T(\ov g,\ov{\Rd f_{\ast}\mathcal{F}}^{\vee}\Lotimes
      \bomega _{\ov g},\ov {\Rd (g\circ
        f)_{\ast} \mathcal{F}}^{\vee})\\
      &\phantom{AAA}+
      \sigma \ov g_{\ttwist} T(\ov f,\ov{\mathcal{F}}^{\vee}\Lotimes
      \bomega _{\ov f}\Lotimes
      \Ld f^{\ast}\bomega _{\ov g}, \ov{\Rd
        f_{\ast}\mathcal{F}}^{\vee}\Lotimes
      \bomega _{\ov g})\\
      &=T^{\vee}(\ov g,\ov{\Rd f_{\ast}\mathcal{F}},\ov {\Rd (g\circ
        f)_{\ast} \mathcal{F}})\\
      &\phantom{AAA}+\sigma g_{\ast}(T(\ov
      f,\ov{\mathcal{F}}^{\vee}\Lotimes
      \bomega _{\ov f},
    \ov {\Rd f_{\ast} \mathcal{F}}^{\vee})\bullet \ch(\bomega _{\ov g})
    \bullet \Td(\ov g))\\
    &=T^{\vee}(\ov g)+\ov g_{\ttwist} T^{\vee}(\ov f).
  \end{align*}
  Therefore, $T^{\vee}$ satisfies also the transitivity property. Hence it
  is a generalized theory of analytic torsion classes.
\end{proof}

\begin{definition}
  A theory of generalized analytic torsion classes $T$ is called \emph{self-dual} when $T^{\vee}=T$.
\end{definition}

We want to characterize the self-dual theories of generalized analytic
torsion classes.

\begin{theorem}
  The homogeneous theory of generalized analytic
  torsion classes is self-dual.
\end{theorem}
\begin{proof}
  By the uniqueness of the homogeneous theory, it is enough to prove
  that, if $T$ is homogeneous then $T^{\vee}$ is homogeneous. Let $X$
  be a smooth complex variety and let $\ov N$ be a hermitian vector
  bundle of rank $r$ on $X$. Put $P=\mathbb{P}(N\oplus 1)$ and let
  $s\colon X\to
  P$ be the zero section and $\pi \colon P\to X$ the projection. Let
  $\ov Q$ be the tautological quotient bundle with the induced metric
  and $\ov K(s)$ the
  Koszul resolution associated to the section $s$. Since the normal
  bundle $N_{X/P}$ can be identified with $N$, on the map $s$  we can
  consider the hermitian structure given by the hermitian metric on
  $N$. Then $\det \ov Q$ is a complex concentrated in degree $r$. Moreover
  \begin{displaymath}
    s^{\ast} \det \ov Q=\det \ov N=\bomega _{\ov s}.
  \end{displaymath}
  The Koszul resolution satisfies the duality property
  \begin{math}
    \ov K(s)^{\vee}=\ov K(s)\otimes \det \ov Q.
  \end{math}
  The theory $T$ is homogeneous if and only if the class
  \begin{displaymath}
    T(\ov s,\ov {\mathcal{O}}_{X},\ov K(s))\bullet \Td(\ov Q)
  \end{displaymath}
  is homogeneous of bidegree $(2r-1,r)$ in the Deligne complex. Then
  \begin{align*}
    T^{\vee}(\ov s,\ov {\mathcal{O}}_{X},\ov K(s))\bullet \Td(\ov Q)&=
    \sigma T(\ov s,\bomega _{\ov s},\ov K(s)^{\vee})\bullet \Td(\ov
    Q)\\
    &=\sigma T(\ov s,\Ld s^{\ast} \det \ov Q,\ov K(s)\otimes \det \ov
    Q)\bullet \Td(\ov Q)\\
    &=\sigma (T(\ov s,\ov {\mathcal{O}}_{X},\ov K(s))\bullet \ch(\det
    \ov Q))\bullet \Td(\ov Q)\\
    &=\sigma T(\ov s,\ov {\mathcal{O}}_{X},\ov K(s))\bullet \ch(\det
    \ov Q ^{\vee})\bullet \Td(\ov Q)\\
    &=(-1)^{r}\sigma (T(\ov s,\ov {\mathcal{O}}_{X},\ov K(s))\bullet
    \Td(\ov Q))
  \end{align*}
  is homogeneous of bidegree $(2r-1,r)$ in the Deligne complex.
\end{proof}

\begin{proposition}\label{prop:14}
Let
  \begin{displaymath}
    S(x)=\sum _{n=0}^{\infty}a_{n}x^{n}\in \RR[[x]]
  \end{displaymath}
  be a power series in one
  variable with real coefficients. Denote by $S$ the corresponding
  real additive genus and by $T_{S}$ the associated theory of analytic
  torsion classes. Then the dual theory $T^{\vee}_{S}$ has
  corresponding real additive genus $S^{\sigma}(x):=-S(-x)$.
\end{proposition}
\begin{proof}
  Let $\ov{\xi}=(\ov{f},\ov{\mathcal{F}},\ov{\Rd
    f_{\ast}\mathcal{F}})$ be a relative metrized complex. If $e$ is
  the relative dimension of $f$, then we have $\sigma
  f_{\ast}=(-1)^{e}f_{\ast}\sigma$. Then the proposition readily
  follows from the definition of $T^{\vee}_{S}$, the self-duality of
  $T^{h}$ and Proposition \ref{prop:13}.
\end{proof}

We can now characterize the self-dual theories of analytic torsion classes.

\begin{corollary}\label{cor:char_self_dual}
  The theory of analytic torsion classes $T_{S}$ attached to the real
  additive genus $S(x)=\sum_{n\geq 0}a_{n}x^{n}$ is self-dual if and
  only if $a_{n}=0$ for $n$ even.
\end{corollary}
 \begin{proof}
   By the proposition, $T_{S}^{\vee}=T_{S^{\sigma}}$, hence $T$ is
   self-dual if, and only if, $S^{\sigma}=S$. The corollary follows.
\end{proof}

 In particular we recover the following fact, which is well
 known if we restrict to K\"ahler relative metrized complexes.

\begin{corollary}\label{cor:TBK_self-dual}
The theory of analytic torsion classes of Bismut-K\"ohler $T^{BK}$ is
self-dual.
\end{corollary}
\begin{proof}
We just remark that the even coefficients of the $R$-genus vanish
\eqref{eq:81}.
\end{proof}

We now elaborate on an intimate relation between self-duality
phenomena and the analytic torsion of de Rham complexes. Let $f\colon
X\to Y$ be a smooth projective morphism of smooth algebraic varieties,
of relative dimension $e$. Let $\ov{T}_{X/Y}$ denote the vertical
tangent bundle, endowed with a hermitian metric. Write $\ov{f}$
for the corresponding morphism in $\ov{\Sm}_{\ast/\CC}$. On the
locally free sheaves $\Omega_{X/Y}^{p}=\Lambda^{p}\Omega_{X/Y}$ we put
the induced hermitian structures. The metrized de
Rham complex is
\begin{displaymath}
  0\to\ov{\OO}_{X}\overset{0}{\to}\ov{\Omega}_{X/Y}\overset{0}{\to}
  \ov{\Omega}_{X/Y}^{2} \overset{0}{\to}\dots\overset{0}{\to}
  \ov{\Omega}_{X/Y}^{e}\to
  0
\end{displaymath}
with 0 differentials. In fact, we are really considering the de Rham
graded sheaf and converting it into a complex in a trivial way.
 We refer to the corresponding
object of $\oDb(X)$ by $\ov{\Omega}_{X/Y}^{\bullet}$
(\cite[Def.~3.37]{BurgosFreixasLitcanu:HerStruc}). The individual
terms $\ov{\Omega}_{X/Y}^{p}$ will be
considered as complexes concentrated in degree $p$. We then obviously have:

\begin{lemma}\label{lem:van_tor_1}
The objects
$(\ov{\Omega}_{X/Y}^{\bullet})^{\vee}\otimes\bomega_{\ov{f}}$ and
$\ov{\Omega}_{X/Y}^{\bullet}[2e]$ are tightly isomorphic.
\end{lemma}

For every $p$, $q$, the cohomology sheaf $R^{q}
f_{\ast}\Omega_{X/Y}^{p}$ is locally free, because the Hodge numbers
$h^{p,q}$ of the fibers of $f$ (which are projective, hence K\"ahler)
are known to be locally constant. Every stalk of this sheaf is endowed
with the usual $L^2$ metric of Hodge theory. This family of $L^2$
metrics on $R^{q}f_{\ast}\Omega_{X/Y}^{p}$ glue into a smooth
metric. Because the Hodge star operators $\ast$ act by isometries, it
is easily shown that Serre duality becomes an isometry for the $L^2$
structures: the isomorphism
\begin{displaymath}
	(R^{q}f_{\ast}\Omega_{X/Y}^{p})^{\vee}\overset{\sim}{\longrightarrow}
	R^{e-q} f_{\ast}((\Omega_{X/Y}^{p})^{\vee}\otimes\bomega_{f})
	= R^{e-q}f_{\ast}\Omega_{X/Y}^{e-p}
\end{displaymath}
preserves the $L^2$ hermitian structures. For every $p$, let $\ov{\Rd
  f_{\ast}\Omega_{X/Y}^{p}}$ denote the object of $\oDb(Y)$ with the
metric induced by the $L^2$ metrics on its cohomology pieces
(\cite[Def.~3.47]{BurgosFreixasLitcanu:HerStruc}). Here $\Rd f_{\ast}$
stands for the derived direct image. By
\cite[Prop.~3.48]{BurgosFreixasLitcanu:HerStruc}, Grothendieck duality
\begin{displaymath}
  (\Rd f_{\ast}\Omega_{X/Y}^{p})^{\vee}\overset{\sim}{\longrightarrow}
  \Rd f_{\ast}\Omega_{X/Y}^{e-p}[2e]
\end{displaymath}
is a tight isomorphism. Finally, let $[\ov{\Rd
  f_{\ast}\Omega_{X/Y}^{\bullet}}]$ be the object of $\oDb(Y)$
provided by \cite[Def.~3.39]{BurgosFreixasLitcanu:HerStruc}. The next
lemma follows easily from the construction of
\cite[Def.~3.39]{BurgosFreixasLitcanu:HerStruc}.

\begin{lemma}\label{lem:van_tor_2}
Grothendieck duality defines a tight isomorphism  in $\oDb(Y)$ 
\begin{displaymath}
[\ov{\Rd f_{\ast}\Omega_{X/Y}^{\bullet}}]^{\vee}\cong [\ov{\Rd
  f_{\ast}\Omega_{X/Y}^{\bullet}}][2e].   
\end{displaymath}
\end{lemma}

\begin{theorem}\label{thm:char_van_dR}
Let $T$ be a theory of analytic torsion classes. The following assertions are equivalent:
\begin{enumerate}
\item the theory $T$ is self-dual;
\item for every $f$, $\ov{T}_{f}$,
  $\ov{\Omega}_{X/Y}^{\bullet}$ and $[\ov{\Rd
    f_{\ast}\Omega_{X/Y}^{\bullet}}]$ as above and for every
  odd integer $p\geq 1$, the part of bidegree $(2p-1,p)$ (in
  the Deligne complex) of
  $T(\ov{f},\ov{\Omega}_{X/Y}^{\bullet},[\ov{\Rd
    f_{\ast}\Omega_{X/Y}^{\bullet}}])$ vanishes.
\end{enumerate}
\end{theorem}
\begin{proof}
  Assume first of all that $T$ is self-dual. We apply the
  definition of $T^{\vee}$, the self-duality assumption and lemmas
  \ref{lem:van_tor_1} and \ref{lem:van_tor_2}. We find the equality
  \begin{displaymath}
    \begin{split}
      T(\ov{f},\ov{\Omega}_{X/Y}^{\bullet},[\ov{\Rd
        f_{\ast}\Omega_{X/Y}^{\bullet}}])
      &=\sigma T(\ov{f},\ov{\Omega}_{X/Y}^{\bullet}[2e], [\ov{\Rd
        f_{\ast}\Omega_{X/Y}^{\bullet}}][2e])\\
      &=(-1)^{2e}\sigma T(\ov{f},\ov{\Omega}_{X/Y}^{\bullet}, [\ov{\Rd
        f_{\ast}\Omega_{X/Y}^{\bullet}}])\\
      &=\sigma T(\ov{f},\ov{\Omega}_{X/Y}^{\bullet}, [\ov{\Rd
        f_{\ast}\Omega_{X/Y}^{\bullet}}]).
    \end{split}
  \end{displaymath}
  The sign operator $\sigma$ changes the sign of the components of
  bidegree $(2p-1,p)$ for odd $p$. Hence
  $T(\ov{f},\ov{\Omega}_{X/Y}^{\bullet},[\ov{\Rd
    f_{\ast}\Omega_{X/Y}^{\bullet}}])^{(2p-1,p)}$ vanishes for $p\geq 1$
  odd.

  For the converse implication, let $S(x)=\sum_{n\geq 0}a_{n}x^{n}$ be
  the real additive genus attached to $T$ via Theorem
  \ref{thm:comparison_genus}. By Corollary \ref{cor:char_self_dual},
  we have to show that the coefficients $a_{n}$ with $n$ even
  vanish. Let us look at a smooth morphism $f\colon X\to Y$ of relative
  dimension $1$, with an arbitrary metric on $T_{f}$. Then, developing
  the power series of $\ch$ and $\Td$ and taking into account that
  $\Omega ^{1}_{X/Y}=T_{f}^{\vee}=\omega _{X/Y}$,  we compute
  \begin{displaymath}
    f_{\ast}[\ch(\Omega_{X/Y}^{\bullet})\Td(T_{f})S(T_{f})\bullet\mathbf{1}_{1}]
    =\sum_{n\geq
      0}(-1)^{n+1}a_{n}f_{\ast}[c_{1}(\omega_{X/Y})^{n+1}\bullet\mathbf{1}_{1}].
  \end{displaymath}
  Therefore, for $p\geq 1$ odd, we have
  \begin{multline}
    (-1)^{p}a_{p-1}f_{\ast}[c_{1}(\omega_{X/Y})^{p}\bullet\mathbf{1}_{1}]=
    \\(T(\ov{f},\ov{\Omega_{X/Y}^{\bullet}}, [\ov{\Rd
      f_{\ast}\Omega_{X/Y}^{\bullet}}])
    -T^{h}(\ov{f},\ov{\Omega_{X/Y}^{\bullet}}, [\ov{\Rd
      f_{\ast}\Omega_{X/Y}^{\bullet}}]))^{(2p-1,p,p)}= 0.
  \end{multline}
  Hence it is enough that for every odd integer $p\geq 1$, we find a
  relative curve $f\colon X\to Y$ such that
  $f_{\ast}(c_{1}(\omega_{X/Y})^{p})\neq 0$ in the cohomology group
  $H^{2p}(Y,\CC)$. Let $d=p-1$ and choose $Y$ to be a smooth
  projective variety of dimension $d$. Let $L$ be an ample line bundle
  on $Y$ and take $X=\PP(L\oplus\OO_{Y})$. Consider the tautological
  exact sequence
  \begin{displaymath}
    0\longrightarrow\OO(-1)\longrightarrow
    f^{\ast}(L\oplus\OO_{Y})\longrightarrow Q\longrightarrow 0.
  \end{displaymath}
  We easily derive the relations
  \begin{align}
    &\pi^{\ast}c_{1}(L)=c_{1}(Q)-c_{1}(\OO(1))\label{eq:van_dR_1}\\
    &c_{1}(\OO(-1))c_{1}(Q)=0.\label{eq:van_dR_2}
  \end{align}
  Moreover we have
  \begin{equation}\label{eq:van_dR_3}
    c_{1}(\omega_{X/Y})=-c_{1}(Q)-c_{1}(\OO(1)).
  \end{equation}
  From \eqref{eq:van_dR_1}--\eqref{eq:van_dR_3} and because $d=p-1$ is
  even, we compute
  \begin{displaymath}
    c_{1}(\omega_{X/Y})^{d}=c_{1}(Q)^{d}+c_{1}(\OO(1))^{d}=\pi^{\ast}c_{1}(L)^{d}.
  \end{displaymath}
  Therefore we find
  \begin{equation}\label{eq:van_dR_4}
    c_{1}(\omega_{X/Y})^{p}=\pi^{\ast}c_{1}(L)^{d} c_{1}(\omega_{X/Y}).
  \end{equation}
  Finally, $f$ is a fibration in curves of genus $0$, hence
  $f_{\ast}(c_{1}(\omega_{X/Y}))=-2$. We infer that
  \eqref{eq:van_dR_4} leads to
  \begin{displaymath}
    f_{\ast}(c_{1}(\omega_{X/Y})^{p})=-2c_{1}(L)^{d}.
  \end{displaymath}
  This class does not vanish, since $Y$ is projective of dimension $d$
  and $L$ is ample.
\end{proof}

We end with a characterization of the theory of analytic torsion
classes of Bismut-K\"ohler.

\begin{theorem} \label{thm:23}
  The theory of analytic torsion classes of Bismut-K\"ohler $T^{BK}$
  is the unique theory of generalized analytic torsion classes such
  that, for every $\ov{f}\colon X\to Y$, a K\"ahler fibration  in
  $\ov{\Sm}_{\ast/\CC}$, we have the vanishing
  \begin{displaymath}
    T^{BK}(\ov{f}, \ov{\Omega}_{X/Y}^{\bullet},
    [\ov{f_{\ast}\Omega_{X/Y}^{\bullet}}])=0.
  \end{displaymath}
\end{theorem}
\begin{proof}
  That the theory $T^{BK}$ vanishes for de Rham complexes of K\"ahler
  fibrations is a theorem of Bismut \cite{Bismut:deRham}. For the
  uniqueness, let $T$ be a theory of generalized analytic torsion
  classes vanishing on de Rham complexes of K\"ahler
  fibrations. Denote by $S(x)=\sum_{k\geq 0}a_{k}x^{k}$ its
  corresponding genus. If $\ov{f}$ is a relative curve with a
  structure of K\"ahler fibration, then by Theorem
  \ref{thm:comparison_genus}
\begin{equation}\label{eq:van_dR}
  T^{h}(\ov{f},\ov{\Omega}_{X/Y}^{\bullet},[\ov{f_{\ast}\Omega_{X/Y}^{\bullet}}])=
  \sum_{k\geq
    0}(-1)^{k}a_{k}f_{\ast}[c_{1}(\omega_{X/Y})^{k+1}\bullet\mathbf{1}_{1}].
\end{equation}
It is enough to find, for every $k\geq 0$, a relative curve $f$ such
that $f_{\ast}(c_{1}(\omega_{X/Y})^{k+1})$ does not vanish. The
elementary construction in the proof of Theorem \ref{thm:char_van_dR}
works whenever $k$ is even, but one easily sees it fails for $k$
odd. Fortunately, there is an alternative argument. Let $g\geq 2$ and
$n\geq 3$ be integers. Consider the fine moduli scheme of smooth
curves of genus $g$ with a Jacobi structure of level $n$
\cite[Def. 5.4]{Deligne-Mumford}, to be denoted
$\mathcal{M}_{g}^{n}$. Let
$\pi\colon\mathcal{C}_{g}^{n}\to\mathcal{M}_{g}^{n}$ be the universal
curve. An example of K\"ahler fibration structure on $\pi$ is provided
by Teichm\"uller theory (see for instance
\cite[Sec. 5]{Wolpert:chern}). By \cite[Thm. 1]{Faber-Pandharipande},
the tautological class
$\kappa_{g-2}:=\pi_{\ast}(c_{1}(\omega_{\pi})^{g-1})\in
H^{2(g-2)}(\mathcal{M}_{g}^{n},\CC)$ does not vanish. Taking $g=k+2$
and $f=\pi$, we conclude the proof of the theorem.
\end{proof}

We note that in the previous theorem, the existence is provided by
Bismut's theorem. It would be interesting to have a proof of the
existence of a theory satisfying the condition of Theorem \ref{thm:23}
from the axiomatic point of view.

\section{Direct images of hermitian structures}
\label{sec:direct-imag-compl}

As another application of a theory of generalized analytic
torsion classes, we construct direct images of metrized complexes. It
turns out that the natural place to define direct images is not the
category $\oDb(\cdot)$ but a new category $\hDb(\cdot)$ that is
the analogue to the arithmetic $K$-theory of Gillet and
Soul\'e \cite{GilletSoule:vbhm}.

Let $X$ be a smooth complex variety. The fibers of the forgetful
functor $\oDb(X)\to\Db(X)$ have a structure of $\KA(X)$-torsor, for
the action of $\KA(X)$ by translation of the hermitian structures
(\cite[Thm.~3.13]{BurgosFreixasLitcanu:HerStruc}). At the same time,
the group $\KA(X)$ acts on the group
$\oplus_{p}\widetilde{\mathcal{D}}^{2p-1}_{D}(X,p)$ by translation,
via the Bott-Chern character $\widetilde{\ch}$
(\cite[Prop.~4.6]{BurgosFreixasLitcanu:HerStruc}). Observe that all
Bott-Chern classes live in these
groups, as for the analytic torsion classes. It is thus natural to
build a product category over~$\KA(X)$.

\begin{definition} \label{def:hDb}
  Let $S\subset T^{\ast}X_{0}$ be a closed conical subset. We define
  \begin{displaymath}
    \hDb(X,S)=\oDb(X)\times_{\KA(X)}
    \bigoplus_{p}\widetilde{\mathcal{D}}^{2p-1}_{D}(X,S,p)
  \end{displaymath}
  to be the category whose objects are equivalence classes
  $[\ov{\mathcal{F}},\eta]$ of pairs $(\ov{\mathcal{F}},\eta)$
  belonging to
  $\Ob\oDb(X)\times\oplus_{p}\widetilde{\mathcal{D}}^{2p-1}_{D}(X,S,p)$,
  under the equivalence relation
  \begin{displaymath}
    (\ov{\mathcal{F}},\eta)\sim
    (\ov{\mathcal{F}}+[\ov{E}],\eta-\widetilde{\ch}(\ov E))
  \end{displaymath}
  for $[\ov E]\in\KA(X)$, and with morphisms
  \begin{displaymath}
    \Hom_{\hDb(X)}([\ov{\mathcal{F}},\eta],[\ov{\mathcal{G}},\nu])=
    \Hom_{\Db(X)}(\mathcal{F},\mathcal{G}).
  \end{displaymath}
\end{definition}

If $S\subset T$ are closed conical subsets of
$T^{\ast}X_{0}$, then $\hDb(X,S)$ is naturally a full subcategory of
$\hDb(X,T)$.

In the sequel, we extend the main operations in $\Db(X)$ to the
categories $\hDb(X,S)$. In particular, we use the theory of generalized
analytic torsion classes to construct push-forward morphisms attached
to morphisms in $\ov{\Sm}_{\ast/\CC}$.

The category $\hDb(X,S)$ has a natural additive structure. More
generally, if $S,T$ are closed conical subsets of $T^{\ast}X_{0}$,
then there is an obvious addition functor
\begin{displaymath}
  \hDb(X,S)\times\hDb(X,T)\overset{\oplus}{\longrightarrow}\hDb(X,S\cup T).
\end{displaymath}
The object $[\ov{0},0]$ is a neutral element for this
operation. If $S + T$ is disjoint with the zero section in
$T^{\ast}X$, then there is a product defined by the functor
\begin{equation}\label{eq:hDb_1}
  \begin{split}
    \Ob\hDb(X,S)\times\Ob\hDb(X,T)& \overset{\otimes}{\longrightarrow}
    \Ob\hDb(X,(S+T)\cup S\cup T)\\
    ([\ov{\mathcal{F}},\eta],[\ov{\mathcal{G}},\nu]) &\longmapsto
    [\ov{\mathcal{F}}\otimes\ov{\mathcal{G}},
    \ch(\ov{\mathcal{F}})\bullet\nu+
    \eta\bullet\ch(\ov{\mathcal{G}})+\dd_{D}\eta\bullet\nu]
  \end{split}
\end{equation}
and the obvious assignment for morphisms. This product is commutative
up to natural isomorphism. It induces on $\hDb(X,\emptyset)$ a
structure of commutative and associative ring category. Also,
$[\ov{\mathcal{O}}_{X},0]$ is a unity object for the product
structure. More generally, the category $\hDb(X,S)$ becomes a left and
right $\hDb(X,\emptyset)$ module. Under the same assumptions on $S$,
$T$ we may define an internal $\Hom$. For this, let
$[\ov{\mathcal{F}},\eta]\in\Ob\hDb(X,S)$ and
$[\ov{\mathcal{G}},\nu]\in\Ob\hDb(X,T)$. Then we put
\begin{displaymath}
  \underline{\Hom}([\ov{\mathcal{F}},\eta],[\ov{\mathcal{G}},\nu])
  =[\underline{\Hom}(\ov{\mathcal{F}},\ov{\mathcal{G}}),(\sigma\ch(\ov{\mathcal{F}}))\bullet\nu
  +(\sigma\eta)\bullet\ch(\ov{\mathcal{G}})+(\dd_{D}\sigma\eta)\bullet\nu],
\end{displaymath}
where we recall that $\sigma$ is the sign operator (Definition
\ref{def:20}). Using Corollary \ref{cor:7}, it is easily seen this is
well defined. In particular, we put
\begin{displaymath}
  [\ov{\mathcal{F}},\eta]^{\vee}:=\underline{\Hom}([\ov{\mathcal{F}},\eta],[\ov{\OO}_{X},0])=
  [\ov{\mathcal{F}}^{\vee},\sigma\eta].
\end{displaymath}

The shift $[1]$ on $\oDb(X)$ induces a well defined shift functor on
$\hDb(X,S)$, whose action on objects is
\begin{displaymath}
  [\ov{\mathcal{F}},\eta][1]=[\ov{\mathcal{F}}[1],-\eta].
\end{displaymath}

There is a Chern character
\begin{displaymath}
    \ch:\Ob\hDb(X,S)\longrightarrow\bigoplus_{p}\widetilde{\mathcal{D}}_{D}^{2p}(X,S,p),\quad
    [\ov{\mathcal{F}},\eta]\longmapsto\ch(\ov{\mathcal{F}})+\dd_{D}\eta,
\end{displaymath}
which is well defined because $\dd_{D}\widetilde{\ch}(\ov E)=\ch(\ov
E)$ for $[\ov E]\in\KA(X)$. The Chern character is additive and
compatible with the product structure:
\begin{displaymath}
  \ch([\ov{\mathcal{F}},\eta]\otimes[\ov{\mathcal{G}},\nu])=
  \ch([\ov{\mathcal{F}},\eta])\bullet\ch([\ov{\mathcal{G}},\nu]).	
\end{displaymath}
Notice the relations
\begin{align*}
  &\ch([\ov{\mathcal{F}},\eta]^{\vee})=\sigma\ch([\ov{\mathcal{F}},\eta]),\\
  &\ch([\ov{\mathcal{F}},\eta][1])=-\ch([\ov{\mathcal{F}},\eta]).
\end{align*}

We may also define Bott-Chern classes for isomorphisms and
distinguished triangles. Let
$\widehat{\varphi}\colon[\ov{\mathcal{F}},\eta]\dra
[\ov{\mathcal{G}},\nu]$ be an isomorphism in $\hDb(X,S)$, whose
underlying morphism in $\Db(X)$ we denote $\varphi$. While the class
$\widetilde{\ch}(\varphi\colon\ov{\mathcal{F}}\dra\ov{\mathcal{G}})$
depends on the representatives $(\ov{\mathcal{F}},\eta)$,
$(\ov{\mathcal{G}},\nu)$, the class
\begin{displaymath}
  \widetilde{\ch}(\widehat{\varphi})
  :=\widetilde{\ch}(\varphi\colon\ov{\mathcal{F}}\dra\ov{\mathcal{G}})
  +\nu-\eta
\end{displaymath}
is well defined.

\begin{lemma}\label{lemma:hDb_1}
  Let $\widehat{\varphi}:[\ov{\mathcal{F}},\eta]\dra
  [\ov{\mathcal{G}},\nu]$ be an isomorphism in $\hDb(X,S)$, with
  underlying morphism $\varphi$ in $\Db(X)$. Then, the following
  conditions are equivalent:
  \begin{enumerate}
  \item there exists $[\ov E]\in\KA(X)$ such that $\varphi$ induces a
    tight isomorphism between $\ov{\mathcal{F}}+[\ov E]$ and
    $\ov{\mathcal{G}}$, and $\nu=\eta-\widetilde{\ch}(\ov E)$;
  \item $\widetilde{\ch}(\widehat{\varphi})=0$.
  \end{enumerate}
\end{lemma}
\begin{proof}
  This is actually a tautology. Because $\KA(X)$ acts freely and
  transitively on the possible hermitian structures on $\mathcal{F}$,
  there exists a unique $[\ov E]\in\KA(X)$ such that
  $\ov{\mathcal{F}}+[\ov E]$ is tightly isomorphic to
  $\ov{\mathcal{G}}$ via the morphism $\varphi$. Then we have
  \begin{displaymath}
    \widetilde{\ch}(\widehat{\varphi})=\widetilde{\ch}(\ov E)+\nu-\eta.
  \end{displaymath}
  The lemma follows.
\end{proof}

\begin{definition}
  Let $\widehat{\varphi}$ be an isomorphism in $\hDb(X,S)$. We say
  that $\widehat{\varphi}$ is tight if the equivalent conditions of
  Lemma \ref{lemma:hDb_1} are satisfied.
\end{definition}
In particular, if $\varphi\colon
\ov{\mathcal{F}}\dra\ov{\mathcal{G}}$ is a tight isomorphism in
$\oDb(X)$, then $\varphi$ induces a tight isomorphism
$[\ov{\mathcal{F}},\eta]\dra[\ov{\mathcal{G}},\nu]$ if and only if
$\eta=\nu$.

The following lemma provides an example involving the notion of
tight~isomorphism.
\begin{lemma}\label{lemma:hDb_2}
  Let $[\ov{\mathcal{F}},\eta]\in\hDb(X,S)$ and
  $[\ov{\mathcal{G}},\nu]\in\hDb(X,T)$. Assume that $S+T$ does not
  cross the zero section. Then there is a functorial tight isomorphism
  \begin{displaymath}
    [\ov{\mathcal{F}},\eta]^{\vee}\otimes [\ov{\mathcal{G}},\nu]
    \cong\underline{\Hom}([\ov{\mathcal{F}},\eta],[\ov{\mathcal{G}},\nu]).
  \end{displaymath}
\end{lemma}

Assume now given a distinguished triangle
\begin{displaymath}
  \widehat{\tau}\colon\quad [\ov{\mathcal{F}},\eta]\dra [\ov{\mathcal{G}},\nu]\dra [\ov{\mathcal{H}},\mu]\dra [\ov{\mathcal{F}},\eta][1].
\end{displaymath}
Let $\ov{\tau}$ denote the distinguished triangle
$\ov{\mathcal{F}}\dra\ov{\mathcal{G}}\dra\ov{\mathcal{H}}\dra$ in
$\oDb(X)$. Then we put
\begin{displaymath}
  \widetilde{\ch}(\widehat{\tau})=
  \widetilde{\ch}(\ov{\mathcal{\tau}})+\eta-\nu+\mu.
\end{displaymath}
By \cite[Thm.~3.33~(vii)]{BurgosFreixasLitcanu:HerStruc}, this class
does not depend on the representatives and is thus well defined.

We study now the functoriality of $\hDb(X,S)$ with respect to
inverse and direct images.
Let $f\colon X\to Y$ be a morphism of smooth complex varieties. Let
$T\subset T^{\ast}Y_{0}$ be a closed conical subset disjoint with
$N_{f}$. The action of the left inverse image functor on objects
is
\begin{displaymath}
    \Ld f^{\ast}\colon\Ob\hDb(Y,T)\longrightarrow\Ob\hDb(X,f^{\ast}T),\qquad
    [\ov{\mathcal{F}},\eta]\longmapsto [\Ld
    f^{\ast}\ov{\mathcal{F}},f^{\ast}\eta].
\end{displaymath}
That this assignment is well defined amounts to the functoriality of
$\widetilde{\ch}$.

Let $\ov{f}$ be a morphism in the category $\ov\Sm_{\ast/\CC}$. The
definition of a direct image functor attached to $\ov{f}$ depends upon
the choice of a theory of generalized analytic torsion classes. Let
$T$ be such a theory. Then we define a functor $\Rd\ov{f}_{\ast}$
whose action on objects is
\begin{equation}\label{eq:hDb_2}
  \begin{split}
    \Rd\ov{f}_{\ast}\colon\Ob\hDb(X,S)&\longrightarrow\Ob\hDb(Y,f_{\ast}S)\\
    [\ov{\mathcal{F}},\eta]&\longmapsto [\ov{\Rd
      f_{\ast}\mathcal{F}},\ov{f}_{\ttwist}(\eta)-T(\ov{f},\ov{\mathcal{F}},\ov{\Rd
      f_{\ast}\mathcal{F}})],
  \end{split}
\end{equation}
where $\ov{\Rd f_{\ast}\mathcal{F}}$ is an arbitrary choice of
hermitian structure on $\Rd f_{\ast}\mathcal{F}$. By the anomaly
formulas, this definition does not depend on the representative
$(\ov{\mathcal{F}},\eta)$ nor on the choice of hermitian structure on
$\ov{\Rd f_{\ast}\mathcal{F}}$.

\begin{theorem}\label{thm:hDb_1}
  Let $\ov{f}:X\to Y$ and $\ov{g}:Y\to Z$ be morphisms in
  $\ov{\Sm}_{\ast/\CC}$. Let $S\subset T^{\ast}X_{0}$ and $T\subset
  T^{\ast}Y_{0}$ be closed conical subsets.
  \begin{enumerate}
  \item Let $[\ov{\mathcal{F}},\eta]\in\Ob\hDb(X,S)$. Then there is a
    functorial tight isomorphism
    \begin{displaymath}
      \Rd(\ov{g}\circ \ov{f})_{\ast}([\ov{\mathcal{F}},\eta])\cong \Rd \ov{g}_{\ast}\Rd \ov{f}_{\ast}([\ov{\mathcal{F}},\eta]).
    \end{displaymath}
  \item (Projection formula) Assume that $T\cap N_{f}=\emptyset$ and
    that $T+f_{\ast}S$ does not cross the zero section of
    $T^{\ast}Y$. Let $[\ov{\mathcal{F}},\eta]\in\Ob\hDb(X,S)$ and
    $[\ov{\mathcal{G}},\nu]\in\Ob\hDb(Y,T)$. Then there is a
    functorial tight isomorphism
    \begin{displaymath}
      \Rd\ov{f}_{\ast}([\ov{\mathcal{F}},\eta] \otimes\Ld
      f^{\ast}[\ov{\mathcal{G}},\nu])
      \cong \Rd\ov{f}_{\ast}[\ov{\mathcal{F}},\eta]\otimes [\ov{\mathcal{G}},\nu]
    \end{displaymath}
    in $\hDb(Y,W)$, where
    \begin{math}
      W=f_{\ast}(S + f^{\ast}T)\cup f_{\ast}S\cup f_{\ast}f^{\ast}T.
    \end{math}
  \item (Base change) Consider a cartesian diagram
    \begin{displaymath}
      \xymatrix{
        X'\ar[r]^{h'}\ar[d]_{f'}	&X\ar[d]^{f}\\
        Y'\ar[r]^{h}	&Y.
      }
    \end{displaymath}
    Suppose that $f$ and $h$ are transverse and that $N_{h'}$ is
    disjoint with $S$. Equip $f'$ with the hermitian structure induced
    by the natural isomorphism $\Ld h^{\ast} T_{f}\dra T_{f'}$. Let
    $[\ov{\mathcal{F}},\eta]\in\Ob\hDb(X,S)$. Then there is a
    functorial tight isomorphism
    \begin{displaymath}
      h^{\ast}\ov{f}_{\ast}[\ov{\mathcal{F}},\eta]\cong \ov{f}'_{\ast}{h'}^{\ast}[\ov{\mathcal{F}},\eta]
    \end{displaymath}
    in $\hDb(Y',f'_{\ast}{h'}^{\ast}S)$.
  \end{enumerate}
\end{theorem}
\begin{proof}
  The first and the second assertions follow from Proposition
  \ref{prop:12}, the transitivity and
  the projection formula for $T$. For the third item, one uses the
  functoriality of the analytic torsion classes and
  Proposition \ref{prop:11}.
\end{proof}
We close this section with an extension of Grothendieck duality to
$\hDb$. Let $\ov{f}:X\to Y$ be a
morphism is $\ov{\Sm}_{\ast/\CC}$. To enlighten notations, we denote
by $\bomega_{\ov{f}}$ the object $[\bomega_{\ov{f}},0]$ in
$\hDb(X,\emptyset)$ (Definition \ref{def:21}). Suppose given a closed
conical subset $T\subset T^{\ast}Y_{0}$ such that $T\cap
N_{f}=\emptyset$. Then we define the functor $\ov{f}^{!}$ whose action
on objects is
\begin{displaymath}
    \ov{f}^{!}:\Ob\hDb(Y,T)\longrightarrow\Ob\hDb(X,f^{\ast}T),\qquad
    [\ov{\mathcal{F}},\eta]\longmapsto
    f^{\ast}[\ov{\mathcal{F}},\eta]\otimes\bomega_{\ov{f}}.
\end{displaymath}
Observe the equality
\begin{equation}\label{eq:hDb_4}
  [\ov{\mathcal{G}},\nu]\otimes\bomega_{\ov{f}}
  =[\ov{\mathcal{G}}\otimes\bomega_{\ov{f}},\nu\bullet\ch(\bomega_{\ov{f}})].
\end{equation}
Now fix a theory of generalized analytic torsion classes. To the
morphism $\ov{f}$ we have attached the direct image functor
$\Rd\ov{f}_{\ast}$. We denote by
$\Rd\ov{f}^{\vee}_{\ast}$ the direct image functor associated to
$\ov{f}$ and the dual theory (Theorem
Definition \ref{thm-def:T_dual}).

\begin{theorem}[Grothendieck duality for $\hDb$]\label{thm:groth_dual}
  Let $\ov{f}:X\to Y$ be a morphism in $\ov{\Sm}_{\ast/\CC}$. Let
  $S\subset T^{\ast}X_{0}$ and $T\subset T^{\ast}Y_{0}$ be closed
  conical subsets such that $T\cap N_{f}=\emptyset$ and $T+f_{\ast}S$
  is disjoint with the zero section. Let
  $[\ov{\mathcal{F}},\eta]\in\Ob\hDb(X,S)$ and
  $[\ov{\mathcal{G}},\sigma]\in\Ob\hDb(Y,T)$. Then there is a
  functorial tight isomorphism
  \begin{displaymath}
    \underline{\Hom}(\Rd \ov{f}_{\ast}[\ov{\mathcal{F}},\eta],[\ov{\mathcal{G}},\nu])
    \cong \Rd\ov{f}^{\vee}_{\ast}\underline{\Hom}([\ov{\mathcal{F}},\eta], \ov{f}^{!}[\ov{\mathcal{G}},\nu])
  \end{displaymath}
  in $\hDb(Y,W)$, where
  \begin{math}
    W=f_{\ast}(S + f^{\ast}T)\cup f_{\ast}S\cup f_{\ast}f^{\ast}T.
  \end{math}
\end{theorem}
In particular, we have
\begin{equation}\label{eq:hDb_3}
  (\Rd
  \ov{f}_{\ast}[\ov{\mathcal{F}},\eta])^{\vee}\cong
  \Rd\ov{f}^{\vee}_{\ast}([\ov{\mathcal{F}},\eta]^{\vee}
  \otimes\bomega_{\ov{f}}).
\end{equation}
\begin{proof}
  By Lemma \ref{lemma:hDb_2} and Proposition \ref{thm:hDb_1}, we are
  reduced to establish the functorial tight isomorphism
  (\ref{eq:hDb_3}). The proof follows readily from the definitions,
  Grothendieck duality and the following two observations. First of
  all, if $T$ is the theory of analytic torsion classes, then by the
  very definition of $T^{\vee}$ we find
  \begin{displaymath}
    \sigma T(\ov{f},\ov{\mathcal{F}},\ov{\Rd f_{\ast}\mathcal{F}})=
    T^{\vee}(\ov{f},\ov{\mathcal{F}}^{\vee}\otimes\bomega_{\ov{f}},
    \ov{\Rd f_{\ast}(\mathcal{F}^{\vee}\otimes\bomega_{f})}),
  \end{displaymath}
  where the metric on $\ov{\Rd
    f_{\ast}(\mathcal{F}^{\vee}\otimes\bomega_{f})}$ is chosen so that
  Grothendieck duality provides a tight isomorphism
  \begin{displaymath}
    \ov{\Rd f_{\ast}\mathcal{F}}^{\vee}\cong \ov{\Rd f_{\ast}(\mathcal{F}^{\vee}\otimes\bomega_{f})}.
  \end{displaymath}
  Secondly, for direct images of currents, we compute
  \begin{displaymath}
      \sigma\ov{f}_{\ttwist}(\eta)=\sigma
      f_{\ast}(\eta\bullet\Td(T_{\ov{f}}))
      =(-1)^{e}f_{\ast}(\sigma\eta\bullet\sigma\Td(T_{\ov{f}}))
      =f_{\ast}(\sigma\eta\bullet\ch(\bomega_{\ov{f}})\bullet\Td(T_{\ov{f}})).
  \end{displaymath}
  Here $e$ is the relative dimension of $f$, and to derive the last
  equality we appeal to Proposition \ref{prop:13}. To conclude, we
  recall equation \eqref{eq:hDb_4}.
\end{proof}
\begin{corollary}
  Let $T$ be a self-dual theory of generalized analytic torsion
  classes.
  \begin{enumerate}
  \item Then there is a functorial isomorphism
    \begin{math}
      (\Rd \ov{f}_{\ast}[\ov{\mathcal{F}},\eta])^{\vee}\cong
      \Rd\ov{f}_{\ast}([\ov{\mathcal{F}},\eta]^{\vee}
      \otimes\bomega_{\ov{f}}).
    \end{math}
  \item If the hermitian structure of $\ov{f}$ comes from chosen
    metrics on $T_{X}$, $T_{Y}$ and $\bomega_{X}$, $\bomega_{Y}$ are
    equipped with the induced metrics, then we have a commutative
    diagram
    \begin{displaymath}
      \xymatrix{
        \hDb(X,S)\ar[d]_{\ov{f}_{\ast}}\ar[r]^{(\cdot)^{\vee}\otimes\ov{\bomega}_{X}	}
        &\hDb(X,S)\ar[d]^{\ov{f}_{\ast}}\\
        \hDb(Y,f_{\ast}S)\ar[r]^{(\cdot)^{\vee}\otimes\ov{\bomega}_{Y}}
        &\hDb(Y,f_{\ast}S).
      }
    \end{displaymath}
  \end{enumerate}
\end{corollary}
\begin{proof}
  The first claim is immediate from Theorem \ref{thm:groth_dual}. The
  second item follows from the first one and the projection formula
  (Proposition \ref{thm:hDb_1}).
\end{proof}

\section{Analytic torsion for degenerating families of curves}
\label{sec:analyt-tors-degen}
As a second example of application of the theory developed in this article,
we describe the singularities of the analytic torsion for degenerating
families of curves. The results we prove are particular instances of
those obtained by Bismut-Bost \cite{Bismut-Bost}, Bismut
\cite{Bismut:degeneracy} and Yoshikawa \cite{Yoshikawa}. Although the
methods of this section can be extended to recover the results of
Yoshikawa in \cite{Yoshikawa}, for simplicity, we will restrict
ourselves to fibrations in curves over a curve.

In fact, our proof is not very different from the one in \cite{Bismut:degeneracy}
and \cite{Yoshikawa}. For instance, one of the
main ingredients of the proof of the results in \cite{Bismut:degeneracy}
and \cite{Yoshikawa} is the Bismut-Lebeau immersion formula. Our
approach implicitly uses Bismut's generalization of the immersion
formula, encoded in the
existence of analytic torsion theories for arbitrary projective
morphisms. We expect that the techniques of this section can be used to
generalize the above results to situations more general
than the ones considered by Yoshikawa.

Let $S$ be a smooth complex curve and $f\colon X\to S$ a projective
morphism of smooth complex varieties, whose fibers are reduced curves
with at most ordinary double singular points. We assume that $f$ is
generically smooth. Following Bismut-Bost
\cite[Sec. 2(b)]{Bismut-Bost}, we call such a family an
f.s.o. (\emph{famille \`a singularit\'es ordinaires}). The singular
locus of $f$, to be denoted $\Sigma$, is a zero dimensional reduced
closed subset of $X$. Its direct image $\Delta=f_{\ast}(\Sigma)$ is
the Weil divisor
\begin{displaymath}
  \Delta=\sum_{p\in S}n_{p}p,
\end{displaymath}
where $n_{p}$ is the number of singular points of the fiber
$f^{-1}(p)$. We will abusively identify $\Delta$ with its
support. With these notations, we put
$V=S\setminus\Delta$. Locally for the analytic topology, the morphism
$f$ can be written in complex coordinates either as
$f(z_{0},z_{1})=z_{0}$ or $f(z_{0},z_{1})=z_{0}z_{1}$
\cite[Sec. 3(a)]{Bismut-Bost}. In the second case, the point of
coordinates $(z_{0},z_{1})=(0,0)$ belongs to the singular locus
$\Sigma$.

For a vector bundle $F$ over $X$, let $\PP(F)$ be the projective
space of lines in $F$. The differential $df\colon T_{X}\to
f^{\ast}T_{S}$ induces a section $\OO_{X}\to\Omega_{X}\otimes
f^{\ast}T_{S}$. Because $f$ is smooth over $X\setminus\Sigma$, this
section does not vanish on $X\setminus\Sigma$. Therefore there is an
induced map
\begin{displaymath}
  \mu\colon X\setminus\Sigma\longrightarrow \PP(\Omega_{X}\otimes
  f^{\ast}T_{S})\cong\PP(\Omega_{X}),
\end{displaymath}
called the \emph{Gauss map}. Notice that this map was already used in
\cite{Bismut:degeneracy} and \cite{Yoshikawa}.

We next study the blow-up $\widetilde{X}=\text{Bl}_{\Sigma}(X)$ of $X$
at $\Sigma$ and relate it to the Gauss map. Let $\pi\colon
\widetilde{X}\to X$ be the natural projection. Let $E$ be the
exceptional divisor of $\pi$,
\begin{displaymath}
  E=\bigsqcup_{p\in\Sigma}E_{p},\quad E_{p}\cong\PP(T_{p}X),
\end{displaymath}
with the reduced scheme structure.  For every $p\in\Sigma$, there is
an identification $T_{p}X\cong\Omega_{X,p}$ provided by the hessian of
$f$, which is a non-degenerate bilinear form on $T_{p}X$. The local
description of the blow-up at a point implies:

\begin{lemma}
  There is a commutative diagram
  \begin{displaymath}
    \xymatrix{
      E_{p}=\PP(T_{p}X)\ar[r]^{\hspace{0.3cm}\sim}\ar@{^{(}->}[d]	&\PP(\Omega_{X,p})\ar@{^{(}->}[d]\\
      \widetilde{X}\ar[r]^{\widetilde{\mu}}\ar[d]_{\pi}	&\PP(\Omega_{X})\ar[dl]^{p}\\
      X	&X\setminus\Sigma.\ar[u]_{\mu}\ar@{_{(}->}[l]
    }
  \end{displaymath}
  Denote by $\OO(-1)$ the tautological divisor either on
  $\PP(\Omega_{X})$ or on $E_{p}$. Then there is a natural isomorphism
  \begin{math}
    \widetilde{\mu}^{\ast}\OO(-1)\mid_{E_{p}}\cong\OO(-1).
  \end{math}
\end{lemma}

Consider now the short exact sequence of vector bundles on
$\PP(\Omega_{X})$
\begin{displaymath}
  0\to\OO(-1)\to p^{\ast}\Omega_{X}\to Q\to 0,
\end{displaymath}
where $Q$ is the universal quotient bundle. Observe that $Q$ is of
rank 1. The dual exact sequence is
\begin{displaymath}
  0\to U\to p^{\ast}T_{X}\to\OO(1)\to 0,
\end{displaymath}
$U$ being the universal vector subsheaf. We denote by $\eta $
the induced exact sequence on~$\widetilde X$
\begin{equation} \label{eq:87}
\eta\colon  0\to \widetilde{\mu}^{\ast} U\to
  \pi^{\ast}T_{X}\to\widetilde{\mu}^{\ast} \OO(1)\to 0,
\end{equation}
From \eqref{eq:87}
and the definition $\omega_{X/S}=\omega_{X}\otimes f^{\ast}T_S$, we
derive a natural isomorphism
\begin{equation}\label{eq:bb_1}
  \widetilde{\mu}^{\ast}U\otimes\pi^{\ast}\omega_{X/S}\cong
  \widetilde{\mu}^{\ast}\OO(-1)
  \otimes\widetilde{f}^{\ast}T_S.
\end{equation}
\begin{lemma}\label{lemma:bb_1}
  We have
  \begin{equation}\label{eq:bb_1bis}
    \widetilde{\mu}^{\ast}\OO(-1)\otimes\widetilde{f}^{\ast}T_S=\OO(E).
  \end{equation}
\end{lemma}
\begin{proof}
  First of all we observe that
  $\widetilde{\mu}^{\ast}U\otimes\pi^{\ast}\omega_{X/S}$ is trivial on
  the open $W=\widetilde{X}\setminus E$. Indeed, by construction of
  the Gauss map we have
  \begin{displaymath}
    \widetilde{\mu}^{\ast}U\mid_{W}
    =\ker(df\colon T_X\to f^{\ast}T_S)\mid_{W}
    =\omega_{X/S}^{\vee}\mid W.
  \end{displaymath}
  Hence by equation \eqref{eq:bb_1} we can write
  \begin{displaymath}
    \widetilde{\mu}^{\ast}\OO(-1)\otimes\widetilde{f}^{\ast}T_S
    =\OO(\sum_{p\in\Sigma}m_{p}E_{p}).
  \end{displaymath}
  To compute the multiplicities $m_{p}$ we use that
  $\widetilde{\mu}^{\ast}\OO(-1)\mid_{E_{p}}=\OO(-1)$, $(E_{p}\cdot
  \widetilde{f}^{\ast}T_S)=0$ and $(E_{p}\cdot E_{p})=-1$:
  \begin{displaymath}
    -m_{p}=\deg(\widetilde{\mu}^{\ast}\OO(-1)\otimes
    \widetilde{f}^{\ast}T_S)\mid_{E_{p}}
    =-1+0=-1.
  \end{displaymath}
  The lemma follows.
\end{proof}
Later we will need the commutative diagram of exact sequences
\begin{equation}\label{eq:bb_2}
  \xymatrix{
    \eta\mid_{W} \colon	&0\ar[r]	&\widetilde{\mu}^{\ast}
    U\mid_{W}\ar[r]\ar[d]^{\alpha}
    &T_X\mid_{W}\ar[r]\ar[d]^{\beta}
    &\widetilde{\mu}^{\ast}\OO(1)\mid_{W}\ar[r]\ar[d]^{\gamma}	&0\\
    \varepsilon \colon	&0\ar[r]
    &\omega_{X/S}^{\vee}\mid_{W}\ar[r]	&T_X\mid_{W}\ar[r]
    &f^{\ast}T_S\mid_{W}\ar[r]	&0.
  }
\end{equation}
After the identification
$\widetilde{\mu}^{\ast}\OO(-1)\otimes\widetilde{f}^{\ast}T_S=\OO(E)$
provided by the lemma, the morphism $\gamma$ is the restriction to $W$
of the natural inclusion
$\widetilde{\mu}^{\ast}\OO(1)\to\widetilde{\mu}^{\ast}\OO(1)\otimes\OO(E)$. This
fact will be used below.

We now proceed to introduce the hermitian vector bundles and the
analytic torsion classes we aim to study. We fix a theory of
generalized analytic torsion classes $T$.

Let $f\colon X\to S$, $\widetilde{f}\colon \widetilde{X}\to S$ be f.s.o. as
above. Recall that we write $W=X\setminus\Sigma=\widetilde{X}\setminus E$
and $V=S\setminus\Delta$, so that $f^{-1}(V)\subset W$. We endow the
tangent spaces $T_{X}$ and $T_{S}$ with smooth hermitian metrics. We
will denote by $\ov{f}$ the corresponding morphism in the category
$\ov{\Sm}_{\ast/\CC}$. On the open subset $W$, there
is a quasi-isomorphism
\begin{displaymath}
  \omega_{X/S}^{\vee}\mid_{W}=\bomega_{X/S}^{\vee}[1]\mid_{W}\to T_{f}
\end{displaymath}
induced by the identification
$\omega_{X/S}^{\vee}\mid_{W}=\ker(T_{X}\mid_{W}\to f^{\ast}T_{S})$. On
$\omega_{X/S}^{\vee}\mid_{W}$, and in particular on
$\omega_{f^{-1}(V)/V}^{\vee}$, we will put the metric induced by
$\ov{T_{X}}\mid_{W}$. We will write $\ov{f}'\colon f^{-1}(V)\to V$ for the
corresponding morphism in $\ov{\Sm}_{\ast/\CC}$. Observe that the
restriction of $f$ to $W$, and hence to $f^{-1}(V)$, may be identified
with the restriction of $\widetilde{f}$. Let $\ov{\mathcal{F}}$ be an
object in $\oDb(X)$
and fix a hermitian structure on $\Rd f_{\ast}\mathcal{F}$. Then we
consider the relative metrized complexes
\begin{displaymath}
  \ov{\xi}=(\ov{f},\ov{\mathcal{F}},\ov{\Rd f_{\ast}\mathcal{F}}),\qquad
  \ov{\xi}'=(\ov{f}',\ov{\mathcal{F}}\mid_{f^{-1}(V)},\ov{\Rd
    f_{\ast}\mathcal{F}}\mid_{V}),
\end{displaymath}
and the corresponding analytic torsion classes
\begin{displaymath}
  T(\ov{\xi})\in\bigoplus_{p}\widetilde{\mathcal{D}}_{D}^{2p-1}(S,N_{f},p),\qquad
  T(\ov{\xi}')
  \in\bigoplus_{p}\widetilde{\mathcal{D}}_{D}^{2p-1}(V,\emptyset,p).
\end{displaymath}
By the functoriality of analytic torsion classes and the anomaly
formulas, we have
\begin{equation}\label{eq:bb_3}
  T(\ov{\xi}')=T(\ov{\xi})\mid_{V}-
  \ov{f}_{\ttwist}[\ch(\ov{\mathcal{F}}\mid_{f^{-1}(V)})
  \widetilde{\Td}_{m}(\ov{\varepsilon}\mid_{f^{-1}(V)})].
\end{equation}
Here $\ov{\varepsilon}$ is the exact sequence in \eqref{eq:bb_2}, with
the hermitian metrics we have just defined. From now on we will 
omit the reference to $f^{-1}(V)$ and $V$ in the formulas.

We consider the hermitian structures on the sheaves $U$ and $\OO(1)$
on $\PP(\Omega_{X})$
induced by $p^{\ast}\ov{T_X}$. We will write $\ov{\eta}$ for the exact
sequence in
\eqref{eq:87} and
$\ov{\alpha}$, $\ov{\beta}$ and $\ov{\gamma}$ for the vertical
isomorphisms in diagram \eqref{eq:bb_2}, all provided with the
corresponding metrics. Notice that
$\ov{\alpha}$ and $\ov{\beta}$ are isometries. By the properties of
the Bott-Chern  class $\widetilde{\Td}_{m}$, we have
\begin{equation}\label{eq:bb_4}
  \widetilde{\Td}_{m}(\ov{\varepsilon})=\widetilde{\Td}_{m}(\ov{\eta})
  +\Td(\ov{\eta})\widetilde{\Td}_{m}(\ov{\gamma}).
\end{equation}
Hence, from \eqref{eq:bb_3}--\eqref{eq:bb_4} and identifying $f$ with
$\widetilde{f}$ over $V$, we have
\begin{multline}\label{eq:bb_4bis}
  T(\ov{\xi}')=T(\ov{\xi})-
  \widetilde{f}_{\ast}[\pi^{\ast}\ch(\ov{\mathcal{F}})\pi^{\ast}\Td(\ov{f})
  \widetilde{\Td}_{m}(\ov{\eta})]\\
  -\widetilde{f}_{\ast}[\pi^{\ast}\ch(\ov{\mathcal{F}})
  \pi^{\ast}\Td(\ov{f})\Td(\ov{\eta})
  \widetilde{\Td}_{m}(\ov{\gamma})].
\end{multline}
It will be convenient to have a precise description of
$\widetilde{\Td}_{m}(\ov{\gamma})$ at our disposal.

For shorthand, we write $L:=\widetilde{\mu}^{\ast}\OO(1)$ and
$\|\cdot\|_{0}$ for its hermitian structure constructed before. We
denote by $\|\cdot\|_{1}$ the metric on $\OO(E)$ such that the
isomorphism $\ov{\OO(E)}_{1}=\ov{L}_{0}^{-1}\otimes
\widetilde{f}^{\ast}\ov{T_S}$ (Lemma \ref{lemma:bb_1}) is an
isometry. Recall that $\gamma$ gets identified with the restriction to
$W$ of the natural inclusion $L\to L\otimes\OO(E)$. We let
$\|\cdot\|_{\infty}$ be the hermitian metric on $L\mid_{W}$ such that
$\gamma$ is an isometry. Hence, if $\mathbf{1}$ denotes the canonical
section of $\OO(E)$ and $\ell$ is any section of $L\mid_{W}$, then we
have
\begin{displaymath}
	\|\ell\|_{\infty}=\|\ell\|_{0}\|\mathbf{1}\|_{1}.
\end{displaymath}
To simplify the notations, we will skip the reference to $W$. We then have on $W$
\begin{displaymath}
  \widetilde{\Td}_{m}(\ov{\gamma})= \widetilde{\Td}_{m}(\ov{L}_{0}
  \overset{\Id}{\to}\ov{L}_{\infty}).
\end{displaymath}
To compute a representative of this class, we fix a smooth function
$h\colon \PP^{1}_{\CC}\to\RR$ such that $h(0)=0$ and
$h(\infty)=1$. Then we proceed by a deformation argument. Let $q\colon
W\times\PP^{1}_{\CC}\to W$ be the projection to the first factor. On
the line bundle $q^{\ast}L$ we put the metric that, on the fiber at the
point $(w,t)\in W\times\PP^{1}_{\CC}$, is determined by the formula
\begin{displaymath}
  \|\ell\|_{(w,t)}=\|\ell\|_{0,w}\|\mathbf{1}\|_{1,w}^{h(t)}.
\end{displaymath}
We will write $\|\cdot\|_{t}$ for this family of metrics parametrized
by $\PP^{1}_{\CC}$. Define
\begin{displaymath}
  \ov{\Td}(\ov{L}_{0}\to\ov{L}_{\infty})=
  \frac{1}{2\pi i}\int_{\PP^{1}_{\CC}}
  \frac{-1}{2}\log(t\ov{t})(\Td(\ov{q^{\ast}L}_{t})- \Td(\ov{q^{\ast}L}_{0})).
\end{displaymath}
Then
\begin{equation}\label{eq:bb_5bis}
  \ov{\Td}_{m}(\ov{\gamma})=\Td^{-1}(\ov{L}_{0})
  \ov{\Td}(\ov{L}_{0}\to\ov{L}_{\infty})
\end{equation}
represents the class $\widetilde{\Td}_{m}(\gamma)$. Let us develop
$\ov{\Td}_{m}(\ov{\gamma})$. If $\ov{\OO}_{t}$ denotes the trivial
line bundle on $W\times\PP^{1}_{\CC}$ with the norm
$\|\mathbf{1}\|_{t}=\|\mathbf{1}\|_{1}^{h(t)}$, then we compute
\begin{displaymath}
	\Td(\ov{q^{\ast}L}_{t})-\Td(\ov{q^{\ast}L}_{0})
	=\frac{1}{2}c_{1}(\ov{\OO}_{t})+\frac{1}{6}c_{1}(\ov{\OO}_{t})q^{\ast}c_{1}(\ov{L}_{0})
	+\frac{1}{12}c_{1}(\ov{\OO}_{t})^{2}.
\end{displaymath}
By the very definition of $c_{1}$, we find
\begin{multline*}
    c_{1}(\ov{\OO}_{t})=\pd\cpd\log\|\mathbf{1}\|_{t}^{2}=
    \pd\cpd(h(t)\log\|\mathbf{1}\|_{1}^{2})\\
    =h(t)c_{1}(\ov{\OO(E)}_{1})
    +\log\|\mathbf{1}\|_{1}^{2}\pd\cpd h(t)
    +\pd h(t)\wedge\cpd\log\|\mathbf{1}\|_{1}^{2}
    +\pd\log\|\mathbf{1}\|_{1}^{2}\wedge\cpd h(t).
\end{multline*}
We easily obtain
\begin{align}
  &\frac{1}{2\pi i}\int_{\PP^{1}_{\CC}}\frac{-1}{2}\log(t\ov{t})
  \frac{1}{2}c_{1}(\ov{\OO}_{t})
  =-\frac{1}{2}\log\|\mathbf{1}\|_{1},\label{eq:bb_5}\\
  &\frac{1}{2\pi i}\int_{\PP^{1}_{\CC}}
  \frac{-1}{2}\log(t\ov{t}) \frac{1}{6}q^{\ast}c_{1}(\ov{L}_{0})
  c_{1}(\ov{\OO}_{t})=-\frac{1}{6}
  \log\|\mathbf{1}\|_{1}c_{1}(\ov{L}_{0}).\label{eq:bb_6}
\end{align}
With some more work, we have
\begin{equation}\label{eq:bb_7}
  \begin{split}
    \frac{1}{2\pi
      i}\int_{\PP^{1}_{\CC}}\frac{-1}{2}\log(t\ov{t})
    \frac{1}{12}c_{1}(\ov{\OO}_{t})^{2}=&
    -\frac{a}{6}\log\|\mathbf{1}\|_{1}c_{1}(\ov{\OO(E)}_{1})\\
    &
    +\frac{b}{3}\pd(\log\|\mathbf{1}\|_{1}\
    \cpd\log\|\mathbf{1}\|_{1}),
  \end{split}
\end{equation}
where
\begin{equation}\label{eq:bb_9}
  a=\frac{1}{2\pi
    i}\int_{\PP^{1}_{\CC}}\log(t\ov{t})\frac{1}{2}\pd\cpd(h(t)^{2}), \qquad
  b=\frac{1}{2\pi i}\int_{\PP^{1}_{\CC}}\log(t\ov{t})\pd
  h(t)\wedge\cpd h(t).
\end{equation}
We observe that
\begin{displaymath}
  a=\frac{1}{2\pi i}\int_{\PP^{1}_{\CC}}\log(t\ov{t})
  \frac{1}{2}\pd\cpd(h(t)^{2})=\frac{1}{2},
\end{displaymath}
which is independent of $h$. All in all, equations
\eqref{eq:bb_5bis}--\eqref{eq:bb_9} provide the following expression
for the representative $\ov{\Td}_{m}(\ov{\gamma})$ of
$\widetilde{\Td}_{m}(\ov{\gamma})$:
\begin{equation}\label{eq:bb_7bis}
  \begin{split}
    \ov{\Td}_{m}(\ov{\gamma})=\Td^{-1}(\ov{L}_{0})
    \Big(&-\frac{1}{2}\log\|\mathbf{1}\|_{1}
    -\frac{1}{6}\log\|\mathbf{1}\|_{1}c_{1}(\ov{L}_{0})\\
    &-\frac{1}{12}\log\|\mathbf{1}\|_{1} c_{1}(\ov{\OO(E)}_{1})
    +\frac{b}{3}\pd(\log\|\mathbf{1}\|_{1}\cpd\log\|\mathbf{1}\|_{1})\Big).
  \end{split}
\end{equation}

Given a current $\eta\in \mathcal{D}_{D}^{n}(X,p)$, we will call
$(n,p)$ its \emph{Deligne bidegree}, while we will call the
\emph{Dolbeault bidegree} to the bidegree in the Dolbeault
complex. When it is clear from the context to which bidegree we are
referring, we call it bidegree.

We now study the singularities of the component of Deligne bidegree
$(1,1)$
of $T(\ov{\xi}')$ near the divisor
$\Delta$. For this we first recall the decomposition of equation
\eqref{eq:bb_4bis}. Observe that
$\widetilde{\mathcal{D}}_{D}^{1}(V,\emptyset,1)$ gets identified with
the space of smooth real functions on $V$. In the sequel, for an
element $\vartheta\in\oplus_{p}\widetilde{D}_{D}^{2p-1}(\ast,p)$, we
write $\vartheta^{(2r-1,r)}$ to refer to its component of bidegree
$(2r-1,r)$. By construction of the Deligne complex, an element of
Deligne bidegree $(2r-1,r)$ is just a current of Dolbeault bidegree
$(r-1,r-1)$.

The following assertion is well-known. See for instance \cite[Lemma 2.1,
Cor.~2.2]{Wolpert}.
\begin{lemma}\label{lemma_wolpert}
  Let $\Omega\subset\CC$ be an open subset and $\vartheta$ a current
  of Dolbeault bidegree $(0,0)$ on $\Omega$. Let $\Delta$ be the standard
  laplacian. If the current $\Delta \vartheta$ is represented by a
  locally bounded measurable function, then $\vartheta$ is represented
  by a continuous function.
\end{lemma}

\begin{proposition}\label{prop:bb_1}
  The current
  $T(\ov{\xi})^{(1,1)}\in\widetilde{\mathcal{D}}_{D}^{1}(S,N_{f},1)$
  is represented by a continuous function on $S$.
\end{proposition}
\begin{proof}
The differential equation satisfied by $T(\ov{\xi})^{(1,1)}$ is
\begin{equation}\label{eq:bb_11}
	\dd_{\mathcal{D}} T(\ov{\xi})^{(1,1)}=\ch(\ov{\Rd f_{\ast}\mathcal{F}})^{(2,1)}
	-f_{\ast}[\ch(\ov{\mathcal{F}})\Td(\ov{f})]^{(2,1)}.
\end{equation}
In local coordinates, the operator $\dd_{\mathcal{D}}=-2\pd\cpd$ is a
rescaling of the laplacian $\Delta$. By the lemma, it is enough we
prove that the current at the right hand side of \eqref{eq:bb_11} is
represented by a locally bounded measurable differential form. Because
$\ch(\ov{\Rd f_{\ast}\mathcal{F}})^{(2,1)}$ and
$\ch(\ov{\mathcal{F}})\Td(\ov{f})$ are smooth differential forms, we
are reduced to study currents of the form $f_{\ast}[\theta]^{(2,1)}$,
where $\theta$ is a smooth differential form. By a partition of unity
argument, we reduce to the case where $f\colon \CC^{2}\to\CC$ is the
morphism $f(z_{0},z_{1})=z_{0}z_{1}$ and $\theta$ is a differential
form of Dolbeault bidegree (2,2) with compact support. Then we need to
prove that the fiber integral
\begin{displaymath}
	G(w)=\int_{z_{0}z_{1}=w}\theta
\end{displaymath}
is a bounded form in a neighborhood of $w=0$. Write
\begin{math}
	\theta=h(z_{0},z_{1})dz_{0}\wedge d\ov{z}_{0}\wedge dz_{1}\wedge d\ov{z}_{1}.
\end{math}
We reduce to study integrals of the form
\begin{displaymath}
	G(w)=\left(\int_{|w|<|z_{0}|<1}h(z_{0},z_{0}/w)\frac{|w|^{2}}{|z_{0}|^{4}}dz_{0}\wedge\ov{z}_{0}\right)dw\wedge d\ov{w}.
\end{displaymath}
The property follows from an easy computation in polar coordinates.
\end{proof}

\begin{proposition}\label{prop:f_ast_theta}
Let $\theta$ be a differential form of Dolbeault bidegree (1,1) on
$\widetilde{X}$. Then the current $\widetilde{f}_{\ast}[\theta]$ is
represented by a bounded function on $S$.
\end{proposition}
\begin{proof}
The proof is the same as in
\cite[Prop. 5.2]{Bismut-Bost}. One only has to show that the argument
in \emph{loc. cit.} carries over to the case of the non-reduced
fibres that have appeared when blowing up the nodes.
\end{proof}

\begin{corollary}\label{cor:bb_1}
The current
$\widetilde{f}_{\ast}[\pi^{\ast}\ch(\ov{\mathcal{F}})
\pi^{\ast}\Td(\ov{f})\widetilde{\Td}_{m}(\ov{\eta})]$
is represented by a bounded function on $S$.
\end{corollary}
\begin{proof}
It suffices to observe that the differential form
$\pi^{\ast}\ch(\ov{\mathcal{F}})
\pi^{\ast}\Td(\ov{f})\widetilde{\Td}_{m}(\ov{\eta})$
is actually smooth on the whole $\widetilde{X}$.
\end{proof}
According to \eqref{eq:bb_4bis}, it remains to study the current
\begin{displaymath}
  \widetilde{f}_{\ast}[\pi^{\ast}\ch(\ov{\mathcal{F}})
  \pi^{\ast}\Td(\ov{f})\Td(\ov{\eta})
  \widetilde{\Td}_{m}(\ov{\gamma})]\mid_{V}.
\end{displaymath}
The main difference with the situation in Corollary \ref{cor:bb_1} is
that the class $\widetilde{\Td}_{m}(\ov{\gamma})$ is not defined on
the whole $\widetilde{X}$, but only on $W=\widetilde{X}\setminus
E$. In the following discussion we will use the representative
$\ov{\Td}_{m}(\ov{\gamma})$ defined in \eqref{eq:bb_5bis} at the place
of $\widetilde{\Td}_{m}(\ov{\gamma})$. In view of equations
\eqref{eq:bb_5}--\eqref{eq:bb_7}, the first result we need is the following
statement.

\begin{proposition}
  Let $\theta$ be a smooth and $\pd, \cpd$ closed differential form on
  $\widetilde{X}$, of Dolbeault bidegree $(1,1)$. Let $w$ be an
  analytic coordinate in a neighborhood of $p\in\Delta$ with
  $w(p)=0$. Write $D_{p}=E\cap\widetilde{f}^{-1}(p)$. Then, the
  current
  \begin{displaymath}
    \widetilde{f}_{\ast}[\log\|\mathbf{1}\|_{1}\theta]- \left(\frac{1}{2\pi
        i}\int_{D_{p}}\theta\right)[\log|w|]
  \end{displaymath}
  is represented by a continuous function in a neighborhood of $p$. In
  particular, if $\theta$ is cohomologous to a form
  $\pi^{\ast}\vartheta$, where $\vartheta$ is a smooth and $\pd,\cpd$
  closed differential form on $X$, then
  $\widetilde{f}_{\ast}[\log\|\mathbf{1}\|_{1}\theta]$ is represented by
  a continuous function on $S$.
\end{proposition}
\begin{proof}
  Recall that the Poincar\'e-Lelong formula provides the equality of
  currents
  \begin{displaymath}
    \dd_{\mathcal{D}}[\log\|\mathbf{1}\|_{1}^{-1}]=[c_{1}(\ov{\OO(E)}_{1})]-\delta_{E}.
  \end{displaymath}
  Moreover, the operator $\dd_{\mathcal{D}}$ commutes with proper
  push-forward. Therefore, taking into account that $\theta$ is $\pd$
  and $\cpd$ closed, the
  equation
  \begin{equation}\label{eq:bb_corr_1}
    \dd_{\mathcal{D}}\widetilde{f}_{\ast}[\log\|\mathbf{1}\|_{1}\theta]=
    \left(\frac{1}{2\pi i}\int_{D_{p}}\theta\right)\delta_{p}-
    \widetilde{f}_{\ast}[c_{1}(\ov{\OO(E)}_{1})\theta].
  \end{equation}
  holds in a neighborhood of $p$. On the other hand, the Poincar\'e-Lelong equation also gives
  $\dd_{\mathcal{D}}[\log|w|]=\delta_{p}$.
  Using \eqref{eq:bb_corr_1}, we see that
  \begin{displaymath}
    \dd_{\mathcal{D}}\left(\widetilde{f}_{\ast}[\log\|\mathbf{1}\|_{1}\theta]
      -\left( \frac{1}{2\pi i}\int_{D_{p}}\theta\right)[\log|w|]\right)=
    -\frac{1}{2}\widetilde{f}_{\ast}[c_{1}(\ov{\OO(E)}_{1})\theta].
  \end{displaymath}
  Finally, by Proposition \ref{prop:f_ast_theta}, the current
  $\widetilde{f}_{\ast}[c_{1}(\ov{\OO(E)}_{1})\theta]$ is represented by
  a continuous function on $S$. Hence the first assertion follows from
  Lemma \ref{lemma_wolpert}. For the second assertion, we just observe
  that, in this case,
  \begin{displaymath}
    \int_{D_{p}}\theta=\int_{D_{p}}\pi^{\ast}\vartheta=0.
  \end{displaymath}
  The proof is complete.
\end{proof}
\begin{corollary}\label{cor:bb_corr_1}
  Let $n_{p}$ be the multiplicity of $\Delta$ at $p$ and $O(1)$
  the current represented by a locally bounded function. 
  The following estimates hold in a neighborhood of $p$
  \begin{align*}
    &\widetilde{f}_{\ast}[\log\|\mathbf{1}\|_{1}c_{1}(\pi^{\ast}\ov{T_{X}})]=O(1),\\
    &\widetilde{f}_{\ast}[\log\|\mathbf{1}\|_{1}c_{1}(\ov{\OO(E)}_{1})]=-n_{p}[\log|w|]+O(1),\\
    &\widetilde{f}_{\ast}[\log\|\mathbf{1}\|_{1}c_{1}(\ov{L}_{0})]=n_{p}[\log|w|]+O(1),\\
    &\widetilde{f}_{\ast}[\log\|\mathbf{1}\|_{1}
    c_{1}(\widetilde{\mu}^{\ast}\ov{U})]=-n_{p}[\log|w|]+O(1).
  \end{align*}
\end{corollary}
\begin{proof}
  We use 
  \eqref{eq:bb_1}--\eqref{eq:bb_1bis} and the intersection numbers
  $(D_{p}\cdot D_{p})=(D_{p}\cdot E)=-n_{p}$.
\end{proof}
\begin{corollary}\label{cor:bb_2}
  With the notations above, the development
  \begin{displaymath}
    \widetilde{f}_{\ast}[\pi^{\ast}\ch(\ov{\mathcal{F}})
    \pi^{\ast}\Td(\ov{f})\Td(\ov{\eta}) \widetilde{\Td}_{m}(\ov{\gamma})]^{(3,2)}
    =\rk(\ov{\mathcal{F}})\frac{n_{p}}{6}[\log|w|]+O(1)
  \end{displaymath}
  holds in a neighborhood
  of $p$.
\end{corollary}
\begin{proof}
  We take into account the expression \eqref{eq:bb_7bis} for the
  representative $\ov{\Td}_{m}(\ov{\gamma})$, the developments of the
  smooth differential forms $\ch(\ov{\mathcal{F}})$, $\Td(\ov{f})$,
  $\Td(\ov{\eta})$ and $\Td^{-1}(\ov{L}_{0})$, and then apply
  Corollary \ref{cor:bb_corr_1}. We find
  \begin{multline*}
    \widetilde{f}_{\ast}[\pi^{\ast}\ch(\ov{\mathcal{F}})
    \pi^{\ast}\Td(\ov{f})\Td(\ov{\eta}) \widetilde{\Td}_{m}(\ov{\gamma})]^{(3,2)}
    =\\ \rk(\ov{\mathcal{F}})\frac{n_{p}}{6}[\log|w|]+
    \rk(\ov{\mathcal{F}})\frac{b}{3}
    \widetilde{f}_{\ast}[\pd(\log\|\mathbf{1}\|_{1}\cpd\log\|\mathbf{1}\|_{1})]+O(1).
  \end{multline*}
  To conclude we observe that, on $V$, the term
  $\widetilde{f}_{\ast}[\pd(\log\|\mathbf{1}\|_{1}\cpd\log\|\mathbf{1}\|_{1})]$
  vanishes. Indeed, the morphism $\widetilde{f}_{\ast}$ is smooth on
  $V$ with one dimensional fibers. Hence this current is represented
  by the function
  \begin{displaymath}
      V\ni s\mapsto \frac{1}{2\pi i}\int_{\widetilde{f}^{-1}(s)}
      \pd(\log\|\mathbf{1}\|_{1}\cpd\log\|\mathbf{1}\|_{1})=
      \frac{1}{2\pi i}\int_{\widetilde{f}^{-1}(s)}
      d(\log\|\mathbf{1}\|_{1}\cpd\log\|\mathbf{1}\|_{1})=0.
  \end{displaymath}
  This ends the proof.
\end{proof}

The results of this section are summarized in the following statement.
\begin{theorem}
  Let $p\in\Delta$ and let $n_{p}$ be the number of singular points of
  $f\colon X\to S$ lying above $p$. Let $w$ be a local coordinate on
  $S$, centered at $p$. Then, in a neighborhood of $p$, we have the
  estimate
  \begin{displaymath}
    T(\ov{\xi}')^{(1,1)}=-\frac{\rk{\ov{\mathcal{F}}}}{6}n_{p}[\log|w|]+O(1).
  \end{displaymath}
\end{theorem}
\begin{proof}
  It is enough to join \eqref{eq:bb_4bis},
  Proposition \ref{prop:bb_1}, Corollary \ref{cor:bb_1} and
  \ref{cor:bb_2}.
\end{proof}

\begin{corollary}
  Assume that $\ov{\mathcal{F}}=\ov{E}$ is a vector bundle placed in
  degree $0$, and that $R^{1}f_{\ast}E=0$ on $S$. Endow $f_{\ast}E$
  with the $L^2$ metric on $V$ depending on $\ov{E}$ and the metric on
  $\ov{\omega_{f^{-1}(V)/V}}$. Write
  $\xi''=(\ov{f}',\ov{E},\ov{f_{\ast}E}_{L^{2}})$ for the
  corresponding relative metrized complex on $V$. Let $p$ and $w$ be
  as in the theorem. Then we have
  \begin{displaymath}
    T(\ov{\xi}'')^{(1,1)}=-\frac{\rk(\ov{\mathcal{F}})}{6}n_{p}[\log|w|]+O(\log\log|w|^{-1})
  \end{displaymath}
  as $w\to 0$.
\end{corollary}
\begin{proof}
  Introduce an auxiliary smooth hermitian metric on the vector bundle
  $f_{\ast}E$ on $S$, and let
  $\ov{\xi}'=(\ov{f}',\ov{E},\ov{f_{\ast}E})$ be the corresponding
  relative metrized complex. Then the theorem applies to
  $\ov{\xi}'$. By the anomaly formulas, on $V$ we have
  \begin{displaymath}
    T(\ov{\xi}'')^{(1,1)}=T(\ov{\xi}')^{(1,1)}+
    \widetilde{\ch}(\ov{f_{\ast}E},\ov{f_{\ast}E}_{L^2})^{(1,1)}.
  \end{displaymath}
  By \cite[Prop. 7.1]{Bismut-Bost}, the $L^2$ metric has logarithmic
  singularities near $w=0$ and
  \begin{displaymath}
    \widetilde{\ch}(\ov{f_{\ast}E},\ov{f_{\ast}E}_{L^2})=O(\log\log|w|^{-1})
  \end{displaymath}
  as $w\to 0$. This proves the corollary.
\end{proof}
\begin{remark}
  The corollary is to be compared with
  \cite[Thm. 9.3]{Bismut-Bost}. The difference of sign is due to the
  fact that Bismut and Bost work with the inverse of the usual
  determinant line bundle. The approach of \emph{loc. cit.} is more
  analytic and requires the spectral description of the Ray-Singer
  analytic torsion.
\end{remark}

\bigskip
\footnotesize
\noindent\textit{Acknowledgments.}
During the elaboration of this paper we have benefited from
conversations with many colleagues, that helped us to understand some
points, to clarify others or to find the relevant bibliography. Our
thanks to J.-M. Bismut, J.-B. Bost, D. Burghelea, D. Eriksson, J. Kramer,
U. K\"uhn, X. Ma, V. Maillot, D. R\"ossler, C. Soul\'e.

We would like to thank the following institutions where part of the
research conducting to this paper was done: the CRM in Bellaterra
(Spain), the CIRM in Luminy (France), the Morningside Institute of Beijing
(China), the University of Barcelona and the IMUB, the Alexandru Ioan
Cuza University of Iasi,
the Institut de Math\'ematiques de Jussieu and the ICMAT (Madrid).

Burgos and Freixas were partially supported by grant MTM2009-14163-
C02-01, Burgos was partially supported by CSIC research project
2009501001, Li\c tcanu was partially supported by CNCSIS -UEFISCSU,
project number PNII - IDEI 2228/2008 .


\newcommand{\noopsort}[1]{} \newcommand{\printfirst}[2]{#1}
  \newcommand{\singleletter}[1]{#1} \newcommand{\switchargs}[2]{#2#1}
\providecommand{\bysame}{\leavevmode\hbox to3em{\hrulefill}\thinspace}
\providecommand{\MR}{\relax\ifhmode\unskip\space\fi MR }
\providecommand{\MRhref}[2]{%
  \href{http://www.ams.org/mathscinet-getitem?mr=#1}{#2}
}
\providecommand{\href}[2]{#2}

\end{document}